\documentclass[11pt]{amsart}
\usepackage{amsmath}
\usepackage{amssymb}
\usepackage{amsthm}
\usepackage{graphicx}
\usepackage{mathrsfs}
\usepackage{dsfont}
\usepackage{pifont}
\usepackage{ytableau}
\usepackage{youngtab}
\usepackage{tikz}
\usepackage{bookmark}
\usepackage{upgreek}
\hypersetup{hidelinks}

\hypersetup{colorlinks,linkcolor=blue,urlcolor=cyan,citecolor=blue}

\newtheorem{thm}{Theorem}[section]
\theoremstyle{plain}

\newtheorem{lem}[thm]{Lemma}
\newtheorem{prop}[thm]{Proposition}
\newtheorem{cor}[thm]{Corollary}

\theoremstyle{definition}
\newtheorem{defn}[thm]{Definition}
\newtheorem{example}[thm]{Example}

\theoremstyle{remark}
\newtheorem{rem}[thm]{Remark}

\definecolor{A}{rgb}{.75,1,.75}

\newcommand{\la}{\lambda}
\newcommand{\Ga}{\Ga}

\newcommand{\op}{\operatorname}

\renewcommand{\b}{\mathtt{b}}

\newcommand{\un}{\underline}
\newcommand{\aHn}{\mathcal{H}^{\mathrm{aff}}_n(q)}
\newcommand{\Hn}{\mathcal{H}_n(q)}
\newcommand{\Hnn}{\mathcal{H}_{n+1}(q)}
\newcommand{\Hr}{\mathcal{H}_r(q)}
\newcommand{\Hrr}{\mathcal{H}_{r+1}(q)}
\newcommand{\I}{\mathbb{I}}
\newcommand{\RP}{\mathcal{RP}}
\newcommand{\DRP}{\mathcal{DRP}}
\newcommand{\CSP}{\mathcal{CSP}}

\numberwithin{equation}{section}

\begin{document}

\title[Hecke-Clifford algebras]{Representations of Hecke-Clifford superalgebras at roots of unity}

\author{Minjia Chen}\address{School of Mathematics and Statistics\\
  Beijing Institute of Technology\\
  Beijing, 100081, P.R. China}
  \email{scarleturanus@163.com}

\author{Jinkui Wan}
\address{Jinkui Wan, School of Mathematical Sciences, Shenzhen University, Shenzhen, 518060, P.R. China}
\email{wjk302@hotmail.com}
  


\begin{abstract}
{\color{black} In this article, we give a classification of irreducible completely splittable representations of affine Hecke-Clifford superalgebras $\mathcal{H}_n^{\mathrm{aff}}(q)$ when $q^2$ is a primitive $h$-th root of unity. As an application, we derive a necessary and sufficient condition for the finite Hecke-Clifford superalgebra $\mathcal{H}_n(q)$ to be semisimple. Specially we show that $\mathcal{H}_n(q)$ is semisimple if and only $h >n$ in the case $h$ is odd and    $h >2n$ in the case $h$ is even.}
\end{abstract}
\maketitle
\setcounter{tocdepth}{1}
 \tableofcontents


\section{Introduction}

\subsection{} An explicit construction, including a dimension formula, for the irreducible representations of symmetric groups $\mathfrak{S}_n$ in module case over an algebraically closed field $\mathbb{F}$ of characteristic $p>0$ is still an open important problem.  
A landmark result in this direction was
obtained by Mathieu~\cite{M}, who computed the dimensions of the
irreducible modules associated to partitions
$\lambda = (\lambda_1, \ldots, \lambda_l)$ of $n$ with $l=\ell(\la)$ satisfying
$\lambda_1 - \lambda_l+l \leq p$, by means of the classical Schur-Weyl
duality.  Kleshchev~\cite{K1} later characterized these modules representation theoretically by showing  they
coincide precisely with those whose restriction to every subgroup
$\mathfrak{S}_k$ for all $1\leq k < n$ is semisimple, or equivalently, on which the Jucys-Murphy elements of $\mathbb{F}\mathfrak{S}_n$ act
semisimply.  Following~\cite{K1}, such modules are called \emph{completely
splittable} (also called calibrated or homogeneous in literature).  Using the modular branching rules for $\mathfrak{S}_n$ (cf. \cite{K2}), one then obtains a dimension formula in terms of paths in Young modular graphs, which recovers Mathieu's result.

The theory of completely splittable representations has been extended to various algebras beyond $\mathfrak{S}_n$. 
Generalizing previous work~\cite{K1, M}, Ruff~\cite{Ru} classified the irreducible completely splittable modules for degenerate affine Hecke algebras. 
Earlier constructions and classifications over the complex field were given by Cherednik~\cite{C1,C2} and Ram~\cite{Ra} for related algebras. 
Analogous theories have since been developed for degenerate affine Hecke-Clifford superalgebras~\cite{Wa,HKS}, for Khovanov-Lauda-Rouquier algebras~\cite{KR}, and for quiver Hecke superalgebras~\cite{KL}.

\subsection{}  In order to study $q$-Young symmetrizers arising in the projective representation theory of
$\mathfrak{S}_n$ initialed by \cite{Sch},   Jones and Nazarov~\cite{JN} introduced the notion of non-degenerate affine Hecke-Clifford superalgebras $\aHn$. The study of $\aHn$ and its associated cyclotomic quotient algebras has made substantial progress recently (cf. 
\cite{BK1, KKT, KL, KMS, LS, Mo, N1, N2, Ol, Ts, SW, Sh}). 
The present paper is aimed to solve two interrelated problems concerning representations of $\aHn$ over $\mathbb{F}$.  The first is to
classify and explicitly construct all irreducible completely splittable
$\aHn$-modules when the parameter $q$ is a root of unity.  The second is to derive
necessary and sufficient conditions for the finite Hecke-Clifford
superalgebra $\Hn$ to be semisimple. Our construction of irreducible completely splittable
$\aHn$-modules is inspired by Young's seminormal construction for symmetric
groups and affine Hecke algebras of type $A$, and the underlying philosophy is close to the approach of Okounkov and
Vershik~\cite{OV} for symmetric groups over $\mathbb{C}$.

\subsection{}

Denote by $X_1^{\pm 1}, \ldots, X_n^{\pm 1}$ the invertible Laurent
polynomial generators of $\aHn$; see subsection~\ref{subsec:aHn} for the
precise definition and relations.  Following \cite{BK1}
(and also~\cite[Part~II]{K2}), the classification of finite-dimensional
$\aHn$-modules reduces to that of \emph{integral} modules, that is, those
on which each operator $X_j + X_j^{-1}$ ($1 \leq j \leq n$) has eigenvalues
lying in the set $\{\mathtt{q}(i) \mid i \in \I\}$ (see~\eqref{substitution0} and \eqref{defn:I} for notations).  For such modules, a standard argument yields
the weight space decomposition
\[
  M \;=\; \bigoplus_{\,\underline{i}\,\in\,\I^n} M_{\underline{i}},
\]
where $M_{\underline{i}}$ is the simultaneous generalized eigenspace for the
commuting operators $X_1 + X_1^{-1}, \ldots, X_n + X_n^{-1}$ corresponding
to eigenvalues $\mathtt{q}(i_1), \ldots, \mathtt{q}(i_n)$, respectively.
A tuple $\underline{i} \in \I^n$ is called a \emph{weight} of $M$ if
$M_{\underline{i}} \neq 0$.  A finite-dimensional $\aHn$-module is defined to be 
completely splittable if the operators $X_1, X_2, \ldots,
X_n$ act semisimply, so that the weight space decomposition
is a direct sum of eigenspaces (as opposed to generalized eigenspaces).

We first establish a set of equivalent characterizations for irreducible completely splittable $\aHn$-modules, stated precisely in
Proposition~\ref{prop:equiv.cond.}. One equivalence is that an irreducible $\aHn$-module $M$ is completely splittable if and only if its weight space $M_{\un i}$ is isomorphic to the irreducible module $L(\un i)$ over the subalgebra $\mathcal{A}_n$ of $\aHn$ which is generated by $X_1^{\pm 1},X_2^{\pm 2},\ldots,X_n^{\pm 1}, C_1,C_2,\ldots,C_n$.  Another key equivalence is that an irreducible $\aHn$-module is completely splittable if and only if its
restriction to the subalgebra $\mathcal{H}^{\mathrm{aff}}_{(r,1^{n-r})}(q)$ associated to the Young subgroup $\mathfrak{S}_{(r,1^{n-r})}$
is semisimple for every $1 \leq r \leq n$. In particular, any  irreducible completely splittable $\aHn$-module is semisimple when restricted to 
the subalgebra generated by $T_k,\, C_k,\, C_{k+1},\, X_k^{\pm 1},\,
X_{k+1}^{\pm 1}$ for each fixed $1 \leq k \leq n - 1$, which is isomorphic
to $\mathcal{H}^{\mathrm{aff}}_2(q)$, see Corollary \ref{Cor-1}. 
A thorough study of the irreducible $\mathcal{H}^{\mathrm{aff}}_2(q)$-modules
then yields an explicit formula for the action of the element
$T_1,T_2,\ldots, T_{n-1}$ on irreducible completely splittable $\aHn$-modules and a complete
determination of all their possible weights which results in the explicit actions of $C_1,\ldots, C_n, X_1^{\pm 1},\ldots,X_n^{\pm 1}$.  From this, we construct an
explicit family of irreducible completely splittable $\aHn$-modules and prove
that this exhausts non-isomorphic irreducible completely splittable $\aHn$-modules. 
\subsection{}
Let $\Hn$ be the finite Hecke-Clifford superalgebra.  A $\Hn$-module is
called completely splittable if the Jucys-Murphy elements $L_1, \ldots, L_n$
(see~(\ref{eq:JM})) act semisimply on it.  According to~\cite{N2}, there is a
surjective algebra homomorphism $\aHn \twoheadrightarrow \Hn$ sending $X_k$
to the Jucys-Murphy element $L_k$ for each $1 \leq k \leq n$.  Via  this homomorphism,
we classify all irreducible completely splittable $\Hn$-modules and derive
an explicit dimension formula for them. Recall that the irreducible representations of $\Hn$ over $\mathbb{F}$ are
parameterized by the set $\mathcal{RP}_h(n)$ of $h$-restricted $h$-strict partitions
of $n$ when $h$ is odd~\cite{BK1}, and by the set $\DRP_h(n)$ of doubly
restricted $\tfrac{h}{2}$-strict partitions of $n$ when $h$ is
even~\cite{Ts, Mo}.  One of our main results identifies the subset
$\CSP_h(n) \subseteq \RP_h(n)$ (resp. $\CSP_h(n) \subseteq
\DRP_h(n)$)  in the case $h$ is odd (resp.  $h$ is even) consisting of exactly those partitions that label irreducible
completely splittable $\Hn$-modules.


As an application of the classification, we show that when $n = h$ with $h$
odd, or when $n = 2h$ with $h$ even, every irreducible representation of
$\Hn$ is completely splittable.  For these critical values of $n$, this
provides an explicit construction of all irreducible $\Hn$-modules together
with closed dimension formulas.  A comparison of total dimensions then shows
that $\Hn$ is \emph{not} semisimple at these values.  To extend
non-semisimplicity to larger $n$, we establish that $\mathcal{H}_{n+1}(q)$
fails to be semisimple whenever $\Hn$ does; the proof
exploits the identification of the branching graph for $\Hn$ with the crystal
graph of type $A^{(2)}_{h-1}$ (when $h$ is odd) and $D^{(2)}_{h/2}$ (when
$h$ is even).  Taken together, these arguments establish that the semisimplicity
of $\Hn$ is equivalent to the known sufficient condition given in \cite{SW}, thereby recovering
and strengthening theorem of~\cite{Sh} for $\Hn$  by removing an additional
hypothesis that was assumed there.

\subsection{}

The paper is organized as follows.  Section~\ref{sec:aHn} includes
basics on superalgebras and the affine
Hecke-Clifford superalgebra $\aHn$.  In Section~\ref{sec:weights}, we examine the weight structure of
completely splittable $\aHn$-modules, developing the technical details needed
for the classification.  The classification of irreducible completely
splittable $\aHn$-modules is then carried out in Section~\ref{sec:classifications}.
The passage to the finite Hecke-Clifford 
superalgebra $\Hn$ and the combinatorial classification of its irreducible completely splittable
representations are treated in Section~\ref{sec:cs-Hn}.  In 
Section~\ref{sec:semsimplicity}, we proved a semisimplicity criteria for the finite Hecke-Clifford superalgebras $\Hn$.


Throughout this paper, $\mathbb{F}$ is an algebraically closed field with $\op{char}\mathbb{F}\neq 2$, and $q \in \mathbb{F}^{*}$ satisfies
$q \neq \pm 1$.  We assume that $q$ is a primitive $k$-th root of unity and $q^2$ is a primitive $h$-th root of unity.
As observed in~\cite[Introduction]{Ts}, the case $k = 2(2\ell + 1)$ for
$\ell \geq 1$ reduces to the case $k$ being odd; accordingly, it suffices to
treat the two cases $k = 2\ell + 1$ (odd) and $k =2(2\ell)$ (even) separately.
Thus, $h =2\ell+1$ when $k=2\ell+1$ in which case $h$ is also odd;  and $h =\frac{k}{2}=2\ell$ when $k=2(2\ell)$ in which chase $h$ is even.  If no such integer exists, we put $h = \infty$.  We further set
$\varepsilon:= q - q^{-1}$.


\section{Basics on the non-degenerate affine Hecke-Clifford superalgebra $\aHn$}\label{sec:aHn}
\subsection{Some basics about superalgebras} 
 We shall recall some basic notions of superalgebras, referring the
reader to~\cite[\S 2-b]{BK1}. {\color{black} By a superspace over $\mathbb{F}$, we mean a $\mathbb{Z}_2$-graded vector space.} Let us denote by
$\bar{v}\in\mathbb{Z}_2$ the parity of a homogeneous vector $v$ of a
vector superspace. By a superalgebra, we mean a
$\mathbb{Z}_2$-graded associative algebra. Let $\mathcal{A}$ be a
superalgebra. {\color{black}By an $\mathcal{A}$-module, we mean a $\mathbb{Z}_2$-graded
left $\mathcal{A}$-module.} A homomorphism $f:V\rightarrow W$ of
$\mathcal{A}$-modules $V$ and $W$ means a linear map such that $
f(av)=(-1)^{\bar{f}\bar{a}}af(v).$  It is common in superalgebra to write expressions that are initially defined only for homogeneous elements. Their intended meaning for arbitrary elements is then obtained by extending the definition linearly from the homogeneous case.  Let $V$ be a finite dimensional
$\mathcal{A}$-module. Let $\Pi
 V$ be the same underlying vector space but with the opposite
 $\mathbb{Z}_2$-grading. The new action of $a\in\mathcal{A}$ on $v\in\Pi
 V$ is defined in terms of the old action by $a\cdot
 v:=(-1)^{\bar{a}}av$. Note that the identity map on $V$ defines
 an isomorphism from $V$ to $\Pi V$. By forgeting the grading we may consider any superalgebra $\mathcal{A}$ as the usual algebra which will be denoted by $|\mathcal{A}|$. Similarly, any $\mathcal{A}$-supermodule $V$ can be considered as a usual $|\mathcal{A}|$-module denoted by $|V|$. A superalgebra analog of Schur's Lemma (cf. \cite{K2}) states that the endomorphism
 algebra $\text{End}_{\mathcal{A}}(V)$ of a finite dimensional irreducible $\mathcal{A}$-module $V$  is either one dimensional or two dimensional. In the
 former case, we call the module $V$ of {\em type }\texttt{M} while in
 the latter case the module $V$ is called of {\em type }\texttt{Q}.  A superalgebra $\mathcal{A}$ is said to be semisimple if the usual algebra $|\mathcal{A}|$ is semisimple. 


\begin{lem}{\color{blue}\cite[Lemma 12.2.1, Corollary 12.2.10]{K2}}\label{lem:type MQ}
	{\color{black}Suppose $\mathcal{A}$ is a finite dimensional superalgebra and $V$ is an irreducible $\mathcal{A}$-module}. If $V$ is of type $\texttt{M}$, then by forgetting the grading, $|V|$ is an irreducible $|\mathcal{A}|$-module. If $V$ is of type $\texttt{Q}$, then by forgetting the grading, $|V|$  is isomorphic to a direct sum of two non-isomorphic irreducible $|\mathcal{A}|$-modules. That is, there exist two non-isomorphic irreducible $|\mathcal{A}|$-modules $V^+,V^-$ such that $|V|\cong V^+\oplus V^-$ as $|\mathcal{A}|$-modules. Moreover if $V_1,\cdots,V_m$ (resp. $V_{m+1},\cdots,V_n$) are pairwise non-isomorphic irreducible $\mathcal{A}$-modules of type \texttt{M} (resp. \texttt{Q}), then $$\{|V_1|, \cdots,|V_m|, V_{m+1}^\pm, \cdots, V_{n}^\pm\}$$ is a complete set of pairwise non-isomorphic $|\mathcal{A}|$-modules.  If in addition $\mathcal{A}$ is semisimple, then 
\begin{equation}\label{eq:dimequal}
\op{dim} \mathcal{A}=\sum_{i=1}^m(\op{dim} V_i)^2+\sum_{j=m+1}^n\frac{(\op{dim} V_j)^2}{2}. 
\end{equation}

\end{lem} 

Given two superalgebras $\mathcal{A}$ and $\mathcal{B}$, we view
the tensor product of superspaces $\mathcal{A}\otimes\mathcal{B}$
as a superalgebra with multiplication defined by
\begin{equation}\label{eq:tensor-AB}
(a\otimes b)(a'\otimes b')=(-1)^{\bar{b}\bar{a'}}(aa')\otimes (bb')
\qquad (a,a'\in\mathcal{A}, b,b'\in\mathcal{B}).
\end{equation}
Suppose $V$ is an $\mathcal{A}$-module and $W$ is a
$\mathcal{B}$-module. Then $V\otimes W$ affords $\mathcal{A}\otimes \mathcal{B}$-module
denoted by $V\boxtimes W$ via
\begin{equation}\label{eq:tensor-AB-1}
(a\otimes b)(v\otimes w)=(-1)^{\bar{b}\bar{v}}av\otimes bw,~a\in \mathcal{A},
b\in \mathcal{B}, v\in V, w\in W.
\end{equation}
If $V$ is an irreducible $\mathcal{A}$-module and $W$ is an
irreducible $\mathcal{B}$-module, $V\boxtimes W$ may not be
irreducible. Indeed, we have the following standard lemma {\color{blue} (cf.
\cite[Lemma 12.2.13]{K2}).}
\begin{lem}\label{tensorsmod}
Let $V$ be an irreducible $\mathcal{A}$-module and $W$ be an
irreducible $\mathcal{B}$-module.
\begin{enumerate}
\item If both $V$ and $W$ are of type $\texttt{M}$, then
$V\boxtimes W$ is an irreducible
$\mathcal{A}\otimes\mathcal{B}$-module of type $\texttt{M}$.

\item If one of $V$ or $W$ is of type $\texttt{M}$ and the other one
is of type $\texttt{Q}$, then $V\boxtimes W$ is an irreducible
$\mathcal{A}\otimes\mathcal{B}$-module of type $\texttt{Q}$.

\item If both $V$ and $W$ are of type $\texttt{Q}$, then
$V\boxtimes W\cong X\oplus \Pi X$ for a type $\texttt{M}$
irreducible $\mathcal{A}\otimes\mathcal{B}$-module $X$.
\end{enumerate}
Moreover, all irreducible $\mathcal{A}\otimes\mathcal{B}$-modules
arise as constituents of $V\boxtimes W$ for some choice of
irreducible modules $V,W$.
\end{lem}
If $V$ is an irreducible $\mathcal{A}$-module and $W$ is an
irreducible $\mathcal{B}$-module, denote by $V\circledast W$ an
irreducible component of $V\boxtimes W$. Thus,
$$
V\boxtimes W=\left\{
\begin{array}{ll}
V\circledast W\oplus \Pi (V\circledast W), & \text{ if both } V \text{ and } W
 \text{ are of type }\texttt{Q}, \\
V\circledast W, &\text{ otherwise }.
\end{array}
\right.
$$
\subsection{The non-degenerate affine Hecke-Clifford superalgebra $\aHn$}\label{subsec:aHn} Recall $\varepsilon=q-q^{-1}$ with $q\neq \pm 1\in\mathbb{F}^*$. 
 Let $\aHn$ be the (non-degenerate) affine Hecke-Clifford superalgebra over $\mathbb{F}$ generated by even elements $T_1,\cdots,T_{n-1}$, $X^{\pm 1}_1,\cdots,X_n^{\pm 1}$ and odd elements $C_1,\cdots, C_n$ subject to the relations: 
\begin{align}
T_i^2&=\varepsilon T_i+1\quad T_iT_j=T_jT_i,\quad T_iT_{i+1}T_i=T_{i+1}T_iT_{i+1},\quad |i-j|>1,\label{TT}\\
X_iX_j&=X_jX_i,\quad X_iX_i^{-1}=X_i^{-1}X_i=1,\quad 1\leq i,j\leq n\label{eq:Poly}\\
C_i^2&=1,\quad C_iC_j=-C_jC_i,\quad 1\leq i\neq j\leq n,\label{eq:Clifford}\\
T_iX_i&=X_{i+1}T_i-\varepsilon(X_{i+1}+C_iC_{i+1}X_i),\label{TX1}\\
T_iX_j&=X_jT_i,\quad j\neq i,i+1,\label{TX2}\\
T_iC_i&=C_{i+1}T_i,\quad T_iC_{i+1}=C_iT_i-\varepsilon(C_i-C_{i+1}),\quad T_iC_j=C_jT_i,\quad j\neq i,i+1,\label{TC}\\
X_iC_i&=C_iX_i^{-1},\quad X_iC_j=C_jX_i,\quad 1\leq i\neq j\leq n.\label{XC}
\end{align}
Obviously (\ref{TX1}) is equivalent to one of the following equivalent equations:
\begin{align}
T_iX_i^{-1}&=X_{i+1}^{-1}T_i+\varepsilon(X_{i}^{-1}+X_{i+1}^{-1}C_iC_{i+1}),\notag\\
T_iX_{i+1}&=X_{i}T_i+\varepsilon(1-C_iC_{i+1})X_{i+1},\label{TX3}\\
T_iX_{i+1}^{-1}&=X_{i}^{-1}T_i-\varepsilon X_{i}^{-1}(1-C_iC_{i+1})\notag
\end{align}
For each permutation $w\in\mathfrak{S}_n$ with an reduced expression $w=s_{i_1}s_{i_2}\cdots s_{i_r}$ for some $1\leq i_1,\ldots,i_{r}\leq n-1$ with $r\geq 0$, there exists an element $T_w:=T_{i_1}\cdots T_{i_r}$ and it is independent of the choice of the reduced expression of $w$ due to the Braid relation \eqref{TT}. For $\alpha=(\alpha_1,\ldots,\alpha_n)\in\mathbb{Z}^n$ and
$\beta=(\beta_1,\ldots,\beta_n)\in\mathbb{Z}_2^n$, set
$X^{\alpha}=X_1^{\alpha_1}\cdots X_n^{\alpha_n}$ and
$C^{\beta}=C_1^{\beta_1}\cdots C_n^{\beta_n}$. Then we have the
following.
\begin{lem}\cite[Theorem 2.3]{BK1}\label{lem:PBWNon-dege}
The set $\{X^{\alpha}C^{\beta}T_w~|~ \alpha\in\mathbb{Z}^n,
\beta\in\mathbb{Z}_2^n, w\in {\mathfrak{S}_n}\}$ forms a basis of $\aHn$.
\end{lem}

For any $a\in \mathbb{F}$, we fix a solution of the equation $x^2=a$ and denote it by $\sqrt{a}$. Following \cite{BK1}, for $i\in \mathds{Z}$ we define
\begin{equation}\label{substitution0}
\mathtt{q}(i)=2\frac{q^{2i+1}+q^{-2i-1}}{q+q^{-1}},\quad\mathtt{b}_\pm(i)=\frac{\mathtt{q}(i)}{2}\pm\sqrt{\frac{\mathtt{q}(i)^2}{4}-1}.
\end{equation}
In addition, following \cite{BK1} and \cite{Ts}, we take the subset $\mathbb{I}\subset \mathbb{Z}$ via 
\begin{equation}\label{defn:I} 
\mathbb{I}=\begin{cases}\mathbb{Z}_{\geq 0}, &\text{ if }h=\infty,\\
\{0,1,\cdots,\frac{h-1}{2}\}, &\text{ if }h\text{ is odd,}\\
\{0,1,\cdots,\frac{h}{2}-1\}, &\text{ if }h\text{ is even}.
\end{cases}
\end{equation}
It is easy to verify $\{\mathtt{q}(i)|i\in\mathbb{Z}\}=\{\mathtt{q}(i)\mid i\in\mathbb{I}\}$ and moreover $\mathtt{q}(i)\neq \mathtt{q}(j)$ if $i\neq j\in\mathbb{I}$. 
In addition we have 
\begin{equation}\label{eq:qi=2}
\begin{aligned}
&\mathtt{q}(i)=\pm 2 \text{ if and only if }i=0 \text{ in the case }h\text{ is odd},\\
&\mathtt{q}(i)=\pm 2 \text{ if and only if }i=0\text{ or }\frac{h}{2}-1 \text{ in the case }h\text{ is even}. 
\end{aligned}
\end{equation}

Let $\mathcal{A}_n$ be the subalgebra of $\aHn$  generated by even generators $X^{\pm}_1,\ldots,X^{\pm 1}_n$ and odd generators $C_1,\ldots,C_n$. By Lemma~\ref{lem:PBWNon-dege}, $\mathcal{A}_n$ actually can be identified with the superalgebra generated by even generators $X^{\pm 1}_1,\ldots,X^{\pm 1}_n$ and odd generators $C_1,\ldots,C_n$ subject to relations \eqref{eq:Poly}, \eqref{eq:Clifford}, \eqref{XC}.  For a
composition $\mu=(\mu_1,\mu_2,\ldots,\mu_r)$ of $n$, we define
$\mathcal{H}^{\mathrm{aff}}_{\mu}(q)$ to be the subalgebra of $\aHn$
generated by $\mathcal{A}_n$ and $T_j$ such that $s_j\in S_{\mu}=S_{\mu_1}\times\cdots
\times S_{\mu_r}$. Note that
$\mathcal{A}_n=\mathcal{H}^{\mathrm{aff}}_{(1^n)}(q)$.
 A $\mathcal{A}_n$-module $M$ is called {\em integral} if it is finite
dimensional and moreover all eigenvalues of $X_1 + X_1^{-1},X_2+X_2^{-1},\ldots, X_n + X_n^{-1}$ on $M$ are
of the form $\mathtt{q}(i)$ with  $i\in \mathbb{I}$. Call an $\aHn$-module, or more generally an
$\mathcal{H}_\mu^{\mathrm{aff}}(q)$-module for $\mu$ a composition of $n$, integral if it is integral on restriction to $\mathcal{A}_n$.
In what follows we will restrict our attention to these modules, and write $\op{Rep}_{\mathbb{I}} \aHn$ the category of integral $\aHn$-modules. 
For each $i\in \mathbb{I}$, let $L(i)$ be the $2$-dimensional
$\mathcal{A}_1$-module $L(i)=\mathbb{F}v_0\oplus \mathbb{F}v_1$  with 
\begin{equation}\label{eq:Li}
X^{\pm 1}_1 v_0=\mathtt{b}_{\pm}(i)v_0,\quad X^{\pm 1}_1 v_1=\mathtt{b}_{\mp}(i)v_1, \quad
C_1v_0=v_1,\quad C_1v_1=v_0.
\end{equation}
Clearly $L(i)\cong L(j)$ if and only if $i=j\in\mathbb{I}$. 
\begin{lem} 
The $\mathcal{A}_1$-module $L(i)$  is irreducible of type $\texttt{M}$ if $\mathtt{q}(i)\neq \pm 2 $,
and irreducible of type $\texttt{Q}$ if $\mathtt{q}(i)= \pm 2 $. Moreover, $\{L(i)\mid i\in\mathbb{I}\}$ is a complete set of pairwise non-isomorphic integral finite dimensional irreducible $\mathcal{A}_1$-module. 

\end{lem}
\begin{proof}
One can easily check that $L(i)$ is of type $\texttt{M}$ if and only if $\mathtt{b}_+(i)=\mathtt{b}_-(i)$ or equivalently $\mathtt{q}(i)=\pm 2$ by \eqref{substitution0}. Hence the first statement holds. Suppose $U$ is an integral finite dimensional irreducible $\mathcal{A}_1$-module. Then there exists an eigenvector $u_0\in U_{\bar{0}}$ of $X_1$ with respect to an eigenvalue $z$ for some $z\in\mathbb{F}^*$. Then $(X_1+X_1^{-1})u_0=(z+z^{-1})u_0$. Since $U$ is integral, we obtain $z+z^{-1}=\mathtt{q}(i)$ for some $i\in\mathbb{I}$. Then by \eqref{substitution0} one can show  $U\cong L(i)$.  Putting together, we obtain that $\{{L}(i)|i\in\mathbb{I}\}$ is a complete set of pairwise non-isomorphic finite dimensional irreducible $\mathcal{A}_1$-module. 
\end{proof}

Clearly we have $$\mathcal{A}_n\cong \mathcal{A}_1\otimes\cdots\otimes \mathcal{A}_1.$$
{\color{black} For each $\underline{i}=(i_1,i_2,\ldots,i_n)\in\I^n$, set 
\begin{equation}\label{L-under-a-nondege}
L(\underline{i})=L(i_1)\circledast L(i_2)\circledast\cdots\circledast L(i_n), 
\end{equation}
 then $L(\underline{i})\cong L(\underline{j})$ if and only if $i_k= j_k$ for $1\leq k\leq n$. } 
 
 \begin{cor}\cite[Corollary 3.4]{SW} \label{cor:irrepAn}
The $\mathcal{A}_n$-modules
$$
\{L(\underline{i})=L(i_1)\circledast
L(i_2)\circledast\cdots\circledast
L(i_n)|~\underline{i}=(i_1,\ldots,i_n)\in\mathbb{I}^n\}
$$
forms a complete set of pairwise non-isomorphic finite dimensional irreducible modules in the category $\op{Rep}_{\I}\mathcal{A}_n$. 
Moreover, setting $\gamma_0=\{1\leq k\leq n\mid \mathtt{q}(i_k)=\pm 2\}$,  then $L(\underline{i})$ is of type $\texttt{M}$ if
$\gamma_0$ is even and type $\texttt{Q}$ if $\gamma_0$ is odd.
Furthermore,
$\text{dim}~L(\underline{i})=2^{n-\lfloor\frac{\gamma_0}{2}\rfloor}$,
where $\lfloor\frac{\gamma_0}{2}\rfloor$ denotes the greatest
integer less than or equal to $\frac{\gamma_0}{2}$ .
\end{cor}


\begin{rem}\label{rem:Ltau}
Following \cite[Remark 2.5]{Wa}, we observe that each permutation $\tau\in {\mathfrak{S}_n}$ defines a superalgebra
isomorphism $\tau:\mathcal{A}_n \rightarrow \mathcal{A}_n$ by mapping $X^{\pm 1} _k$ to
$X^{\pm 1}_{\tau(k)}$ and $C_k$ to  $C_{\tau(k)}$, for $1\leq k\leq n$. For
$\underline{i}\in\mathbb{I}^n$, the twist of the action of
$\mathcal{A}_n$ on $L(\underline{i})$ with
$\tau^{-1}$ leads to a new $\mathcal{A}_n$-module denoted by
$L(\underline{i})^{\tau}$ with
$$
L(\underline{i})^{\tau}=\{z^{\tau}~|~z\in L(\underline{i})\} ,\quad
fz^{\tau}=(\tau^{-1}(f)z)^{\tau}, \text{ for any }f\in
\mathcal{A}_n, z\in L(\underline{i}).
$$
So in particular we have 
\begin{equation}\label{X-z-tau}
(X^{\pm 1} _kz)^{\tau}=X^{\pm 1}_{\tau(k)}z^{\tau}, 
(C_kz)^{\tau}=C_{\tau(k)}z^{\tau}
\end{equation}
 for each $1\leq k\leq n$ and $z\in L(\un i), \tau\in\mathfrak{S}_n$. It is easy to see that $L(\underline{i})^{\tau}\cong L(\tau\cdot \underline{i})$, where
$\tau\cdot \underline{i}=(i_{\tau^{-1}(1)},\ldots,i_{\tau^{-1}(n)})$
for $\underline{i}=(i_1,\ldots,i_n)\in\mathbb{I}^n$ and $\tau\in {\mathfrak{S}_n}$. Moreover it is straightforward to show that the following holds 
\begin{equation}\label{L-tau-sigma}
((L(\underline{i}))^\tau)^\sigma\cong L(\underline{i})^{\sigma\tau}. 
\end{equation}

\end{rem}
\subsection{Intertwining elements for $\aHn$}
Given $1\leq k<n$, we define the intertwining element $\widetilde{\Phi}_k$ in $\aHn$ as follows: 
\begin{equation} \label{eq:zi}
\mathsf{z}_k:= (X_k+X^{-1}_k)-(X_{k+1}+X^{-1}_{k+1})= X^{-1}_k(X_kX_{k+1}-1)(X_kX^{-1}_{k+1}-1),
\end{equation}
\begin{equation}\label{intertwinNon-dege}
\widetilde{\Phi}_k:=\mathsf{z}^2_k T_k+\varepsilon\frac{\mathsf{z}^2_k}{X_k X^{-1}_{k+1}-1}-\varepsilon\frac{\mathsf{z}^2_k}{X_k X_{k+1}-1}C_k C_{k+1}.
\end{equation} These elements satisfy the following properties (cf. \cite[(3.7),Proposition 3.1]{JN} and \cite[(4.11)-(4.15)]{BK1}) 
\begin{align}
\widetilde{\Phi}^2_k&=\mathsf{z}^2_k\left(\mathsf{z}^2_k-\varepsilon^2 \left(X^{-1}_k X^{-1}_{k+1}(X_k X_{k+1}-1)^2-X^{-1}_kX_{k+1}(X_k X^{-1}_{k+1}-1)^2\right)\right)\label{Sqinter},\\
\widetilde{\Phi}_k X^{\pm 1}_k&=X^{\pm 1}_{k+1}\widetilde{\Phi}_k, \widetilde{\Phi}_kX^{\pm 1}_{k+1}=X^{\pm 1}_k\widetilde{\Phi}_k,
\widetilde{\Phi}_k X^{\pm 1}_l=X^{\pm 1}_l\widetilde{\Phi}_k \label{Xinter},\\
\widetilde{\Phi}_k C_k&=C_{k+1}\widetilde{\Phi}_k, \widetilde{\Phi}_k C_{k+1}=C_k \widetilde{\Phi}_k,
\widetilde{\Phi}_kC_l=C_l\widetilde{\Phi}_k \label{Cinter},\\
\widetilde{\Phi}_j \widetilde{\Phi}_k&=\widetilde{\Phi}_k \widetilde{\Phi}_j,
\widetilde{\Phi}_k\widetilde{\Phi}_{k+1}\widetilde{\Phi}_k=\widetilde{\Phi}_{k+1}\widetilde{\Phi}_k  \widetilde{\Phi}_{k+1}\label{Braidinter}
\end{align}
for all admissible $j,k,l$ with $l\neq k, k+1$ and $|j-k|>1$. Observe that we can rewrite $\widetilde{\Phi}^2_k$ as 
\begin{equation}\label{eq:Phi-square}
\widetilde{\Phi}^2_k=\mathsf{z}^4_k\varepsilon^2\left(\frac{1}{\varepsilon^2}-\frac{X_kX_{k+1}}{(X_kX_{k+1}-1)^2}-\frac{X_kX_{k+1}^{-1}}{(X_kX_{k+1}^{-1}-1)^2}\right). 
\end{equation}
Inspired by the above formula, 
{\color{black}for any pair of $(x,y)\in (\mathbb{F}^*)^2$ with $x\neq y^{\pm 1}$, we consider the following condition 
\begin{align}\label{invertible}
	\frac{xy}{(xy-1)^2}+\frac{xy^{-1}}{(xy^{-1}-1)^2}=\frac{1}{\varepsilon^2}.
\end{align}
According to \cite{JN}, via the substitution 
$
x+x^{-1}=2\frac{qu+q^{-1}u^{-1}}{q+q^{-1}}, y+y^{-1}=2\frac{qv+q^{-1}v^{-1}}{q+q^{-1}}
$
with $u,v\in\mathbb{F}$,  
the condition \eqref{invertible}  is  equivalent to the condition which states that $u,v$ satisfy one of the following four equations 
\begin{align}\label{invertible2}
	v=q^2u,\quad v=q^{-2}u,\quad v=u^{-1},\quad v=q^{-4}u^{-1}, 
\end{align}
In particular,  by \eqref{invertible2} one can verify that 
\begin{equation}\label{eq:special-square}
\begin{aligned}
&\frac{xy^{-1}}{(xy^{-1}-1)^2}+\frac{xy}{(xy-1)^2}=\frac{1}{\varepsilon^2}  \text{ with } x=\mathtt{b}_{\pm}(i), y=\mathtt{b}_{\pm}(j) \text{ for some }i,j\in\I\\
&\Longleftrightarrow 
|i-j|=1. 
\end{aligned}
\end{equation}
}

\section{Weights of irreducible completely splittable $\aHn$-modules}\label{sec:weights}
In this section, we shall introduce the notion of completely splittable $\aHn$-modules and then  give a description of the weights of irreducible completely splittable $\aHn$-modules which is parallel to \cite[section 3]{Wa} with more complicated reasoning and calculation.

\subsection{Equivalent conditions for irreducible completely splittable $\aHn$-modules}
For $\underline{i}=(i_1,\cdots,i_n)\in\I^n$ and $M\in\op{Rep}_{\mathbb{I}}\aHn$,  set
\begin{equation}\label{eq:wt-decomp-0}
M_{\underline{i}}=\{z\in M|(X_k+X_k^{-1}-\mathtt{q}(i_k))^Nz=0\text{ for }N\gg 0,1\leq k\leq n\}
\end{equation}
If $M_{\underline{i}}\neq 0,$ then $\underline{i}$ is called a \emph{weight} of $M$ and $M_{\underline{i}}$ is called a weight space. Since the polynomial generators $X_1^{\pm 1},\cdots,X_n^{\pm 1}$ commute, we have
\begin{equation}\label{eq:wt-decomp}
M=\oplus_{\underline{i}\in\mathbb{I}^n}M_{\underline{i}}.
\end{equation}
For $i\in\mathbb{I}$ and $1\leq m\leq n,$ set
\[\Theta_{i^m}M=\{z\in M|(X_j+X_j^{-1}-\mathtt{q}(i))^Nz=0\text{ for }N\gg 0,n-m+1\leq j\leq n\}.\]
It's easy to see that
\begin{equation}
\begin{aligned}\label{1}
&(X_k+X_k^{-1})T_k\\
&=T_k(X_{k+1}+X_{k+1}^{-1})+\varepsilon\left[X_k^{-1}(1-C_kC_{k+1})-(1-C_kC_{k+1})X_{k+1}\right]\\
&=T_k(X_{k+1}+X_{k+1}^{-1})+\varepsilon\left[X_k^{-1}(1-X_kX_{k+1})-X_k^{-1}(1-X_kX_{k+1}^{-1})C_kC_{k+1}\right]
\end{aligned}
\end{equation}
Hence $\Theta_{i^m}$ defines an exact functor
\[\Theta_{i^m}:\op{Rep}_{\mathbb{I}}\aHn\rightarrow\op{Rep}_{\mathbb{I}}\mathcal{H}^{\mathsf{aff}}_{n-m,m}(q).\]
Moreover as $\mathcal{H}^{\mathsf{aff}}_{n-1,1}(q)$-modules, we have
\begin{equation}\label{decomp.2}
\op{res}_{\mathcal{H}^{\mathsf{aff}}_{n-1,1}(q)}^{\aHn}M=\oplus_{i\in\mathbb{I}}\Theta_{i}M.
\end{equation}
For $i\in\mathbb{I}$ and $M\in\op{Rep}_{\mathbb{I}}\aHn,$ define
\[\epsilon_i(M)=\op{max}\{m\geq 0|\Theta_{i^m}M\neq 0\}.\]

\begin{lem}\cite[Theorem 4.16, Lemma 5.4]{BK1}\label{lem:K1-1}
For each $i\in\I$, the induced module $L(\un i^m):=\op{ind}^{\mathcal{H}^{\mathsf{aff}}_m(q)}_{\mathcal{A}_m}L(i)\circledast L(i)\circledast\cdots\circledast L(i)$ is irreducible. In addition, suppose  $M\in\op{Rep}_\mathbb{I}\aHn$ is irreducible. Let $i\in \mathbb{I}$ and $m = \epsilon_i(M).$ Then there exists an irreducible $N \in\op{Rep}_\mathbb{I} \mathcal{H}^{\mathsf{aff}}_{n-m}(q)$ with $\epsilon_i(N) = 0$ such that $\Theta_{i^m}M\cong N \circledast L(\un i^m)$. 
\end{lem}

\begin{defn}
A $\aHn$-module is called completely splittable if $X_1^{\pm 1},\cdots,X_n^{\pm 1}$ act semisimply.
\end{defn}

\begin{rem} \label{rem:cs}
Observe that if $M\in \op{Rep}_\mathbb{I}\aHn$ is completely splittable, then for $\underline{i}\in\mathbb{I}^n,$
\begin{equation}\label{eq:Mi}
\begin{aligned}
M_{\underline{i}}&=\{z\in M|(X_k+X_k^{-1})z=\mathtt{q}(i_k)z,1\leq k\leq n\}\\
&=\text{span-}\{z\in M|X_kz=\mathtt{b}_\pm(i_k)z,1\leq k\leq n\}. 
\end{aligned}
\end{equation}
\end{rem}

\begin{lem}\label{lem-0}
For any $i\in \mathbb{I}$ such that $i\neq 0,\frac{h}{2}-1$ and any $0\neq z\in L(i)\boxtimes \cdots\boxtimes L(i)$($n$ copies), we have $X_j^{-1}(1-C_jC_{j+1})z\neq (1-C_jC_{j+1})X_{j}z$ for any $1\leq j\leq n-1$.
\end{lem}
\begin{proof}
Assume $1\leq j\leq n-1$. Observe that the elements $X_j^{-1}(1-C_jC_{j+1})$ and $(1-C_jC_{j+1})X_{j+1}$ are even and act on the $j$-th, $(j+1)$-th factor of the tensor product $L(i)\boxtimes \cdots\boxtimes L(i)$. So it suffices to consider the case $n = 2$ and $j = 1$. Let $\{v,w\}$ be a basis of $L(i)$ such that $v\in L(\un i)_{\bar{0}}, w\in L(\un i)_{\bar{1}}$ and $X_1^\pm v=\mathtt{b}_\pm(i)v,X_1^\pm w=\mathtt{b}_\mp(i)w, C_1v=w, C_1w=v$. Then for any
\[z=av\otimes v+bv\otimes w+cw\otimes v+dw\otimes w\in L(i)\boxtimes L(i),\]
with $a,b,c,d\in \mathbb{F}$, by \eqref{eq:tensor-AB-1} we have
\[\begin{aligned}
X_1^{-1}(1-C_1C_{2})z=&(\b_-(i)a+\b_-(i)d)v\otimes v+(\b_-(i)b+\b_-(i)c)v\otimes w\\
&+(\b_+(i)c-\b_+(i)b)w\otimes v+(\b_+(i)d-\b_+(i)a)w\otimes w,\\
(1-C_1C_{2})X_{2}z=&(\b_+(i)a+\b_-(i)d)v\otimes v+(\b_-(i)b+\b_+(i)c)v\otimes w\\
&+(\b_+(i)c-\b_-(i)b)w\otimes v+(\b_-(i)d-\b_+(i)a)w\otimes w.
\end{aligned}\]
Since $\b_+(i)\neq\b_-(i)$ for $i\neq 0,\frac{h}{2}-1$, it is easy to verify that $X_1^{-1}(1-C_1C_2)z=(1-C_1C_2)X_2z$ iff $a=b=c=d=0$ or equivalently $z=0$. This proves the lemma. 
\end{proof}
\begin{lem}\label{lem:separate weig.}
Suppose that $M\in\op{Rep}_\mathbb{I}\aHn$ is completely splittable and that $M_{\underline{i}}\neq 0$ for some $\underline{i}\in\mathbb{I}^n.$ Then $i_k\neq i_{k+1}$ for all $1\leq k\leq n-1.$
\end{lem}
\begin{proof}
Suppose $i_k=i_{k+1}$ for some $1\leq k\leq n-1.$ Let $0\neq z \in M_{\underline{i}}.$ Since $M$ is completely splittable, $(X_k+X_k^{-1}-\mathtt{q}(i_k))z=0=(X_{k+1}+X_{k+1}^{-1}-\mathtt{q}(i_{k+1}))z$ by Remark  \ref{rem:cs}.  Then by (\ref{1}),
\[(X_k+X_k^{-1}-\mathtt{q}(i_k))T_kz=\varepsilon\left[X_k^{-1}(1-C_kC_{k+1})-(1-C_kC_{k+1})X_{k+1}\right]z.\]
and
\[(X_k+X_k^{-1}-\mathtt{q}(i_k))^2T_kz=\varepsilon\left[X_k^{-1}(1-C_kC_{k+1})-(1-C_kC_{k+1})X_{k+1}\right](X_k+X_k^{-1}-\mathtt{q}(i_k))z=0.\]
Similarly, we see that
$
(X_{k+1}+X_{k+1}^{-1}-\mathtt{q}(i_{k+1})^2)T_kz=0.
$
Hence 
\begin{equation}\label{eq:Tkz}
T_kz\in M_{\underline{i}}.
\end{equation} Therefore by \eqref{eq:Mi} we have 
\[\varepsilon\left[X_k^{-1}(1-C_kC_{k+1})-(1-C_kC_{k+1})X_{k+1}\right]z=0.\]
This together with Lemma \ref{lem-0}  implies  $i_k=i_{k+1}=0$ in the case $h$ is odd and $i_k=i_{k+1}=0 \text{ or }i_k=i_{k+1}=\frac{h}{2}-1$ in the case $h$ is even. Then by \eqref{eq:Mi} and \eqref{substitution0} 
as well as \eqref{eq:qi=2} we obtain that $X_k=X_{k+1}=1 \text { or }-1 $ acting on $M_{\underline{i}}$, and  then by (\ref{TX3}) and \eqref{eq:Tkz} one can get 
\[\varepsilon(1-C_kC_{k+1})X_{k+1}z=T_kX_{k+1}z-X_kT_kz=0.\]
 This means $2\varepsilon z=\varepsilon(1+C_kC_{k+1})(1-C_kC_{k+1})z=0.$ Hence $z=0$ since $q\neq \pm1.$ This contradicts to the assumption $z\neq 0.$ Thus the lemma follows.
\end{proof}
\begin{cor}\label{cor:variM}
If $M\in\op{Rep}_{\mathbb{I}}\aHn$ is completely splittable, then $\epsilon_i(M)\leq 1$ for any $i\in\mathbb{I}.$
\end{cor}
\begin{prop}\label{prop:equiv.cond.}
Let $M\in \op{Rep}_{\mathbb{I}}\aHn$ be irreducible. The following are equivalent.
\begin{enumerate}
\item $M$ is completely splittable.
\item For any $\underline{i}\in\mathbb{I}^n$ with $M_{\underline{i}}\neq 0,$ we have $i_k\neq i_{k+1}$ for any $1\leq k\leq n-1.$
\item The restriction $\op{res}_{\mathcal{H}^{\mathsf{aff}}_{(r,1^{n-r})}(q)}^{\aHn}M$ is semisimple for any $1\leq r\leq n.$
\item For any $\underline{i}\in\mathbb{I}^n$ with $M_{\underline{i}}\neq 0,$ we have $M_{\underline{i}}\cong L(\un i)$ as $\mathcal{A}_n$-modules.
\end{enumerate}
\end{prop}
\begin{proof}
The proof here is similar to that of  \cite[Proposition 3.6]{Wa}. By Lemma~\ref{lem:separate weig.}, $(1)$
implies $(2)$. Suppose $(2)$ holds, then by Lemma~\ref{lem:K1-1} and
Corollary~\ref{cor:variM} we have $\Theta_iM$ is either zero or
irreducible for each $i\in\I$ and hence by~(\ref{decomp.2}) ${\rm
res}^{\aHn}_{\mathcal{H}^{\mathrm{aff}}_{(n-1,1)}(q)}M$ is
semisimple. Observe that if $\Theta_iM\cong N\circledast L(i)$ for
some irreducible
$N\in\operatorname{Rep}_{\I}\mathcal{H}^{\mathrm{aff}}_{n-1}(q)$, then
$(2)$ also holds for $N$. This implies ${\rm
res}^{\mathcal{H}^{\mathrm{aff}}_{n-1}(q)}_{\mathcal{H}^{\mathrm{aff}}_{(n-2,1)}(q)}N$
is semisimple. Therefore ${\rm
res}^{\aHn}_{\mathcal{H}^{\mathrm{aff}}_{(n-2,1,1)}(q)}M$
is semisimple by~(\ref{decomp.2}). Continuing in this way we see
that the restriction ${\rm
res}^{\aHn}_{\mathcal{H}^{\mathrm{aff}}_{(r,1^{n-r})}(q)}M$ is
semisimple for any $1\leq r\leq n$, whence $(3)$.

 Now assume $(3)$ holds.  In particular
${\rm res}^{\aHn}_{\mathcal{H}^{\mathrm{aff}}_{(1^{n})}(q)}M$ is
semisimple, that is, $M$ is  isomorphic to a direct sum of
$L(\underline{i})$ as $\mathcal{A}_n$-modules. It is clear that
$X_1^\pm,\ldots,X_n^\pm$ act semisimply on $L(\underline{i})$ for each
$\underline{i}\in\I^n$, whence $(1)$.

Clearly $(1)$ holds if $(4)$ is true. Now suppose $(1)$ holds and we
shall prove $(4)$ by induction on $n$. Suppose
$M_{\underline{i}}\neq 0$ for some $\underline{i}\in\I^n$. Observe
that $M_{\underline{i}}\subseteq\Theta_{i_n}M\neq 0$. By
Lemma~\ref{lem:K1-1} and Corollary~\ref{cor:variM},
$\Theta_{i_n}M\cong N\circledast L(i_n)$ for some irreducible
$N\in\operatorname{Rep}_{\I}\mathcal{H}^{\mathrm{aff}}_{n-1}(q)$. This
means $M_{\underline{i}}\cong N_{\underline{i}^{\prime}}\circledast
L(i_n)$, where $\underline{i}^{\prime}=(i_1,\ldots,i_{n-1})$. Note
that $N$ is completely splittable and hence by induction
$N_{\underline{i}^{'}}\cong L(i_1)\circledast\cdots\circledast
L(i_{n-1})$. Therefore $M_{\underline{i}}\cong
L(i_1)\circledast\cdots\circledast L(i_n)=L(\un i)$. Putting together, the proposition is proved. 
\end{proof}

\begin{rem} \label{rem:twist}
Note that $\aHn$ possesses an automorphism $\sigma_n$
which sends $T_k$ to $-T_{n-k}+\varepsilon$, $X_j$ to $X_{n+1-j} $ and $C_j$ to
$C_{n+1-j}$ for $1\leq k\leq n-1$ and $1\leq j\leq n$. Moreover
$\sigma_n$ induces an algebra isomorphism for each composition
$\mu=(\mu_1,\ldots,\mu_m)$ of $n$
$$\sigma_{\mu}:
\mathcal{H}^{\mathrm{aff}}_{\mu}(q)\longrightarrow
\mathcal{H}^{\mathrm{aff}}_{\mu^t}(q),
$$
where $\mu^t=(\mu_m,\ldots,\mu_1)$. Given
$M\in\mathcal{H}^{\mathrm{aff}}_{\mu^t}(q)$, we can twist with
$\sigma_{\mu}$ to get a $\mathcal{H}^{\mathrm{aff}}_{\mu}(q)$-module
$M^{\sigma_{\mu}}$. Observe that for $\aHn$-module $M$, we have
$$
\big({\rm
res}^{\aHn}_{\mathcal{H}^{\mathrm{aff}}_{(r,1^{n-r})}(q)}M^{\sigma_n}\big)^{\sigma_{(1^{n-r},r)}}
\cong{\rm res}^{\aHn}_{\mathcal{H}^{\mathrm{aff}}_{(1^{n-r},r)}(q)}M.
$$
Hence $M\in\op{Rep}_\mathbb{I}\aHn$ is irreducible completely splittable  if and
only if ${\rm
res}^{\aHn}_{\mathcal{H}^{\mathrm{aff}}_{(1^{n-r},r)}(q)}M$ is
semisimple for any $1\leq r\leq n$ by
Proposition~\ref{prop:equiv.cond.}.
\end{rem}

\begin{cor}\label{Cor-1}
Let $M\in\op{Rep}_{\mathbb{I}}\aHn$ be irreducible completely splittable. Then the restriction $\op{res}_{\mathcal{H}^{\mathsf{aff}}_{(1^{k-1},2,1^{n-k-1})}}^{\aHn}M$ is semisimple for any $1\leq k\leq n-1.$ Hence $M$ is semisimple on restriction to the subalgebra generated by $T_k,X_k^\pm,X_{k+1}^\pm,C_k,C_{k+1}$ which is isomorphic to $\mathcal{H}_2^{\mathrm{aff}}(q)$ for fixed $1\leq k\leq n-1.$
\end{cor}
\begin{proof}
By Proposition~\ref{prop:equiv.cond.},
${\rm res}^{\aHn}_{\mathcal{H}^{\mathrm{aff}}_{(k+1,1^{n-k-1})}(q)}M$
is semisimple. Hence
$${\rm res}^{\aHn}_{\mathcal{H}^{\mathrm{aff}}_{(1^{k-1},2,1^{n-k-1})}(q)}M={\rm res}^{\mathcal{H}^{\mathrm{aff}}_{(k+1,1^{n-k-1})}(q)}_{
\mathcal{H}^{\mathrm{aff}}_{(1^{k-1},2,1^{n-k-1})}(q)} \big({\rm
res}^{\aHn}_{\mathcal{H}^{\mathrm{aff}}_{(k+1,1^{n-k-1})}(q)}M\big)$$
is semisimple by Remark~\ref{rem:twist}. Meanwhile $\mathcal{H}^{\mathrm{aff}}_{(1^{k-1},2,1^{n-k-1})}\cong\mathcal{A}_{k-1}\otimes \mathcal{H}^{\mathrm{aff}}_2(q)\otimes\mathcal{A}_{n-k-1}$. Then the corollary follows. 
\end{proof}

\subsection{The weight constraints}
Suppose that $M \in\op{Rep}_{\mathbb{I}} \aHn$ is completely splittable and that $M_{\underline{i}}\neq 0$ for some $\underline{i}\in\mathbb{I}^n.$ Then $i_k\neq i_{k+1}$ for $1\leq k\leq n-1.$ It follows that both $X_kX_{k+1}^{-1}-1$ and $X_kX_{k+1}-1$ act as nonzero scalar on $M_{\underline{i}}$ since $X_k=\mathtt{b}_{\pm }(i_k), X_{k+1}=\mathtt{b}_{\pm }(i_{k+1})$ on  $M_{\underline{i}}$ for each $1 \leq k \leq n-1$ by Remark \ref{rem:cs}.  So we define linear operators $\Xi_k$ and $\Omega_k$ on $M_{\underline{i}}$ such that for any $z\in M_{\underline{i}},$
\begin{align}
\Xi_kz:=&-\varepsilon(\frac{1}{X_kX_{k+1}^{-1}-1}-\frac{1}{X_kX_{k+1}-1}C_kC_{k+1})z,\label{Xi}\\
\Omega_kz:=&\left(\sqrt{1-\varepsilon^2\left(\frac{X_kX_{k+1}^{-1}}{(X_kX_{k+1}^{-1}-1)^2}+\frac{X_kX_{k+1}}{(X_kX_{k+1}-1)^2}\right)}\right)z\notag\\
=&\left(\sqrt{1-\varepsilon^2\left(\frac{(X_k+X_k^{-1})(X_{k+1}+X_{k+1}^{-1})-4}{((X_k+X_{k}^{-1})-(X_{k+1}+X_{k+1}^{-1}))^2}\right)}\right)z\label{Omega-0}\\
=&\left(\sqrt{1-\varepsilon^2\left(\frac{\mathtt{q}(i_k)\mathtt{q}(i_{k+1})-4}{(\mathtt{q}(i_k)-\mathtt{q}(i_{k+1}))^2}\right)}\right)z.\label{Omega}
\end{align}
Both $\Xi_k$ and $\Omega_k$ are well-defined linear operators on $L(\underline{i})$ for $\underline{i} \in\mathbb{I}^n$ whenever $i_k\neq i_{k+1}$ for $1\leq k\leq n$.


\begin{prop}\label{prop:restr-rank-2}
The following holds for $i,j\in\mathbb{I}.$
\begin{enumerate}
\item If $i=j\pm 1,$ then the irreducible $\mathcal{A}_2$-module $L(i)\circledast L(j)$ affords an irreducible $\mathcal{H}^{\mathsf{aff}}_2(q)$-module denoted by $V(i,j)$ with the action $T_1z=\Xi_1z$ for any $z\in L(i)\circledast L(j).$ The $\mathcal{H}_2^{\mathrm{aff}}(q)$-module $V(i,j)$ has the same type as the $\mathcal{A}_2$-module $L(i)\circledast L(j).$ Moreover, it is always completely splittable.
\item If $i\neq j\pm1,$ the $\mathcal{H}_2^{\mathrm{aff}}(q)$-module $V(i,j):=\op{ind}_{\mathcal{A}_2}^{\mathcal{H}^{\mathsf{aff}}_2(q)}L(i)\circledast L(j)$ is irreducible and has the same type as the $\mathcal{A}_2$-module $L(i)\circledast L(j)$. It is completely splittable if and only if $i\neq j$ (and recall $i\neq j\pm1).$
\item Every irreducible module in the category $\op{Rep}_{\mathbb{I}}\mathcal{H}^{\mathsf{aff}}_2(q)$ is isomorphic to some $V(i,j).$
\end{enumerate}
\end{prop}
\begin{proof}
(1) It is routine to verify $T_1X_1z=\left(X_2T_1-\varepsilon(X_{2}+C_1C_2X_1)\right)z$ and $T_1C_2z=\left(C_1T_1-\varepsilon(C_1-C_2)\right)z$ with $T_1z=\Xi_1z$ for any $z\in V(i,j):=L(i)\circledast L(j)$. Hence it remains to prove $T_1^2=\varepsilon T_1+1$ on $V(i,j).$ Indeed, for $z\in V(i,j)$ we have by \eqref{Xi}
\begin{equation}
\begin{aligned}\label{SqXi}
\Xi_1^2z&=\varepsilon^2\left[\frac{1}{(X_1X_2^{-1}-1)^2}+\frac{X_1X_2}{(X_1X_2-1)^2}-\frac{1-X_1X_2^{-1}}{(X_1X_2^{-1}-1)(X_1X_2-1)}C_1C_2\right]z\\
&=\varepsilon\Xi_1z+\varepsilon^2\left(\frac{X_1X_2^{-1}}{(X_1X_2^{-1}-1)^2}+\frac{X_1X_2}{(X_1X_2-1)^2}\right)z. 
\end{aligned}
\end{equation}
Meanwhile since $i = j \pm 1$,we obtain that 
$$
\left(\frac{X_1X_2^{-1}}{(X_1X_2^{-1}-1)^2}+\frac{X_1X_2}{(X_1X_2-1)^2}\right)z=\left(\frac{xy^{-1}}{(xy^{-1}-1)^2}+\frac{xy}{(xy-1)^2}\right)z=\frac{1}{\varepsilon^2} z
$$
with $x=\mathtt{b}_{\pm}(i)$ and $y=\mathtt{b}_{\pm}(j)$ by \eqref{eq:special-square}. This together with \eqref{SqXi} gives rise to $T_1^2z=(\varepsilon T_1+1)z$ for any $z \in V(i,j).$ 
This proves that $V(i,j)$ admits an $\mathcal{H}^{\mathrm{aff}}_2(q)$-module.  In addition,  ${\rm
End}_{\mathcal{A}_2}(L(i)\circledast L(j))\cong{\rm
End}_{\mathcal{H}_2^{\mathrm{aff}}(q)}(V(i,j))$. Hence $V(i,j)$ has
the same type as the $\mathcal{A}_2$-module
$L(i)\circledast L(j)$. Since $X_1^{\pm 1},X_2^{\pm 1}$ act semisimply on
$L(i)\circledast L(j)$, $V(i,j)$ is completely splittable.

(2)  We first show that $V(i,j) =\op{ind}^{\mathcal{H}_2^{\mathrm{aff}}(q)}_{\mathcal{A}_2}L(i)\circledast L(j)$ is irreducible under the assumption $i\neq j \pm1$.  By Lemma \ref{lem:K1-1}, the module $V(i,j)$ is irreducible if $i=j$. Thus it remains to show when $|i-j|>1$ the module $V(i,j)$ is irreducible. Now suppose that $M$ is a nonzero submodule of $V (i,j)$.  Observe that $V (i,j) = 1\otimes (L(i)\circledast L(j))\oplus T_1 \otimes (L(i)\circledast L(j))$ as vector spaces. Then $M$ contains a nonzero vector $0\neq v:=1\otimes v_1 +T_1\otimes v_2$ for some $0\neq v_1, v_2\in L(i)\circledast L(j)$. By \eqref{eq:Li} and Corollary \ref{cor:irrepAn} we have 
\begin{equation}\label{eq:X1X2-act}
\begin{aligned}
(X_1+X_1^{-1})(1\otimes v_k)=\mathtt{q}(i)(1\otimes v_k),  (X_2+X_2^{-1})(1\otimes v_k)=\mathtt{q}(j)(1\otimes v_k) 
\end{aligned}
\end{equation}
for $k=1,2$. 
This together with (\ref{1}) and \eqref{eq:zi},
\[\begin{aligned}
&(\mathtt{q}(i)-\mathtt{q}(j))(X_1+X_1^{-1}-\mathtt{q}(i)).v\\
&=-\left(\mathtt{q}(i)-\mathtt{q}(j)\right)^2T_1\otimes v_2+(\mathtt{q}(i)-\mathtt{q}(j))\varepsilon\left[X_1^{-1}(1-X_1X_2)-X_1^{-1}(1-X_1X_2^{-1})C_1C_{2}\right](1\otimes v_2)\\
&=-T_1\mathsf{z}^2_1(1\otimes v_2)+\varepsilon\left(\frac{-\mathsf{z}_1^2}{X_1X_2^{-1}-1}+\frac{\mathsf{z}_1^2}{X_1X_2-1}C_1C_2\right)(1\otimes v_2)
=-\widetilde{\Phi}_1(1\otimes v_2)\in M
\end{aligned}\]
This implies $\widetilde{\Phi}^2_1(1\otimes v_2)\in M$. Then by \eqref{eq:X1X2-act} and \eqref{eq:Phi-square} we obtain 
\begin{equation}\label{eq:Phi-square-act}
\left(\frac{xy}{(xy-1)^2}+\frac{xy^{-1}}{(xy^{-1}-1)^2}-\frac{1}{\varepsilon^2}\right)(1\otimes v_2)\in M
\end{equation}
with $x=\mathtt{b}_{\pm}(i), y=\mathtt{b}_{\pm}(j)$ since $i\neq j$ which implies $\mathtt{q}(i)\neq \mathtt{q}(j)$. However since $j\neq i\pm 1$, by \eqref{eq:special-square} we obtain 
$\left(\frac{xy}{(xy-1)^2}+\frac{xy^{-1}}{(xy^{-1}-1)^2}-\frac{1}{\varepsilon^2}\right)\neq 0$ with $x=\mathtt{b}_{\pm}(i), y=\mathtt{b}_{\pm}(j)$. Thus by \eqref{eq:Phi-square-act} we obtain $1\otimes v_2\in M$. This means $L(i)\circledast L(j)\subseteq M$ and hence $M=V(i,j)$. This proves that $V(i,j)$ is irreducible. 

Note that if $i\neq j$, then $V(i,j)$ has two weights, that is,
$(i,j)$ and $(j,i)$. By Proposition~\ref{prop:equiv.cond.}, we see
that ${\rm
res}_{\mathcal{A}_2}^{\mathcal{H}_2^{\mathrm{aff}}(q)}V(i,j)$
is semisimple and is isomorphic to the direct sum of
$L(i)\circledast L(j)$ and $L(j)\circledast L(i)$. This means
$${\rm Hom}_{\mathcal{A}_2}(L(i)\circledast
L(j),{\rm
res}^{\mathcal{H}_2^{\mathrm{aff}}(q)}_{\mathcal{A}_2}V(i,j))
\cong {\rm End}_{\mathcal{A}_2}(L(i)\circledast
L(j)).$$ By Frobenius reciprocity we obtain
$$
{\rm End}_{\mathcal{H}_2^{\mathrm{aff}}(q)}(V(i,j))\cong{\rm
Hom}_{\mathcal{A}_2}(L(i)\circledast L(j),{\rm
res}^{\mathcal{H}_2^{\mathrm{aff}}(q)}_{\mathcal{A}_2}V(i,j))
\cong {\rm End}_{\mathcal{A}_2}(L(i)\circledast
L(j)).$$ Hence $V(i,j)$ has the same type as the
$\mathcal{A}_2$-module $L(i)\circledast L(j)$.

Now suppose $i=j$. This implies that $(i,i)$ is a weight of $V(i,i)$
and hence $V(i,i)$ is not completely splittable by
Lemma~\ref{lem:separate weig.}. By
Proposition~\ref{prop:equiv.cond.}, ${\rm
res}_{\mathcal{A}_2}^{\mathcal{H}_2^{\mathrm{aff}}(q)}V(i,i)$
is not semisimple. Note that ${\rm
res}_{\mathcal{A}_2}^{\mathcal{H}_2^{\mathrm{aff}}(q)}V(i,i)$
has two composition factors and both of them are isomorphic to
$L(i)\circledast L(i)$. Therefore the socle of ${\rm
res}_{\mathcal{A}_2}^{\mathcal{H}_2^{\mathrm{aff}}(q)}V(i,i)$
is simple and isomorphic to $L(i)\circledast L(i)$. Hence ${\rm
Hom}_{\mathcal{A}_2}(L(i)\circledast L(i),{\rm
res}^{\mathcal{H}_2^{\mathrm{aff}}(q)}_{\mathcal{A}_2}V(i,i))
\cong {\rm End}_{\mathcal{A}_2}(L(i)\circledast
L(i))$. By Frobenius reciprocity we obtain
$$
{\rm
End}_{\mathcal{H}_2^{\mathrm{aff}}(q)}(V(i,i))\cong\text{Hom}_{\mathcal{A}_2}(L(i)\circledast
L(i),{\rm
res}^{\mathcal{H}_2^{\mathrm{aff}}(q)}_{\mathcal{A}_2}V(i,i))
\cong {\rm End}_{\mathcal{A}_2}(L(i)\circledast
L(i)).$$ Hence $V(i,i)$ has the same type as the
$\mathcal{A}_2$-module $L(i)\circledast L(i)$.

(3) Suppose $M\in\op{Rep}_\mathbb{I}\mathcal{H}^{\mathrm{aff}}_2(q)$ is irreducible, then there exist $i,j\in\mathbb{I}$ such that $L(i)\circledast L(j)\subset \op{res}_{\mathcal{A}_2}^{\mathcal{H}_2^{\mathrm{aff}}(q)}M$. By Frobenius reciprocity,  $M$ is an irreducible quotient of the induced module $\op{ind}^{\mathcal{H}_2^{\mathrm{aff}}(q)}_{\mathcal{A}_2}L(i)\circledast L(j).$ If $i\neq j\pm 1,$ then $M\cong \op{ind}_{\mathcal{A}_2}^{\mathcal{H}_2^{\mathrm{aff}}(q)}L(i)\circledast L(j)$  by (2); otherwise we have $\Xi_1^2=\varepsilon\Xi_1+1$ on $L(i)\circledast L(j)$ by the proof of (1). Then one can show that the vector space
\[L:=\text{span}\{T_1\otimes v-1\otimes \Xi_1v|v\in L(i)\circledast L(j)\}\]
is a $\mathcal{H}_2^{\mathrm{aff}}(q)$-submodule of $\op{ind}_{\mathcal{A}_2}^{\mathcal{H}_2^{\mathrm{aff}}(q)}L(i)\circledast L(j)$ and it is isomorphic to $V(j,i).$ It is easy to check the quotient $\op{ind}_{\mathcal{A}_2}^{\mathcal{H}_2^{\mathrm{aff}}(q)}L(i)\circledast L(j)/L$ is isomorphic to $V(i,j).$ Hence $M\cong V(i,j).$
\end{proof}

Observe from the proof above that if $i\neq j,j \pm 1$ then the completely splittable $\mathcal{H}_2^{\mathrm{aff}}(q)$-module $V (i,j)$ has two weights $(i,j)$ and $(j,i)$ and moreover $T_1-\Xi_1$ gives a bijection between the associated weight spaces. This together with Corollary \ref{Cor-1} and Proposition \ref{prop:restr-rank-2} leads to the following.

\begin{cor}\label{Cor-2}
Let $M \in\op{Rep}_{\mathbb{I}} \aHn$ be irreducible completely splittable. Suppose $0\neq z\in M_{\underline{i}}$ for some $\underline{i}=(i_1,\cdots,i_n)\in\mathbb{I}^n.$ The following holds for $1\leq k\leq n - 1.$\\
(1) If $i_k= i_{k+1}\pm 1,$ then $T_k z=\Xi_k z.$\\
(2) If $i_k\neq i_{k+1}\pm 1,$ then $0\neq (T_k-\Xi_k)z\in M_{s_k\cdot\underline{i}}$ and hence $s_k\cdot\underline{i}$ is a weight of $M$.
\end{cor}

\begin{defn}\label{defn:admissible}
Let $\underline{i}\in\mathbb{I}^n.$ For $1\leq k\leq n-1,$ the simple transposition $s_k$ is called admissible with respect to $\underline{i}$ if $i_k\neq i_{k+1}\pm 1.$
\end{defn}

Let $\mathfrak{P}(\aHn)$ be the set of weights $\underline{i}\in\mathbb{I}^n$ of irreducible completely splittable $\aHn$-modules. By Corollary \ref{Cor-2}, if $\un i \in \mathfrak{P}(\aHn)$ and $s_k$ is admissible with respect to $\underline{i},$ then $s_k \cdot \underline{i}\in \mathfrak{P}(\aHn)$; moreover $\underline{i}$ and $s_k\cdot \underline{i}$ must occur as weights in an irreducible completely splittable $\aHn$-module simultaneously.

\begin{prop}\label{lem:restr.1}
Let $\underline{i}\in \mathfrak{P}(\aHn).$ Suppose that $i_k=i_{k+2}$ for some $1\leq k\leq n-2.$\\
(1) If $h=\infty,$ then $i_k=i_{k+2}=0,i_{k+1}=1.$\\
(2) If $h\geq 3$ is odd, then either $i_k=i_{k+2}=0,i_{k+1}=1$ or $i_k=i_{k+2}=\frac{h-3}{2},i_{k+1}=\frac{h-1}{2}.$\\
(3) If $h\geq 4$ is even, then either $i_k=i_{k+2}=0,i_{k+1}=1$ or $i_k=i_{k+2}=\frac{h}{2}-1,i_{k+1}=\frac{h}{2}-2.$
\end{prop}
\begin{proof}
Suppose $\underline{i}$ occurs in an irreducible completely splittable $\aHn$-module $M$ and $i_k = i_{k+2}=u\in\I$ for some $1\leq k\leq n-2.$ If $i_k\neq i_{k+1} \pm 1,$ then $s_k\cdot \underline{i}$ is a weight of $M$ of the form $(\cdots,u,u,\cdots)$ by Corollary $\ref{Cor-2}$. This contradicts with  Lemma \ref{lem:separate weig.}. Hence $i_k = i_{k+1} \pm 1.$ This together with Corollary \ref{Cor-2} shows that $T_k = \Xi_k$ and $T_{k+1} = \Xi_{k+1}$ on $M_{\underline{i}}$. Thus by \eqref{Xi} we have for any $z\in M_{\un i}$ 
\begin{equation}\label{eq:braid-proof-1}
-\frac{1}{\varepsilon^3}T_kT_{k+1}T_kz=\sum_{a=1}^8 \mathscr{T}_{1,a}z
\end{equation}
with 
\[\begin{aligned}
\mathscr{T}_{1,1}z=&\frac{1}{X_kX_{k+1}^{-1}-1}\frac{1}{X_{k+1}X_{k+2}^{-1}-1}\frac{1}{X_kX_{k+1}^{-1}-1}z,\\
\mathscr{T}_{1,2}z=&\frac{1}{X_kX_{k+1}-1}C_kC_{k+1}\frac{1}{X_{k+1}X_{k+2}^{-1}-1}\frac{1}{X_kX_{k+1}-1}C_kC_{k+1}z,\\
\mathscr{T}_{1,3}z=&\frac{1}{X_kX_{k+1}-1}C_kC_{k+1}\frac{1}{X_{k+1}X_{k+2}-1}C_{k+1}C_{k+2}\frac{1}{X_kX_{k+1}^{-1}-1}z,\\
\mathscr{T}_{1,4}z=&\frac{1}{X_kX_{k+1}^{-1}-1}\frac{1}{X_{k+1}X_{k+2}-1}C_{k+1}C_{k+2}\frac{1}{X_kX_{k+1}-1}C_kC_{k+1}z,\\
\mathscr{T}_{1,5}z=&-\frac{1}{X_kX_{k+1}-1}C_kC_{k+1}\frac{1}{X_{k+1}X_{k+2}^{-1}-1}\frac{1}{X_kX_{k+1}^{-1}-1}z,\\
\mathscr{T}_{1,6}z=&-\frac{1}{X_kX_{k+1}^{-1}-1}\frac{1}{X_{k+1}X_{k+2}-1}C_{k+1}C_{k+2}\frac{1}{X_kX_{k+1}^{-1}-1}z,\\
\mathscr{T}_{1,7}z=&-\frac{1}{X_kX_{k+1}^{-1}-1}\frac{1}{X_{k+1}X_{k+2}^{-1}-1}\frac{1}{X_kX_{k+1}-1}C_kC_{k+1}z,\\
\mathscr{T}_{1,8}z=&-\frac{1}{X_kX_{k+1}-1}C_kC_{k+1}\frac{1}{X_{k+1}X_{k+2}-1}C_{k+1}C_{k+2}\frac{1}{X_kX_{k+1}-1}C_kC_{k+1}z.
\end{aligned}\]
Similarly, for any $z\in M_{\un i}$ we have 
\begin{equation}\label{eq:braid-proof-2}
-\frac{1}{\varepsilon^3}T_{k+1}T_{k}T_{k+1}z=\sum_{a=1}^8 \mathscr{T}_{2,a}z
\end{equation}
with
\[\begin{aligned}
\mathscr{T}_{2,1}z=&\frac{1}{X_{k+1} X_{k+2}^{-1}-1}\frac{1}{X_kX_{k+1}^{-1}-1}\frac{1}{X_{k+1} X_{k+2}^{-1}-1}z,\\
\mathscr{T}_{2,2}z=&\frac{1}{X_{k+1} X_{k+2}-1}C_{k+1} C_{k+2}\frac{1}{X_kX_{k+1}^{-1}-1}\frac{1}{X_{k+1} X_{k+2}-1}C_{k+1} C_{k+2}z,\\
\mathscr{T}_{2,3}z=&\frac{1}{X_{k+1} X_{k+2}-1}C_{k+1} C_{k+2}\frac{1}{X_kX_{k+1}-1}C_{k}C_{k+1}\frac{1}{X_{k+1} X_{k+2}^{-1}-1}z,\\
\mathscr{T}_{2,4}z=&\frac{1}{X_{k+1} X_{k+2}^{-1}-1}\frac{1}{X_kX_{k+1}-1}C_kC_{k+1}\frac{1}{X_{k+1} X_{k+2}-1}C_{k+1} C_{k+2}z,\\
\mathscr{T}_{2,5}z=&-\frac{1}{X_{k+1} X_{k+2}-1}C_{k+1} C_{k+2}\frac{1}{X_kX_{k+1}^{-1}-1}\frac{1}{X_{k+1} X_{k+2}^{-1}-1}z,\\
\mathscr{T}_{2,6}z=&-\frac{1}{X_{k+1} X_{k+2}^{-1}-1}\frac{1}{X_kX_{k+1}-1}C_kC_{k+1}\frac{1}{X_{k+1} X_{k+2}^{-1}-1}z,\\
\mathscr{T}_{2,7}z=&-\frac{1}{X_{k+1} X_{k+2}^{-1}-1}\frac{1}{X_kX_{k+1}^{-1}-1}\frac{1}{X_{k+1} X_{k+2}-1}C_{k+1} C_{k+2}z,\\
\mathscr{T}_{2,8}z=&-\frac{1}{X_{k+1} X_{k+2}-1}C_{k+1} C_{k+2}\frac{1}{X_kX_{k+1}-1}C_kC_{k+1}\frac{1}{X_{k+1} X_{k+2}-1}C_{k+1} C_{k+2}z. 
\end{aligned}\]
Notice that
\[\begin{aligned}
&\mathscr{T}_{1,5}z+\mathscr{T}_{1,7}z\\
=&C_kC_{k+1}\frac{X_kX_{k+1}^2X_{k+2}-X_kX_{k+1}+X_k^2X_{k+1}^2-X_k^2X_{k+1}X_{k+2}}{(X_kX_{k+1}^{-1}-1)(X_kX_{k+1}-1)(X_{k+1}X_{k+2}^{-1}-1)(X_{k+1}X_{k+2}-1)}z\\
=&\mathscr{T}_{2,6}z+\mathscr{T}_{2,8}z\\
\end{aligned}\]
Meanwhile the following holds 
\[\begin{aligned}
&\mathscr{T}_{1,6}z+\mathscr{T}_{1,8}z\\
=&C_{k+1}C_{k+2}\frac{X_{k+1}^2-X_{k+1}X_{k+2}+X_kX_{k+1}^2X_{k+2}-X_kX_{k+1}}{(X_kX_{k+1}^{-1}-1)(X_kX_{k+1}-1)(X_{k+1}X_{k+2}^{-1}-1)(X_{k+1}X_{k+2}-1)}z\\
=&\mathscr{T}_{2,5}z+\mathscr{T}_{2,7}z\\
\end{aligned}\]
This together with \eqref{eq:braid-proof-1} and \eqref{eq:braid-proof-2} gives rise to 

\begin{equation}\label{braid-part2}
\begin{aligned}
0=&-\frac{1}{\varepsilon^3}(T_kT_{k+1}T_k-T_{k+1}T_{k}T_{k+1})z=\pounds_1z+\pounds_2z+\pounds_3z+\pounds_4z
\end{aligned}
\end{equation}
with 
\begin{equation}\label{eq:braid-part3-1}
\begin{aligned}
&\pounds_1z:=(\mathscr{T}_{1,1}-\mathscr{T}_{2,1})z\\
&=\left(\frac{X_k^{-2}X_{k+1}^2}{(X_k^{-1}X_{k+1}-1)^2(X_{k+1}X_{k+2}^{-1}-1)}+\frac{X_k^{-1}X_{k+1}}{(X_k^{-1}X_{k+1}-1)(X_{k+1}X_{k+2}^{-1}-1)^2}\right)z
\end{aligned}
\end{equation}
\begin{equation}\label{eq:braid-part3-2}
\begin{aligned}
&\pounds_2z:=(\mathscr{T}_{1,2}-\mathscr{T}_{2,2})z\\
&=-\left(\frac{X_kX_{k+1}^2X_{k+2}}{(X_kX_{k+1}-1)^2(X_{k+1}X_{k+2}-1)}+\frac{X_{k+1}X_{k+2}}{(X_kX_{k+1}-1)(X_{k+1}X_{k+2}-1)^2}\right)z
\end{aligned}
\end{equation}
\begin{equation}\label{eq:braid-part3-3}
\begin{aligned}
&\pounds_3z:=(\mathscr{T}_{1,4}-\mathscr{T}_{2,3})z\\
&=-C_{k}C_{k+2}\left(\frac{X_k^2X_{k+1}^2}{(X_kX_{k+1}-1)^2(X_{k+1}X_{k+2}^{-1}-1)}+\frac{X_kX_{k+1}}{(X_{k}X_{k+1}-1)(X_{k+1}X_{k+2}^{-1}-1)^2}\right)z
\end{aligned}
\end{equation}
\begin{equation}\label{eq:braid-part3-4}
\begin{aligned}
&\pounds_4z:=(\mathscr{T}_{1,3}-\mathscr{T}_{2,4})z\\
&=C_{k}C_{k+2}\left(\frac{X_k^{-1}X_{k+1}^2X_{k+2}}{(X_k^{-1}X_{k+1}-1)^2(X_{k+1}X_{k+2}-1)}+\frac{X_{k+1}X_{k+2}}{(X_{k}^{-1}X_{k+1}-1)(X_{k+1}X_{k+2}-1)^2}\right)z
\end{aligned}
\end{equation}
for any $z\in M_{\underline{i}}$.   

Since $X_k^{\pm 1},X_{k+1}^{\pm 1},X_{k+2}^{\pm 1}$ act semisimply on $M_{\underline{i}},$ $M_{\underline{i}}$ admits a decomposition $M_{\underline{i}}=M^{1,+}_{\un i}\oplus M^{1,-}_{\un i}\oplus M^{2,+}_{\un i}\oplus M^{2,-}_{\un i},$ where $M^{1,\pm }_{\un i}=\{z\in M_{\underline{i}}|X_kz=X_{k+2}z=\mathtt{b}_{\pm}(i_k)\}$ and $M^{2,\pm}_{\un i}=\{z\in M_{\underline{i}}|X_kz=X_{k+2}^{-1}z=\mathtt{b}_\pm(i_k)\}.$  For convenience,  set $a=\mathtt{b}_+(i_k)$ or $a=\mathtt{b}_-(i_k)$ 
and set $b=\mathtt{b}_+(i_{k+1})$ or $b=\mathtt{b}_-(i_{k+1})$. Then we have 
\begin{equation}
\begin{aligned}
X_kz=X_{k+2}z=&az, \quad X_{k+1}z=bz,\quad \text{ for any }z\in M^{1,\pm}_{\un i},\\
X_kz=az, \quad X_{k+2}z=&a^{-1}z, \quad X_{k+1}z=bz,\quad \text{ for any }z\in M^{2,\pm}_{\un i}. 
\end{aligned}
\end{equation}
Thus, for $z\in M^{1,\pm}_{\un i},$  by \eqref{eq:braid-part3-1}-\eqref{eq:braid-part3-4} the following holds 
\[
\pounds_1z==\frac{a^{-1}b(a^{-1}b+1)}{(a^{-1}b-1)^3}z,\qquad \pounds_2z=-\frac{ab(ab+1)}{(ab-1)^3}z, 
\]
\[\begin{aligned}
\pounds_3 z&=-C_kC_{k+1}\left(\frac{a^2b^2}{(ab-1)^2(a^{-1}b-1)}+\frac{ab}{(ab-1)(a^{-1}b-1)^2}\right)z,\\
\pounds_4z &=C_kC_{k+1}\left(\frac{b^2}{(ab-1)(a^{-1}b-1)^2}+\frac{ab}{(ab-1)^2(a^{-1}b-1)}\right)z. 
\end{aligned}\]
Notice that
\[\begin{aligned}
\pounds_3 z+\pounds_4z&=C_kC_{k+1}\left(\frac{-ab(ab-1)}{(ab-1)^2(a^{-1}b-1)}+\frac{ab(a^{-1}b-1)}{(ab-1)(a^{-1}b-1)^2}\right)z\\
&=C_kC_{k+1}\left(\frac{-ab}{(ab-1)(a^{-1}b-1)}+\frac{ab}{(ab-1)(a^{-1}b-1)}\right)z\\
&=0.
\end{aligned}\]
These together with \eqref{braid-part2} give rise to 
\begin{equation}\label{eq:braid-rem-1}
\left(\frac{a^{-1}b(a^{-1}b+1)}{(a^{-1}b-1)^3}-\frac{ab(ab+1)}{(ab-1)^3}\right)z=0
\end{equation}
for any $z\in M^{1,\pm}_{\un i}$. 
Similarly,  for $z\in M^{2,\pm}_{\un i},$  by \eqref{eq:braid-part3-1}-\eqref{eq:braid-part3-4}  the following holds 
\[
\begin{aligned}
\pounds_1z&=\left(\frac{a^{-2}b^2}{(ab-1)(a^{-1}b-1)^2}+\frac{a^{-1}b}{(ab-1)^2(a^{-1}b-1)}\right)\\
\pounds_2z&=-\left(\frac{b^2}{(ab-1)^2(a^{-1}b-1)}+\frac{a^{-1}b}{(ab-1)(a^{-1}b-1)^2}\right)
\end{aligned}
\]
\[
\pounds_3z=-C_kC_{k+1}\frac{ab(ab+1)}{(ab-1)^3}z,\qquad \pounds_4z=C_kC_{k+1}\frac{a^{-1}b(a^{-1}b+1)}{(a^{-1}b-1)^3}z
\]
Notice that
\[\begin{aligned}
\pounds_1z+\pounds_2z&=\left(\frac{a^{-1}b(a^{-1}b-1)}{(ab-1)(a^{-1}b-1)^2}+\frac{-a^{-1}b(ab-1)}{(ab-1)^2(a^{-1}b-1)}\right)z\\
&=\left(\frac{a^{-1}b}{(ab-1)(a^{-1}b-1)}+\frac{-a^{-1}b}{(ab-1)(a^{-1}b-1)}\right)z\\
&=0.
\end{aligned}\]
These together with \eqref{braid-part2} give rise to 
\begin{equation}\label{eq:braid-rem-2}
C_kC_{k+1}\left(\frac{a^{-1}b(a^{-1}b+1)}{(a^{-1}b-1)^3}-\frac{ab(ab+1)}{(ab-1)^3}\right)z=0
\end{equation}
for any $z\in M^{2,\pm}_{\un i}$. 
Combining \eqref{eq:braid-rem-1} and \eqref{eq:braid-rem-2}, we have 
\[
\left(\frac{a^{-1}b(a^{-1}b+1)}{(a^{-1}b-1)^3}-\frac{ab(ab+1)}{(ab-1)^3}\right)z=0,
\]
for any $z\in M_{\un i}$. Hence 
\begin{equation}\label{eq:braid-final}
\frac{b^3}{(a^{-1}b-1)^3(ab-1)^3}(a-a^{-1})\left((b+b^{-1})^2+(a+a^{-1})(b+b^{-1})-8\right)=0.
\end{equation}
Since $a=\mathtt{b}_{\pm}(i_k), b=\mathtt{b}_{\pm}(i_{k+1})$, we have $a+a^{-1}=\mathtt{q}(i_k)$ and $b+b^{-1}=\mathtt{q}(i_{k+1})$. Plugging in to \eqref{eq:braid-final} and  the fact that $i_{k+1} = i_k \pm 1$ verified above, one can easily prove that \eqref{eq:braid-final} leads to 
\begin{equation}\label{eq:braid-final-1}
(\mathtt{q}(i_k)^2-4)(q^{4i_k}-1)(q^{4i_k-2}-1)=0\quad  \text{ if } i_{k+1}=i_k-1
\end{equation}
and
\begin{equation}\label{eq:braid-final-2}
(\mathtt{q}(i_k)^2-4)(q^{4i_k+4}-1)(q^{4i_k+6}-1)=0\quad \text{ if }i_{k+1}=i_k+1. 
\end{equation}
Concerning the solutions of \eqref{eq:braid-final-1} and \eqref{eq:braid-final-2}, we have the following three cases: 
\noindent{\bf Case 1} $h=\infty$: the equation \eqref{eq:braid-final-1} has no solution and the solution of \eqref{eq:braid-final-2} is $i_k=0,i_{k+1}=1.$
{\bf Case 2} $h\geq 3$ is odd:  the equation \eqref{eq:braid-final-1} has no solution and the solution of \eqref{eq:braid-final-2} is $i_k=0,i_{k+1}=1$ or $i_k=\frac{h-3}{2},i_{k+1}=\frac{h-1}{2}.$
{\bf Case 3} $h\geq 4$ is even:  the solution of \eqref{eq:braid-final-1} is $i_k=\frac{h}{2}-1,i_{k+1}=\frac{h}{2}-2$, and the solution of \eqref{eq:braid-final-2} is $i_k=0,i_{k+1}=1$. 
Putting together the proposition is verified. 
\end{proof}

\begin{lem}\label{lem:restr.2}
Let $\underline{i}\in \mathfrak{P}(\aHn).$ Suppose $i_k=i_\ell=u\in\I$ for some $1\leq k<\ell\leq n.$
\begin{enumerate}
\item If $h\geq 3$ is odd or $h=\infty$, then $u+1\in\{i_{k+1},\ldots,i_{\ell-1}\}.$
\item If $h\geq 4$ is even, then  $1\in \{i_{k+1},\ldots,i_{\ell-1}\}$ if $u=0$ and  $u-1\in\{i_{k+1},\ldots,i_{\ell-1}\}$ if $u>0$. 
\end{enumerate}
\end{lem}
\begin{proof}
We shall prove (2) first.  
Now assume $h\geq 4$ is even. If $u = 0$, then $1\in \{i_{k+1},\cdots,i_{\ell-1}\}$; otherwise we can apply admissible transpositions to $\underline i$ to obtain an element in $\mathfrak{P}(\aHn)$ of the form $(\ldots, 0, 0,\ldots),$ which contradicts with Lemma \ref{lem:separate weig.}.

Now assume $u \geq 1$ and $u - 1 \notin \{i_{k+1},\ldots ,i_{\ell-1}\}.$ If $u + 1$ does not appear between $i_{k+1}$ and $i_{\ell-1}$ in $\underline i$, then we can apply admissible transpositions to $\underline i$ to obtain an element in $\mathfrak{P}(\aHn)$ of the form $(\ldots ,u,u, \ldots),$ which contradicts with Lemma \ref{lem:separate weig.}. If $u + 1$ appears only once between $i_{k+1}$ and $i_{\ell-1}$ in $\underline i$, then we can apply admissible transpositions to $\underline i$ to obtain an element in $\mathfrak{P}(\aHn)$ of the form $(\ldots,u,u + 1,u,\ldots),$ which contradicts Proposition \ref{lem:restr.1}. Hence $u + 1$ appears at least twice between $i_{k+1}$ and $i_{\ell-1}$ in $\underline i$. This implies that there exist $k<k_1 < \ell_1 < \ell$ such that
\[i_{k_1} = i_{\ell_1} = u + 1,\{u,u + 1\}\cap \{i_{k_1+1},\ldots,i_{\ell_1-1}\}=\emptyset.
\]
An identical argument shows that there exist $k_1 < k_2 < \ell_2 < \ell_1$ such that
\[i_{k_2} = i_{\ell_2} = u + 2,\{u,u + 1,u+2\}\cap \{i_{k_2+1},\ldots,i_{\ell_2-1}\}=\emptyset.
\]
Continuing in this way, we obtain $k<s<t<\ell$ such that
\[i_{s} = i_{t} = \frac{h}{2}-1,\{u,u + 1,\ldots,\frac{h}{2}-1\}\cap \{i_{s+1},\ldots,i_{t-1}\}=\emptyset.
\]
Then we can apply admissible transpositions to $\underline i$ to obtain an element in $\mathfrak{P}(\aHn)$ of the form $(\ldots,\frac{h}{2}-1,\frac{h}{2}-1,\ldots),$ which contradicts with Proposition  \ref{lem:separate weig.}. This proves (2).  We leave the proof of (1) to the reader as it can verified similarly or using the similar argument in  the proof of \cite[Lemma 3.13]{Wa}. 
\end{proof}
\begin{prop}\label{prop:restr.2}
Let $\underline{i}\in \mathfrak{P}(\aHn).$ Then\\
$(1)$ $i_k\neq i_{k+1}$ for all $1\leq k\leq n-1.$\\
$(2)$ If $h\geq 3$ is odd, then $\frac{h-1}{2}$ appears at most once in $\underline{i}.$\\
$(3)$ If $i_k=i_\ell=0$ for some $1\leq k<\ell\leq n,$ then $1\in\{i_{k+1},\ldots,i_{\ell-1}\}.$\\
$(4)$ If $h=\infty$ and $i_k=i_\ell\geq 1$ for some $1\leq k<\ell\leq n,$ then $\{i_{k}-1,i_k+1\}\subset\{i_{k+1},\ldots,i_{\ell-1}\}.$\\
$(5)$ If $h\geq 3$ is odd and $i_k=i_\ell\geq 1$ for some $1\leq k<\ell\leq n,$ then either of the following holds:
\begin{itemize}
\item[(a)] $\{i_{k}-1,i_k+1\}\subset\{i_{k+1},\ldots,i_{\ell-1}\}.$

\item[(b)] there exists a sequence of integers $k\leq r_0<r_1<\cdots<r_{\frac{h-3}{2}-i_k}<q<t_{\frac{h-3}{2}-i_k}<\cdots<t_1<t_0\leq \ell$ such that $i_q=\frac{h-1}{2},i_{r_j}=i_{t_j}=i_k+j$ and $i_k+j$ does not appear betwwen $i_{r_j}$ and $i_{t_j}$ in $\underline{i}$ for each $0\leq j\leq \frac{h-3}{2}-i_k$.
\end{itemize}
$(6)$ If $h\geq 4$ is even and $i_k=i_\ell\geq 1$ for some $1\leq k<\ell\leq n,$ then:
\begin{itemize}
\item[(a)] If $i_k=i_\ell<\frac{h}{2}-1,$ then $\{i_{k}-1,i_k+1\}\subset\{i_{k+1},\ldots,i_{\ell-1}\}.$

\item[(b)] If $i_k=i_\ell=\frac{h}{2}-1,$ then $i_{k}-1\in\{i_{k+1},\ldots,i_{\ell-1}\}.$ Moreover, denote $m=\#\{t|i_t=\frac{h}{2}-1,k\leq t\leq \ell\}$  then there exists a unique set of integers $1\leq k_{a;b}\leq n$ with $1\leq a \leq \min\{m,\frac{h}{2}\},1\leq b \leq m-(a-1)$,  such that
\begin{itemize}
\item[(i)]
$i_{k_{a;b}}=\frac{h}{2}-a$ for all $a,b.$
\item[(ii)]
$k_{a;b}< k_{a+1;b}<k_{a;b+1}$ for all $a,b.$
\item[(iii)]
$\frac{h}{2}-a$ does not appear between $i_{k_{a;b}}$ and $i_{k_{a;b+1}}$ in $\un i$ for all $a,b.$ Moreover, $k_{a+1;b}$ is the unique index in $\{k_{a;b},k_{a;b}+1,\ldots, k_{a,b+1}\}$ such that $i_{k_{a+1;b}}=\frac{h}{2}-a-1.$
\end{itemize}
\end{itemize}

\end{prop}
\begin{proof}
(1) It follows from Lemma~\ref{lem:separate weig.}.

(2) If $\frac{h-1}{2}$ appears more than once in $\underline{i}$,
then it follows from Lemma~\ref{lem:restr.2} that $\frac{h+1}{2}$
appears in $\underline{i}$ which is impossible since
$\frac{h+1}{2}\notin\I$.

(3) It follows from Lemma~\ref{lem:restr.2}.

(4) This is same as \cite[Proposition 3.14(4)]{Wa}. Now suppose $h=\infty$ and $i_k=i_l=u\geq 1$ for some $1\leq k<l\leq
n$. Without loss of generality, we can assume
$u\notin\{i_{k+1},\ldots,i_{l-1}\}$. By Lemma~\ref{lem:restr.2} we
have $u+1\in\{i_{k+1},\ldots,i_{l-1}\}$ and hence it suffices to
show $u-1\in\{i_{k+1},\ldots,i_{l-1}\}$. Now assume
$u-1\notin\{i_{k+1},\ldots,i_{l-1}\}$. Then $u+1$ must appear in the
subsequence $(i_{k+1},\ldots,i_{l-1})$ at least twice, otherwise we
can apply admissible transpositions to $\underline{i}$ to obtain an
element in $\mathfrak{P}(\aHn)$ of the form $(\ldots,u,u+1,u\ldots)$ which
contradicts Proposition~\ref{lem:restr.1}. Hence there exist $k<k_1<l_1<l$
such that
$$
i_{k_1}=i_{l_1}=u+1,\quad  u+1 \text{ does not appear between
}i_{k_1}\text{ and }i_{l_1} \text{ in }\underline{i}.
$$
Since
$u\notin\{i_{k+1},\ldots,i_{l-1}\}\supseteq\{i_{k_1+1},\ldots,i_{l_1-1}\}$,
a similar argument gives $k_2,l_2$ with $k_1<k_2<l_2<l_1$ such
that
$$
i_{k_2}=i_{l_2}=u+2, \quad u+2 \text{ does not appear between
}i_{k_2}\text{ and }i_{l_2} \text{ in }\underline{i}.
$$
Continuing in this way we see that any integer greater than $u$ will
appear in the subsequence $(i_{k+1},\ldots,i_{l-1})$ which is
impossible. Hence $u-1\in\{i_{k+1},\ldots, i_{l-1}\}$.

(5)  We leave the proof of (5) to the reader as it is the same as \cite[Proposition 3.14(5)]{Wa}.

(6) Firstly,  suppose $h\geq 4$ is even and $1\leq i_k=i_\ell=u< \frac{h}{2}-1$ for some $1\leq k <\ell\leq n$. Without loss of generality, we can assume $u\notin\{i_{k+1},\ldots,i_{\ell-1}\}.$ By Proposition \ref{lem:restr.1} we have $u-1\in \{i_{k+1},\ldots,i_{\ell-1}\}.$ Now assume $u+1\notin\{i_{k+1},\ldots,i_{\ell-1}\}$. Then $u-1$ must appear twice between $i_{k+1}$ and $i_{\ell-1}$ in $\underline i$, otherwise we can apply admissible transpositions to $\underline i$ to obtain an element in $\mathfrak{P}(\aHn)$ of the form $(\ldots,u,u - 1,u,\ldots),$ which contradicts Proposition \ref{lem:restr.1}. This implies that there exist $k<k_1 < \ell_1 < \ell$ such that
\[i_{k_1} = i_{\ell_1} = u - 1,\{u,u - 1\}\cap \{i_{k_1+1},\ldots,i_{\ell_1-1}\}=\emptyset.
\]
An identical argument shows that there exist $k_1 < k_2 < \ell_2 < \ell_1$ such that
\[i_{k_2} = i_{\ell_2} = u - 2,\{u,u - 1,u-2\}\cap \{i_{k_2+1},\ldots,i_{\ell_2-1}\}=\emptyset.
\]
Continuing in this way, we obtain $k<s<t<\ell$ such that
\[i_{s} = i_{t} = 0,\{u,u - 1,\ldots,0\}\cap \{i_{s+1},\ldots,i_{t-1}\}=\emptyset.
\]
Then we can apply admissible transpositions to $\underline i$ to obtain an element in $\mathfrak{P}(\aHn)$ of the form $(\ldots,0,0,\ldots),$ which contradicts with Lemma \ref{lem:separate weig.}. This proves part (a) in (6). 

{\color{black} Secondly, assume  $1\leq i_k=i_\ell= \frac{h}{2}-1$ for some $1\leq k <\ell\leq n$. Let $m=\#\{t|i_t=\frac{h}{2}-1,k\leq t\leq \ell\}$, then there exists $k\leq k_{1;1}<k_{1;2}<\cdots<k_{1;m}\leq\ell$ such  $i_{k_{1;1}}=i_{k_{1;2}}=\cdots=i_{k_{1;m}}=\frac{h}{2}-1$. For each $1\leq b\leq m-1$, we have $i_{k_{1;b}}=i_{k_{1;b+1}}=\frac{h}{2}-1$ and $i_s\neq \frac{h}{2}-1$ for all $k_{1;b}<s<k_{1;b+1}$, then by Lemma \ref{lem:restr.2} there exists a unique $k_{2;b}$ satisfying $k_{1;b}<k_{2;b}<k_{1;b+1}$ and moreover $i_{k_{2;b}}=\frac{h}{2}-2, i_s\neq \frac{h}{2}-2$ for $k_{1;b}<s\neq k_{2;b}<k_{1;b+1}$. Now we obtain $i_{k_{2;1}}=i_{k_{2;2}}=\cdots=i_{k_{2;m-1}}=\frac{h}{2}-2$ and moreover   $i_s\neq \frac{h}{2}-2$ for all $k_{2;b}<s<k_{2;b+1}$ and $1\leq b\leq m-2$. Again by Lemma \ref{lem:restr.2} we obtain that  there exists $k_{3;b}$ for each $1\leq b\leq m-2$ satisfying $k_{2;b}<k_{3;b}<k_{2;b+1}$ and $i_{k_{3;b}}=\frac{h}{2}-3$.  Continuing this way, we eventually obtain a set of integers $\{k_{a;b}\mid 1\leq a\leq m, 1\leq b\leq m-a+1\}$ satisfying the properties (i), (ii), (iii) if $1\leq m\leq \frac{h}{2}$. If $m>\frac{h}{2}$, then while continuing the above way one can obtain integers $k_{\frac{h}{2}-1;b}<k_{\frac{h}{2};b}<k_{\frac{h}{2}-1;b+1}$  such that $i_{k_{\frac{h}{2}};b}=0$ for each $1\leq b\leq m-\frac{h}{2}+1$. 
This proves the argument (b) in (6). }
\end{proof}

{
\section{Classification of irreducible completely splittable $\aHn$-modules}\label{sec:classifications}

In this section, we shall give an explicit construction and a classification of irreducible completely splittable $\aHn$-modules.

Recall that for $\un i \in \mathbb I^n$ and $1 \leq k\leq n-1,$ the simple transposition $s_k$ is said to be admissible with respect to $\underline i$ if $i_k\neq i_{k+1}\pm 1.$ Define an equivalence relation $\sim$ on $\mathbb I^n$ by declaring that $\underline i \sim \underline j$ if there exist $s_{k_1},\cdots,s_{k_t}$ for some $t \in \mathbb Z_+$ such that $\underline j = (s_{k_t}\cdots s_{k_1} ) \cdot i$ and $s_{k_\ell}$ is admissible with respect to $(s_{k_{\ell-1}}\cdots s_{k_1} ) \cdot \underline i$ for $1 \leq \ell\leq t.$

Set 
\begin{equation}
\mathfrak{P}'(\aHn)=\{\underline i \in\mathbb I^n\mid \un i \text{ satisfies (1)-(6) in Proposition \ref{prop:restr.2}} \}.
\end{equation}
 Observe that if $\underline i \in \mathfrak{P}'(\aHn)$ and $s_k$ is admissible with respect to $\underline i,$ then the properties in Proposition \ref{prop:restr.2} hold for $s_k\cdot\underline i$ and hence $s_k \cdot \underline i \in \mathfrak{P}^\prime (\aHn).$ This means there is an equivalence relation denoted by $\sim$ on $\mathfrak{P}^\prime (\aHn)$ inherited from the equivalence relation $\sim$ on $\mathbb I^n$. For each $\underline i \in \mathfrak{P}'(\aHn),$ set $\Lambda_{\un i}:=\{\un j\mid \un j\sim \un i\}$ and 
\begin{equation}\label{Pi}
\begin{aligned}
P_{\underline i}=\big\{\tau=s_{k_t}\cdots s_{k_1}\mid & s_{k_\ell}\text{ is admissible with respect to }\\
&s_{k_{\ell-1}}\cdots s_{k_1}\cdot \underline i,1\leq \ell\leq t,t\in\mathbb{Z}_+\big\}.
\end{aligned}
\end{equation}
 The following can be proved using arguments similar to the proof of \cite[Lemma 4.1]{Wa}. 
\begin{lem}\label{lem:Pi}
Let  $\underline i \in\mathfrak{P}'(\aHn).$ Then the map
\[\varphi:P_{\underline i}\rightarrow \Lambda_{\un i},\tau\mapsto\tau\cdot\underline i\]
is bijective.
\end{lem}
\begin{proof}
By the definitions of $P_{\underline{i}}$ and the equivalence
relation $\sim$ on $\mathfrak{P}'(\aHn)$, one can check that $\varphi$ is
surjective. Note that if $\tau, \sigma\in P_{\underline{i}}$ then
$\sigma^{-1}\tau\in P_{\underline{i}}$. Therefore, to check the
injectivity of $\varphi$, it suffices to show that for $\tau\in
P_{\underline{i}}$ if $\tau\cdot\underline{i}=\underline{i}$ then
$\tau=1$. Associated to each $\underline{j}\in \mathfrak{P}'(\aHn)$,
there exists a unique table $\Gamma(\underline{j})$ whose $a$th column
consists of all numbers $k$ with $j_k=a$ and is increasing for each
$a\in\I$. Since $\underline{j}\in \mathfrak{P}'(\aHn)$, $j_k\neq j_{k+1}$ and
hence $k$ and $k+1$ are in different columns in
$\Gamma(\underline{j})$ for each $1\leq k\leq n-1$. This means each
simple transposition $s_k$ can naturally act on the table
$\Gamma(\underline{j})$ by switching $k$ and $k+1$ to obtain a new
table denoted by $s_k\cdot\Gamma(\underline{j})$. 
It is clear that
\begin{align}
s_k\cdot\Gamma(\underline{j})=\Gamma(s_k\cdot\underline{j}), \quad
1\leq k\leq n-1.\label{table}
\end{align}

Since $\tau\in P_{\underline{i}}$, we can write
$\tau=s_{k_t}s_{k_{t-1}}\cdots s_{k_1}$ so that $s_{k_l}$ is
admissible with respect to $s_{k_{l-1}}\cdots
s_{k_1}\cdot\underline{i}$ for each $1\leq l\leq t$. Observe that
$s_{k_{l-1}}\cdots s_{k_1}\cdot\underline{i}\in \mathfrak{P'}(\aHn)$ and
hence there exists a table $\Gamma(s_{k_{l-1}}\cdots
s_{k_1}\cdot\underline{i})$ as defined above for $1\leq l\leq t$.
By (\ref{table}) we have
$$
s_{k_l}\cdot\Gamma(s_{k_{l-1}}\cdots s_{k_1}\cdot\underline{i})
=\Gamma(s_{k_l}s_{k_{l-1}}\cdots s_{k_1}\cdot\underline{i})
$$
for $1\leq l\leq t$. This implies
$$
\tau\cdot\Gamma(\underline{i})=s_{k_t}\cdots
s_{k_1}\cdot\Gamma(\underline{i})=\Gamma(s_{k_t}\cdots
s_{k_1}\cdot\underline{i})=\Gamma(\underline{i}).
$$
Therefore $\tau=1$.

\end{proof}

Let $M\in \op{Rep}_{\mathbb I} \aHn$ be irreducible completely splittable and suppose $M_{\un i}\neq 0$ for some $\un i = (i_1,\cdots,i_n)\in\mathbb I^n.$ Recall the linear operators $\Xi_k$ and $\Omega_k$ on $M_{\un i}$ from (\ref{Xi}) and (\ref{Omega}). If $s_k$ is admissible with respect to $\un i,$ then $i_k\neq i_{k+1} \pm 1$ and hence on $M_{\un i}$ the linear operator $\Omega_k$ acts as a nonzero scalar, which is invertible. Therefore we can define the linear map $\widehat\Phi_k$ as follows:
\[\begin{aligned}
\widehat\Phi_k&:M_{\un i}\rightarrow M,\\
&z\mapsto (T_k-\Xi_k)\frac{1}{\Omega_k}z.
\end{aligned}\]
The following is parallel to the degenerate case \cite[Lemma 4.2]{Wa} and the proof is similar. 
\begin{lem}\label{lem:newoperator}
Let $M\in \op{Rep}_{\mathbb I} \aHn$ be irreducible completely splittable and $1\leq k\leq n-1$. Assume that $M_{\un i}\neq 0$ and that $s_k$ is admissible with respect to $\un i$ for some $\un i = (i_1,\cdots ,i_n)\in \mathbb I^n$. Then,
\begin{itemize}
\item[(1)]
The action of $\widehat\Phi_k$ on $M_{\un i}$ satisfies
\begin{equation}\label{XC-interprime}
\begin{aligned}
\widehat{\Phi}_kX_k^{\pm1}&=X_{k+1}^{\pm1}\widehat{\Phi}_k,\widehat{\Phi}_kX_{k+1}^{\pm1}=X_k^{\pm1}\widehat{\Phi}_k,\widehat{\Phi}_kX_{l}^{\pm1}=X_{l}^{\pm1}\widehat{\Phi}_k,\\
\widehat{\Phi}_kC_k&=C_{k+1}\widehat{\Phi}_k,\widehat{\Phi}_kC_{k+1}=C_k\widehat{\Phi}_k,\widehat{\Phi}_kC_l=C_l\widehat{\Phi}_k
\end{aligned}
\end{equation}
for  $1\leq l\leq n$ with $l\neq k,k+1.$ Hence for each $z\in M_{\un i},\widehat\Phi_k(z)\in M_{s_k\cdot \un i}.$
\item[(2)] 
$\widehat\Phi^2_k =1,$ and hence $\widehat\Phi_k: M_{\un i} \rightarrow M_{s_k\cdot \un i}$ is a bijection. 
\item[(3)] The following holds whenever both sides are well-defined: 
\begin{equation}\label{braidprime}
\begin{aligned}
\widehat{\Phi}_k\widehat{\Phi}_l&=\widehat{\Phi}_l\widehat{\Phi}_k,\quad \text{ for } 1\leq k,l\leq n-1 \text{ with }|k-l|>1,\\
\widehat{\Phi}_k\widehat{\Phi}_{k+1}\widehat{\Phi}_k&=\widehat{\Phi}_{k+1}\widehat{\Phi}_k\widehat{\Phi}_{k+1},\quad \text{ for } 1\leq k\leq n-2. 
\end{aligned}
\end{equation}
\end{itemize}
\end{lem}
\begin{proof}
Notice that 
\begin{equation}
\widehat\Phi_k(z)=\frac{1}{\mathsf{z}_k^2}\widetilde\Phi_k\frac{1}{\Omega_k}z\label{op}
\end{equation}
for any $z\in M_{\un i}$. Then it's esay to check \eqref{XC-interprime} via (\ref{Xinter})-(\ref{Cinter}). By (\ref{XC-interprime}), we have for any $z\in M_{\un i},$
\[(X_k+X^{-1}_k-\mathtt{q}(i_{k+1}))\widehat\Phi_k (z)=0,(X_{k+1}+X^{-1}_{k+1}-\mathtt{q}(i_k))\widehat\Phi_k (z)=0,(X_l+X^{-1}_l-\mathtt{q}(i_l))\widehat\Phi_k (z)=0\]
for all $\ell\neq k,k+1.$ This means $\widehat\Phi_kz\in M_{s_k\un i}.$ Hence (1) holds. 

Meanwhile, it's easy to see $\widetilde\Phi_k\Omega_kz=\Omega_k\widetilde\Phi_kz,\widetilde\Phi_k\mathsf{z}_kz=-\mathsf{z}_k\widetilde\Phi_kz$ for any $z\in M_{\un i}$ by \eqref{eq:zi}, \eqref{Xinter} and \eqref{Omega-0}. By (\ref{op}) and (\ref{Sqinter}), for $z\in M_{\un i},$
\[\widehat\Phi_k^2(z)=\frac{1}{\mathsf{z}_k^4}\widetilde \Phi_k^2\frac{1}{\Omega_k^2}z\]
for any $z\in M_{\un i}$. Since $\widetilde \Phi_k^2=\mathsf{z}_k^4\Omega_k^2$ by \eqref{eq:Phi-square} and \eqref{Omega-0}, we have $\widehat \Phi_k^2z=1$ and so $\widehat \Phi_k: M_{\un i}\rightarrow M_{s_k\cdot \un i}$ is bijective.

Finally, assume $|k-l|>1$ and both $\widehat\Phi_k\widehat\Phi_l$ and $\widehat\Phi_l\widehat\Phi_k$ are well-defined on $M_{\un i},$ for some $\un i\in\I^n$.  Then by (\ref{op}) and (\ref{Braidinter}) we see that
\[\widehat\Phi_k\widehat\Phi_l(z)=\widetilde\Phi_k\widetilde\Phi_l\frac{1}{\mathsf{z}_k^2\mathsf{z}_l^2\Omega_k\Omega_l}z,\quad \widehat\Phi_l\widehat\Phi_k(z)=\widetilde\Phi_l\widetilde\Phi_k\frac{1}{\mathsf{z}_k^2\mathsf{z}_l^2\Omega_k\Omega_l}z\]
for any $z\in M_{\un i}$. 
This together with (\ref{Braidinter}) verifies (\ref{braidprime}).
By (\ref{op}) and \eqref{Omega}, on can check that if both $\widehat\Phi_k\widehat\Phi_{k+1}\widehat\Phi_k$ and $\widehat\Phi_{k+1}\widehat\Phi_k\widehat\Phi_{k+1}$ are well-defined on $M_{\un i}$ for some $\un i \in\I^n$ then
\[
\begin{aligned}\widehat\Phi_k\widehat\Phi_{k+1}\widehat\Phi_k(z)&=C\widetilde\Phi_k\widetilde\Phi_{k+1}\widetilde\Phi_k(z)\\
\widehat\Phi_{k+1}\widehat\Phi_k\widehat\Phi_{k+1}(z)&=C\widetilde\Phi_{k+1}\widetilde\Phi_k\widetilde\Phi_{k+1}(z)
\end{aligned}
\]
where $C$ is the scalar
\[\begin{aligned}
C=\frac{1}{(a-b)^2(a-c)^2(b-c)^2}\frac{1}{\sqrt{1-\frac{\varepsilon^2(ab-4)}{(a-b)^2}}\sqrt{1-\frac{\varepsilon^2(ac-4)}{(a-c)^2}}\sqrt{1-\frac{\varepsilon^2(bc-4)}{(b-c)^2}}}\\
\end{aligned}\]
with $a=\mathtt{q}(i_k), b=\mathtt{q}(i_{k+1}),
c=\mathtt{q}(i_{k+2})$. Hence (\ref{braidprime}) holds by (\ref{Braidinter}).
\end{proof}

\begin{rem}\label{rem:admissible}
Suppose that $M\in \op{Rep}_{\mathbb I}\aHn$ is completely splittable. By Lemma \ref{lem:newoperator}, if $M_{\un i}\neq 0$ and $\un j\sim \un i,$ then $M_{\un j}\neq 0.$
\end{rem}
\begin{lem}\label{lem:newbij.}
Let $M \in  \op{Rep}_{\mathbb I} \aHn$ be irreducible completely splittable. Suppose that $M_{\un i}\neq 0$ for some $i \in  \mathbb I^n$ and $\tau  \in  P_{\un i}.$ Write $\tau  = s_{k_t}\cdots s_{k_1}$ so that $s_{k_\ell}$ is admissible with respect to $s_{k_l-1}\cdots s_{k_1}\cdot \un i$ for $1 \leq l\leq t.$ Then
\[\widehat\Phi_\tau:=\widehat\Phi_{k_t}\cdots \widehat\Phi_{k_1} : M_{\un i} \rightarrow M_{\tau\cdot \un i}\]
is a bijection satisfying $X^\pm_k\widehat\Phi_\tau  =\widehat\Phi_\tau  X^\pm_{\tau (k)}$ and $C_k\widehat\Phi_\tau  =\widehat\Phi_\tau  C_{\tau (k)}$ for $1\leq k\leq n.$ Moreover $\widehat\Phi_\tau$ does not depend on the choice of the expression $s_{k_t}\cdots s_{k_1}$ for $\tau.$
\end{lem}

\begin{proof}
Since $s_{k_l}$ is admissible with respect to $s_{k_{l-1}}\cdots
s_{k_1}\cdot \underline{i}$ for $1\leq l\leq t$, each
$\widehat{\Phi}_{k_l}$ is a well-defined bijection from
$M_{s_{k_{l-1}}\cdot s_{k_1}\cdot\underline{i}}$ to
$M_{s_{k_{l}}\cdot s_{k_1}\cdot\underline{i}}$ by
Lemma~\ref{lem:newoperator} and hence $\widehat{\Phi}_{\tau}$ is
bijective. By~\eqref{braidprime},
$\widehat{\Phi}_{\tau}$ does not depend on the choice of the
expression $s_{k_t}\cdots s_{k_1}$ for $\tau$.
Using~\eqref{XC-interprime}, we obtain
$X_k^\pm\widehat{\Phi}_{\tau}=\widehat{\Phi}_{\tau}X^\pm_{\tau(k)}$ and
$C_k\widehat{\Phi}_{\tau}=\widehat{\Phi}_{\tau}C_{\tau(k)}$ for
$1\leq k\leq n$.
\end{proof}

Suppose $\underline i\in \mathfrak{P}'(\aHn).$ Recall the definition of $L(i)^\tau$ from Remark \ref{rem:Ltau} for $\tau\in P_{\underline i}.$ Denote by $D^{\underline i}$ the $\mathcal{A}_
n$-module defined by
\begin{equation}\label{Di}
D^{\underline i} = \oplus_{\tau\in P_{\underline i}}L(i)^\tau.
\end{equation}
Then we have the following classification of irreducible completely splittable $\aHn$-modules in the category $\op{Rep}_{\mathbb I} \aHn$  which is parallel to the classification in degenerate case which was provided by \cite[Theorem 4.5]{Wa}. In addition, the proof follows a similar argument, which has also been extended to the semisimple setting  of cyclotomic Hecke-Clifford superalgebras in \cite{SW}. We will therefore only outline the proof, omitting detailed calculations. 
\begin{thm}\label{thm:Classficiation}
Suppose $\un i,\un j \in \mathfrak{P}'(\aHn).$ Then
\begin{itemize}
\item[(1)] $D^{\un i}$ affords an irreducible $\aHn$-module via
\begin{equation}\label{eq:actionformula}
T_kz^\tau=\begin{cases}
\Xi_kz^\tau+\Omega_kz^{s_k\tau},&\text{if }s_k\text{ is admissible with respect to }\tau\cdot \un i,\\
\Xi_kz^\tau,&\text{otherwise,}\end{cases}
\end{equation}
for $1\leq k \leq n-1,z \in L(\un i)$ and $\tau\in P_{\un i}.$ In addition, $D^{\un i}$ has the same type as the irreducible $\mathcal{A}_n$-module $L(\un i)$.
\item[(2)]
$D^{\un i} \cong D^{\un j}$ if and only if $\un i\sim \un j.$
\item[(3)]
Every irreducible completely splittable $\aHn$-module in $\op{Rep}_{\mathbb I} \aHn$ is isomorphic to $D^{\un i}$ for some $\un i\in \mathfrak{P}'(\aHn).$ 
\end{itemize}
\end{thm}

\begin{proof}

(1) To show the formula~\eqref{eq:actionformula} defines a
$\aHn$-module structure on $D^{\underline{i}}$, we need to check
the defining relations~\eqref{TT},~\eqref{TX1},~\eqref{TX2}
and~\eqref{TC} on $L(\underline{i})^{\tau}$ for each $\tau\in
P_{\underline{i}}$. One can show using~\eqref{XC} that
\begin{align}
(\Xi_kX_k-X_{k+1}\Xi_k)z=-\varepsilon(X_{i+1}+C_iC_{i+1}X_i)z.\label{Deltax}
\end{align}
for any $z\in L(\un i)$ and $\tau\in P_{\un i}$. 
For $1\leq k\leq n-1$,
$(X^{\pm 1}_{\tau^{-1}(k)}z)^{s_k\tau}=X^{\pm 1}_{k+1}z^{s_k\tau}$ by
 Remark~\ref{rem:Ltau} and hence if
$s_k$ is admissible with respect to $ \tau\cdot\underline{i}$,
then by \eqref{Deltax} we have 
\begin{align}
T_kX_kz^{\tau}=T_k(X_{\tau^{-1}(k)}z)^{\tau}&=
\Xi_k(X_{\tau^{-1}(k)}z)^{\tau}+\Omega_k(X_{\tau^{-1}(k)}z)^{s_k\tau}\notag\\
&=\Xi_kX_kz^{\tau}+X_{k+1}\Omega_kz^{s_k\tau}\notag\\
&=(\Xi_kX_k-X_{k+1}\Xi_k)z^{\tau}+X_{k+1}(\Xi_kz^{\tau}+\Omega_kz^{s_k\tau})\notag\\
&=-\varepsilon(X_{k+1}+C_kC_{k+1}X_i)z^{\tau}+X_{k+1}T_kz^{\tau}\notag.
\end{align}
Otherwise,  we have
\begin{align}
T_kX_kz^{\tau}=T_k(X_{\tau^{-1}(k)}z)^{\tau}
&=\Xi_k(X_kz^{\tau})\notag\\
&=(\Xi_kX_k-X_{k+1}\Xi_k)z^{\tau}+X_{k+1}\Xi_kz^{\tau}\notag\\
&=-\varepsilon(X_{k+1}+C_kC_{k+1}X_i)z^{\tau}+X_{k+1}T_kz^{\tau}. \notag
\end{align}
Therefore~\eqref{TX1} holds acting on $L(\un i)^\tau$. It is routine to check \eqref{TX2}
and~\eqref{TC}.

It remains to prove~\eqref{TT} acting on $L(\un i)^\tau$. It is clear by~\eqref{TX2} that
$T_kT_l=T_lT_k$ if $|l-k|>1$, so it suffices to prove $T_k^2=\varepsilon T+1$
and $T_kT_{k+1}T_k=T_{k+1}T_kT_{k+1}$. For the remaining of the
proof, let us fix $\tau\in P_{\underline{i}}$ and set
$\underline{j}=\tau\cdot\underline{i}$. 
If $s_k$ is admissible with respect to
$\underline{j}=\tau\cdot\underline{i}$, then one can check
using~\eqref{XC} and ~\eqref{eq:actionformula} that
\[\begin{aligned}
T_k^2z^{\tau}=&\left(\Xi_k^2
 +\Omega_k^2\right)z^{\tau}+\Omega_k\left(\left(\Xi_k z^{\tau}\right)^{s_k}+\Xi_kz^{s_k\tau}\right)\\
 =&\varepsilon \Xi_kz^\tau+z^\tau+\varepsilon \Omega_kz^{s_k\tau}\\
=&(\varepsilon T_k+1)z^{\tau}.
\end{aligned}\]
Otherwise we have $j_{k}=j_{k+1}\pm1$. Then as in \eqref{SqXi} we have
\[
T_k^2z^{\tau}=\Xi_k^2z^{\tau}=(\varepsilon\Xi_k+1)=(\varepsilon T_k+1)z^{\tau}.
\]
Therefore $T_k^2=(\varepsilon T_k+1)$ on $D^{\underline{i}}$ for $1\leq k\leq n-1$ .
Next we shall prove $T_kT_{k+1}T_k=T_{k+1}T_kT_{k+1}$ for $1\leq
k\leq n-2$. Set $\widehat{T}_k=T_k-\Xi_{k}$ for $1\leq k\leq n-1$.
It is clear by~\eqref{eq:actionformula} that
\begin{eqnarray*}
\widehat{T}_kz^{\tau}= \left \{
 \begin{array}{ll}
 \Omega_kz^{s_k\tau},
 & \text{ if } s_k \text{ is admissible with respect to } \underline{j}=\tau\cdot\underline{i}, \\
 0, & \text{ otherwise }.
 \end{array}
 \right.
\end{eqnarray*}
If there exist $a\neq b\in\{k,k+1,k+2\}$ such $j_a-j_b=\pm 1$, then
$\widehat{T}_k\widehat{T}_{k+1}\widehat{T}_k
=0=\widehat{T}_{k+1}\widehat{T}_k\widehat{T}_{k+1}$ on
$L(\underline{i})^{\tau}$; otherwise, one can show
using \eqref{Omega} that
\[\begin{aligned}
\widehat{T}_k\widehat{T}_{k+1}\widehat{T}_kz^{\tau}
=&\sqrt{1-\frac{\varepsilon^2(ab-4)}{(a-b)^2}}\sqrt{1-\frac{\varepsilon^2(ac-4)}{(a-c)^2}}\sqrt{1-\frac{\varepsilon^2(bc-4)}{(b-c)^2}}z^{s_ks_{k+1}s_k\tau}\\
=&\widehat{T}_{k+1}\widehat{T}_k\widehat{T}_{k+1}z^{\tau}, 
\end{aligned}\]
for
any  $z\in L(\underline{i})$ and $\tau\in P_{\un i}$, where $a=\mathtt{q}(j_k), b=\mathtt{q}(j_{k+1}),
c=\mathtt{q}(j_{k+2})$. Hence putting together we get
\begin{align}
\widehat{T}_k\widehat{T}_{k+1}\widehat{T}_kz^{\tau}
=\widehat{T}_{k+1}\widehat{T}_k\widehat{T}_{k+1}z^{\tau}, \text{
for any }z\in L(\underline{i}), 1\leq k\leq n-2.\label{braid'}
\end{align}
Recalling the intertwining elements
$\widetilde\Phi_k$ from~(\ref{intertwinNon-dege}), we see that
\begin{align}
\widetilde\Phi_kz^\tau=\widehat{T}_k \mathsf{z}^2_k z^\tau.\notag
\end{align}
and hence one can obtain 
\begin{equation}
\begin{aligned}
\widetilde\Phi_k\widetilde\Phi_{k+1}\widetilde\Phi_kz^\tau=(a-b)^2(a-c)^2(b-c)^2\widehat{T}_k\widehat{T}_{k+1}\widehat{T}_kz^{\tau}, \\
\widetilde\Phi_{k+1}\widetilde\Phi_{k}\widetilde\Phi_{k+1}z^\tau=(a-b)^2(a-c)^2(b-c)^2\widehat{T}_{k+1}\widehat{T}_{k}\widehat{T}_{k+1}z^{\tau}
\end{aligned}
\end{equation}
Hence by~(\ref{braid'}) we see that for any $z\in
L(\underline{i})$ and $\tau\in P_{\un i}$, 
$$
(\widetilde\Phi_k\widetilde\Phi_{k+1}\widetilde\Phi_k-\widetilde\Phi_{k+1}\widetilde\Phi_k\widetilde\Phi_{k+1})z^\tau=0. 
$$
Meanwhile a tedious calculation shows that
$$
(T_kT_{k+1}T_k-T_{k+1}T_kT_{k+1})(\mathsf{z}^2_{k,k+1}\mathsf{z}^2_{k,k+2}\mathsf{z}^2_{k+1,k+2})z^\tau=(\widetilde\Phi_k\widetilde\Phi_{k+1}\widetilde\Phi_k-\widetilde\Phi_{k+1}\widetilde\Phi_k\widetilde\Phi_{k+1})z^\tau, 
$$
where we set 
$\mathsf{z}_{k,l}=(X_k+X_k^{-1})-(X_l+X_l^{-1})$ for $1\leq k,l\leq n$. Thus we have 
\begin{equation}\label{eq:braid-verify-1}
(T_kT_{k+1}T_k-T_{k+1}T_kT_{k+1})(\mathsf{z}^2_{k,k+1}\mathsf{z}^2_{k,k+2}\mathsf{z}^2_{k+1,k+2})z^\tau=0
\end{equation}
for any $z\in L(\un i),\tau\in P_{\un i}$. Recall that we set $\un j=\tau\cdot\un i, a=\mathtt{q}(j_k), b=\mathtt{q}(j_{k+1}), c=\mathtt{q}(j_{k+2})$. Since $\tau\in P_{\un i}$, the tuple 
$\un j$ belongs to $ \mathfrak{P}'(\aHn)$ and hence $a\neq b, b\neq c$.  If in addition $a\neq c$, then 
$$
\mathsf{z}^{2}_{k,k+1}\mathsf{z}^{2}_{k,k+2}\mathsf{z}^{2}_{k+1,k+2}z^\tau=(a-b)^2(a-c)^2(b-c)^2z^\tau. 
$$ 
This together with \eqref{eq:braid-verify-1} leads to 
$$
T_kT_{k+1}T_kz^{\tau}=T_{k+1}T_kT_{k+1}z^{\tau},\quad \text{ for
any }z\in L(\underline{i}), \tau\in P_{\un i}.
$$

Now assume $a=c$ or equivalently $j_k=j_{k+2}$, then by Lemma~\ref{lem:restr.1} we have
either $j_{k}=j_{k+2}=0,j_{k+1}=1$ or $j_{k}=j_{k+2}=\frac{h-3}{2}\text{(resp. }\frac{h}{2}-1),
j_{k+1}=\frac{h-1}{2}\text{(resp. }\frac{h}{2}-2)$. Hence $T_{k}=\Xi_k$ and $T_{k+1}=\Xi_{k+1}$
on $L(\underline{i})^{\tau}$. We see from the proof of
Lemma~\ref{lem:restr.1} that $T_kT_{k+1}T_kz^\tau=T_{k+1}T_kT_{k+1}z^\tau$.
Therefore $D^{\underline{i}}$ affords a $\aHn$-module by the
formula~\eqref{eq:actionformula}.

Suppose $N$ is a nonzero irreducible submodule of
$D^{\underline{i}}$, then $N_{\underline{j}}\neq 0$ for some
$\underline{j}\in\I^n$. This implies
$(D^{\underline{i}})_{\underline{j}}\neq 0$ and hence
$\underline{j}\sim\underline{i}$. Since
$\tau\cdot\underline{i}\sim\underline{i}\sim\underline{j}$, by
Remark~\ref{rem:admissible} we see that
$N_{\tau\cdot\underline{i}}\neq 0$ for all $\tau\in
P_{\underline{i}}$. Observe that
$(D^{\underline{i}})_{\tau\cdot\underline{i}}\cong
L(\tau\cdot\underline{i})$ is irreducible as a $\mathcal{A}_n$-module for
$\tau\in P_{\underline{i}}$. Therefore
$N_{\tau\cdot\underline{i}}=(D^{\underline{i}})_{\tau\cdot\underline{i}}$
for $\tau\in P_{\underline{i}}$ and hence $N=D^{\underline{i}}$.
This means $D^{\underline{i}}$ is irreducible. 
The proof of the remaining argument in (1) and (2) is similar to that of \cite[Theorem 4.5]{Wa} (cf. \cite[Proposition 4.8(1)(2)]{SW}) and thus we omit the details.

Next we shall prove (3).  Suppose $M\in\operatorname{Rep}_{\I}\aHn$ is irreducible
completely splittable with $M_{\underline{i}}\neq 0$ for some
$\underline{i}\in\I^n$. By
Proposition~\ref{prop:restr.2} we have $\mathfrak{P}(\aHn)\subseteq
\mathfrak{P}'(\aHn)$ and hence $\underline{i}\in \mathfrak{P}'(\aHn)$.  By Proposition~\ref{prop:equiv.cond.},
there exists a $\mathcal{A}_n$-isomorphism $\psi:
L(\underline{i}) \rightarrow M_{\underline{i}}$. By
Lemma~\ref{lem:newbij.}, for each $\tau\in P_{\underline{i}}$, there
exists a bijection $\widehat{\Phi}_{\tau}:
M_{\underline{i}}\rightarrow M_{\tau\cdot\underline{i}}$.  Now for
$\tau\in P_{\underline{i}}$, define
$$
\psi^{\tau}: L(\underline{i})^{\tau}\longrightarrow
M_{\tau\cdot\underline{i}},\quad z^{\tau}\mapsto
\widehat{\Phi}_{\tau}(\psi(z)).
$$
By Lemma~\ref{lem:newbij.}, the bijection $\widehat{\Phi}_{\tau}$
satisfies $
\widehat{\Phi}_{\tau}X^\pm_k=X^\pm_{\tau(k)}\widehat{\Phi}_{\tau},
\widehat{\Phi}_{\tau}C_k=C_{\tau(k)}\widehat{\Phi}_{\tau}$ for
$1\leq k\leq n$. Hence for $z\in L(\underline{i}), \tau\in
P_{\underline{i}}$ and $1\leq k\leq n$,
\begin{align}\psi^{\tau}(X_kz^{\tau})&=\psi^{\tau}((X_{\tau^{-1}(k)}z)^{\tau})
=\widehat{\Phi}_{\tau}(\psi(X_{\tau^{-1}(k)}z))\notag\\
&=\widehat{\Phi}_{\tau}(X_{\tau^{-1}(k)})\psi(z)=X_k\widehat{\Phi}_{\tau}(\psi(z))
=X_k\psi^{\tau}(z^{\tau})\notag.
\end{align}
Similarly one can show that
$\psi^{\tau}(C_kz^{\tau})=C_k\psi^{\tau}(z^{\tau})$. Therefore
$\psi^{\tau}$ is a $\mathcal{A}_n$-homomorphism. By the
fact that $\psi^{\tau}$ is a $\mathcal{A}_n$-module homomorphism for each
$\tau\in P_{\underline{i}}$, one can easily check that
$$
\oplus_{\tau\in P_{\underline{i}}}\psi^{\tau}:
D^{\underline{i}}\longrightarrow M
$$
is a $\aHn$-module isomorphism.
\end{proof}

By Proposition~\ref{prop:restr.2} we have $\mathfrak{P}(\aHn)\subset \mathfrak{P}'(\aHn).$ By Theorem \ref{thm:Classficiation} we obtain the following
\begin{cor}\label{cor:coincide}
We have $\mathfrak{P}(\aHn)=\mathfrak{P}'(\aHn).$ That is, the set of weights of irreducible completely splittable $\aHn$-modules is exactly the set of $\un i \in\I^n$ satisfying the properties listed in Proposition~\ref{prop:restr.2}. 
\end{cor}

\begin{rem}
Any representation of a cyclotomic Hecke-Clifford superalgebra can be viewed as a representation of the affine Hecke-Clifford superalgebra \(\aHn\) via inflation, since the former are quotients of the latter. Indeed, under a certain parameter condition on \(q\), the second author and Shi constructed a family of irreducible representations for these cyclotomic algebras in \cite{SW}, indexed by the combinatorial notion of \emph{multipartitions}. One can easily verify that these irreducible representations are completely splittable when inflated to \(\aHn\). It is therefore natural to ask: is there a combinatorial interpretation of the weight set \(\mathfrak{P}(\aHn)\) and the equivalence classes \(\mathfrak{P}(\aHn)/{\sim}\) that parametrize the isomorphism classes of irreducible completely splittable \(\aHn\)-modules?

When \(h = \infty\) or \(h\) is odd, the diagrammatic interpretation in terms of \emph{placed skew shifted Young diagram} in \cite[Section 5]{Wa} also works in our situation, as the properties listed in Proposition~\ref{prop:restr.2} coincide with those in \cite[Proposition 3.14]{Wa}. However, it becomes subtle in the case where \(h\) is even for the affine Hecke-Clifford superalgebra \(\aHn\), while we are indeed able to provide a combinatorial classification for the finite Hecke-Clifford superalgebra \(\Hn\) in the next section.
\end{rem}
}

\section{Irreducible completely splittable module over the finite Hecke-Clifford superalgebra $\Hn$}\label{sec:cs-Hn}
In this section, we shall apply the classification of irreducible completely splittable $\aHn$-modules in previous section to the case of finite Hecke-Clifford superalgebra $\Hn$ and we accordingly provide a combinatorial interpretation in terms of partitions. 

\subsection{A surjective homomorphism}

{\color{black} Let $\Hn$ be the subalgebra of $\aHn$ generated by $T_1,\cdots,T_{n-1},C_1,\cdots,C_n$.  Then $\Hn$ is known as the finite Hecke-Clifford superalgebra and there are the notion of  the Jucy-Murphy elements defined as follows in $\Hn$: 
\begin{equation}\label{eq:JM}
L_1=1, \quad L_{k+1}=(T_k+\varepsilon C_kC_{k+1})L_kT_k,\quad \text{ for }1\leq k\leq n-1. 
\end{equation}
It is known that the elements $L_1,L_2,\ldots, L_n$ are invertible and commute. In addition there exists a surjective homomorphism 
\begin{equation}\label{eq:surj}
\aHn\longrightarrow \Hn, \quad T_k\mapsto T_k, C_j\mapsto C_j, X^{\pm 1}_j\mapsto L^{\pm 1}_j
\end{equation}
for $1\leq k\leq n-1, 1\leq j\leq n$. 

\begin{defn}
A finite dimensional $\Hn$-module $M$ is said to be completely splittable if $L_1,L_2,\ldots,L_n$ act on $M$ semisimply. 
\end{defn}

By~\cite[Lemma 4.4]{BK1} (cf.
\cite[Lemma 15.1.2]{K2}), a $\aHn$-module $M$ belongs to the
category $\text{Rep}_{\I}\aHn$ if all of eigenvalues of $X_j+X_j^{-1}$ on
$M$ are of the form $\mathtt{q}(i)$ with $i\in\mathbb{I}$ for some $1\leq j\leq n$. Hence the
category of finite dimensional completely splittable
$\Hn$-module can be identified with the subcategory of
$\text{Rep}_{\I}\aHn$ consisting of completely splittable
$\aHn$-modules on which $X_1=1$. By \eqref{eq:wt-decomp}, we can
decompose any finite dimensional $\Hn$-module $M$ as
$$
M=\oplus_{\underline{i}\in\I^n}M_{\underline{i}},
$$
where $M_{\underline{i}}=\{z\in M~|~((L_k+L_k^{-1})-\mathtt{q}(i_k))^Nz=0,\text{ for
}N\gg0, 1\leq k\leq n \}$. If $M_{\underline{i}}\neq 0$, then
$\underline{i}$ is called a {\em weight} of $M$.

\begin{defn} Define $\mathfrak{P}(\Hn)$ to be the set of weights
$\underline{i}=(i_1,\ldots,i_n)\in \mathfrak{P}(\aHn)$ satisfying the
following additional conditions:
\begin{equation}\label{wt}
i_1=0,\quad \{i_k-1,i_k+1\}\cap\{i_1,\cdots,i_{k-1}\}\neq \emptyset\text{ for }2\leq k\leq n
\end{equation}
\end{defn}
\begin{prop}\label{prop:Wfin}
$\mathfrak{P}(\Hn)$ is the set of weights occurring in irreducible completely splittable $\Hn$-modules. 
\end{prop}
\begin{proof} 
Suppose $\underline{i}\in \I^n$ occurs in some irreducible
completely splittable representation $M$ of $\Hn$.  Clearly by \eqref{eq:surj} $M$ can be viewed as an irreducible completely splittable $\aHn$-module and hence $\underline{i}\in \mathfrak{P}(\aHn)$. In addition,  since 
$X_1+X_1^{-1}=2$ acting on $M$ and we have $\mathtt{q}(i_1)=2$ which implies $i_1=0$ by \eqref{substitution0}. For $2\leq k\leq n$, if $i_k=0$, then by Proposition \ref{prop:restr.2}(iii) we have $1\in\{i_1,\ldots,i_{k-1}\}$ and hence $\{i_k-1,i_k+1\}\cap\{i_1,\ldots,i_{k-1}\}\neq\emptyset$. Now assume $i_k\geq 1$ and suppose
$\{i_k-1,i_k+1\}\cap\{i_1,\ldots,i_{k-1}\}=\emptyset$. Then $s_l$ is admissible with respect to $s_{l+1}\cdots s_{k-1}\cdot\underline{i}$ for $1\leq l\leq k-1$ and hence $M_{s_1\cdots
s_{k-1}\cdot\underline{i}}\neq 0$ by Corollary \ref{Cor-2}. Set $\underline{j}=s_1\cdots
s_{k-1}\cdot\underline{i}$. Note $j_1=i_k\neq 0$ and this
contradicts to the fact that $X_1=1$ on $M$. Thus $\{i_k-1,i_k+1\}\cap\{i_1,\ldots,i_{k-1}\}\neq\emptyset$. So we obtain $\underline{i}\in\mathfrak{P}(\Hn)$. 

Conversely, let $\underline{i}\in \mathfrak{P}(\Hn)$. Recall
$P_{\underline{i}}$ and $D^{\underline{i}}$ from \eqref{Pi}
and \eqref{Di}, respectively. It can be easily checked that
$\tau\cdot\underline{i}\in \mathfrak{P}(\Hn)$ for each $\tau\in
P_{\underline{i}}$ and hence $X_1=1$ on $D^{\underline{i}}$. This
implies that $D^{\underline{i}}$ affords an irreducible
completely splittable $\Hn$-module. The Proposition
follows from the fact that $\underline{i}$ is a weight of
$D^{\underline{i}}$.
\end{proof}

\subsection{Basics on partitions}

For a partition $\la=(\la_1,\la_2,\cdots )$, we always assume $\la_1\geq\la_2\geq\dots\geq 0$. Denote by $\ell(\la)$ the number of nonzero parts in $\la$. It is known that the partition $\la$ can be drawn as Young diagrams. Denote by $\mathcal{P}(n)$ the set of partitions of $n$. For $k\geq 0$, a partition $\la$ is said to be $k$-strict if $k$ divides $\la_r$ whenever $\la_r=\la_{r+1}$ for $r \geq 1$. Denote by $\mathcal{SP}_k(n)$ the set of $k$-strict partition of $n$. Then the set $\mathcal{SP}_0(n)$ is exactly the set of  usual strict partitions of $n$, that is, 
$$
\mathcal{SP}_0(n)= \{\la\in \mathcal{P}(n)\mid \la_i>\la_{i+1}, 1\leq i\leq \ell(\la)-1\}.
$$  Similar to  the case of partitions, a  strict partition $\la\in \mathcal{SP}_0(n)$ can be identified with the shifted Young diagram which is obtained from the ordinary Young diagram by shifting the $k$-th row to the right by $k-1$ squares for all $k>1$, that is, 
\begin{equation}
\la^{\mathsf{s}}=\{(i,j)\mid 1\leq i\leq\ell(\la), i\leq j\leq i+\la_i-1 \}. 
\end{equation} 
The $(i,j)$-hook of a shifted Young diagram contains all nodes that are either in the same row as $(i,j)$ and to the right of $(i,j)$, or in the same column as $(i,j)$ and below $(i,j)$ including $(i,j)$.  Additionally if $(j,j)$ is included in the hook then nodes in the $(j+1)$-row are also included. Denote by $h^{\mathsf{s}}_\la(i,j)$ the number of nodes in the $(i,j)$-hook.
\begin{example}
Suppose $\la=(7,5,3,2)$, then the $(1,2)$-hook and $(1,4)$-hook of shifted Young diagram $\la^{\mathsf{s}}$ are as below
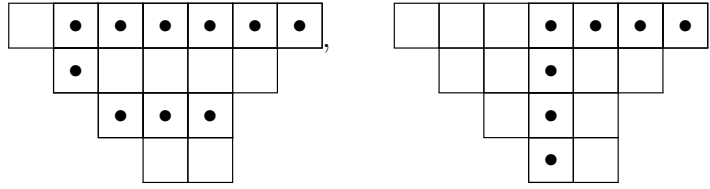
\[\ydiagram{7,1+5,2+3,3+2}*[\bullet]{1+6,1+1,2+3}, 
\qquad
\ydiagram{7,1+5,2+3,3+2}*[\bullet]{3+4,3+1,3+1,3+1},\]
and accordingly $h^{\mathsf{s}}_\la(1,2)=10$ and $h^{\mathsf{s}}_\la(1,4)=7$. 
\end{example}

Denote by $\mathcal{T}^{\mathsf{s}}(\la)$ the set of shifted tableaux of shape $\la^{\mathsf{s}}$; that is, a shifted tableau is a labelling of the nodes in the shifted Young diagram $\la^{\mathsf{s}}$ with the entries $1,2,\dots, n$. Let $T_{(i,j)}$ denote the entry in the node $(i,j)$ and let $T(k)$ be the node which is occupied by the number $k$ for each $1\leq k\leq n$. So if $T_{(i,j)}=k$ then $T(k)=(i,j)$. 
A shifted tableau $T$ is called \emph{standard} if its entries strictly increase from left to right along each row and down each column. We denote by $\op{Std}^{\mathsf{s}}(\la)$ the subset of $\mathcal{T}^{\mathsf{s}}(\la)$ consisting of standard tableaux of shape $\la^{\mathsf{s}}$. We then have the following remarkable hook length formula (cf. \cite[Chapter III, Section 8, Example 12]{Mac})
\begin{equation}\label{eq:hook2}
\sharp \op{Std}^{\mathsf{s}}(\la)=\frac{n!}{\prod_{(i,j)\in\la^{\mathsf{s}}} h^{\mathsf{s}}_\la(i,j)}
\end{equation}
where the product in the denominator is over all nodes in the shifted Young diagram $\la^{\mathsf{s}}$.
}

 {\color{black}\subsection{Irreducible completely splittable $\Hn$-modules in case $h \geq 3$ is odd}
By Proposition \ref{prop:Wfin}, $\mathfrak{P}(\Hn)$ is the set of weights occurring in irreducible completely splittable finite Hecke-Clifford algebra $\Hn$-modules. If $h\geq 3$ is prime, we observe that the conditions that $\mathfrak{P}(\Hn)$ satisfied, including Proposition \ref{prop:restr.2}(1)(2)(3)(5) and \eqref{wt}, are the same as the conditions that weights of degenerate affine Hecke-Clifford algebras over a filed of characteristic $p=h$ satisfied, which has been discussed by the second author in \cite{Wa}. Since the combinatorial construction in \cite{Wa} does not depend on whether $p$ is prime or not, we can use the same combinatorial objects to index the irreducible completely splittable representations of $\Hn$. Here we give a quick review on the details (cf. \cite{CWZ}).
Denote by 
\begin{equation}\label{eq:cpn}
\begin{aligned}
\mathcal{CSP}_h(n)=\Big\{\xi\in\mathcal{SP}_0(n)\mid &\xi_1=h-u,\xi_2\leq u\text{ for some }1\leq u\leq \frac{h-3}{2}\\
&\text{or } 1\leq \xi_1\leq \frac{h+1}{2}\Big\}
\end{aligned}
\end{equation}
For each $\xi\in\mathcal{CSP}_h(n)$, denote by
\begin{equation}
\op{Std}^{\mathsf{s}}_h(\xi)=\begin{cases}
\big\{T\in\op{Std}^{\mathsf{s}}(\xi)|T_{(2,\xi_2+1)}>T_{(1,\xi_1)}\big\}, &\text{if }\xi_1=h-u,\xi_2=u\\
&\text{ for some }1\leq u\leq \frac{h-3}{2},\\
\op{Std}^{\mathsf{s}}(\xi),&\text{otherwise. }
\end{cases}
\end{equation}
 Set
\[\Delta_{h}(n):=\{(\xi,T)|\xi\in\mathcal{CSP}_h(n),T\in\op{Std}^{\mathsf{s}}_h(\xi)\}\]

We label the {\em residue} of nodes in the shifted Young diagram of $\xi\in\mathcal{CSP}_h(n)$ 
using the set $\mathbb{I}=\{0,1,\ldots,\frac{h-1}{2}\}$ in \eqref{defn:I} via  the way that the first node in each row has residue $0$ and then follow the repeating pattern
\begin{equation}\label{eq:residue-odd}
0, 1,\ldots,\frac{h-3}{2},\frac{h-1}{2},\frac{h-3}{2},\ldots,1,0. 
\end{equation}
The residue  in \eqref{eq:residue-odd} is actually to compute
$\mathtt{q}$-values of the usual residue $j-i$ of the nodes $(i,j)$ and the reason for this pattern is due to the observation $\mathtt{q}(a)=\mathtt{q}(b)$ if $a=b\mod h$ or $a+b+1=0\mod h$ for any $a,b\in\mathbb{Z}$. Let $\xi\in\mathcal{CSP}_h(n)$ and suppose $T\in \op{Std}^{\mathsf{s}}_h(\xi)$. Let 
$$
\underline{i}_{(\xi, T)}=(\op{res}(T(1)),\op{res}(T(2)),\cdots,\op{res}(T(n)))\in\mathbb{I}^n
$$ be the residue sequence corresponding to the pair $(\xi,T)$. 

\begin{lem}\cite[Lemma 6.6 and Lemma 6.7]{Wa}\label{lem:odd}
For each $(\xi,T)\in \Delta_h(n)$, the residue sequence $\underline{i}_{(\xi,T)}$ belongs to $\mathfrak{P}(\Hn)$ and this gives rise to a bijection from $\Delta_{h}(n)$ to $\mathfrak{P}(\Hn)$. Moreover, assume $\un j\in \mathfrak{P}(\Hn)$, then $ \un j\sim \underline{i}_{(\xi,T)}$ for some  $(\xi,T)\in\Delta_h(n)$ if and only if $\un j=\un i_{(\xi,S)}$ for some $S\in \op{Std}^{\mathsf{s}}_h(\xi)$.
\end{lem}

Assume $\xi\in\mathcal{CSP}_h(n)$ and $T\in \op{Std}^{\mathsf{s}}_h(\xi)$. Recall the module $D^{\un i}$ for each $\un i\in \mathfrak{P}(\aHn)$ define in \eqref{Di} and set
\[D(\xi):=D^{\un i_{(\xi,T)}}\]
for some $T\in\op{Std}^{\mathsf{s}}_h(\xi)$. Then $D(\xi)$ is independent of the choice of $T \in\op{Std}^{\mathsf{s}}_h(\xi)$  by Lemma  \ref{lem:odd}. 

\begin{thm}\label{dim2}
The set $\{D(\xi)\mid \xi\in\mathcal{CSP}_h(n)\}$ is a complete set of pairwise non-isomorphic irreducible completely splittable $\Hn$-modules. Moreover, $D(\xi)$ is type
 $\texttt{M}$ if $\ell(\xi)$ is even and is type
$\texttt{Q}$ if $\ell(\xi)$ is odd. In addition,
\[
\op{dim}D(\xi)=2^{n-\big\lfloor\frac{\ell(\xi)}{2}\big\rfloor}\sharp\op{Std}^{\mathsf{s}}_h(\xi).
\]
\end{thm}
\begin{proof}
 Assume $\xi\in\mathcal{CSP}_h(n)$. By Lemma  \ref{lem:odd} and Theorem \ref{thm:Classficiation}, every weight $\un i$ of  $D(\xi)$ satisfies $i_1=0$ which implies the $X_1+X_1^{-1}=2$ on $D(\xi)$. Then $X_1=1$ on $D^{\un i}$ as it is completely splittable. Therefore  $D(\xi)$ admits a $\Hn$-module by \eqref{eq:surj}.  Conversely, suppose $M$ is an irreducible completely splittable $\Hn$-module and $\un i$ is a weight of $M$. 
 Observe that  $M$ actually admits an irreducible completely splittable $\aHn$-module via \eqref{eq:surj} and hence $M\cong D^{\un i}$. Meanwhile by Lemma \ref{lem:odd} there exists $(\xi,T)\in\Delta_h(n)$ such that $\un i=\un i_{(\xi,T)}$. This means $M\cong D^{\un i}\cong D(\xi)$. The remaining statement in  the theorem follows from Theorem \ref{thm:Classficiation} and Lemma \ref{lem:odd} as well as \eqref{eq:qi=2} and Corollary \ref{cor:irrepAn}. 
\end{proof}
}
\subsection{Irreducible completely splittable $\Hn$-modules in case $h \geq 4$ is even}

{\color{black} Now, we consider the two kinds of weights described in Proposition \ref{prop:restr.2} (6a-b). Set
\[\begin{aligned}
\mathfrak{P}^1(\aHn)&=\left\{\un i\in \mathfrak{P}(\aHn)\mid \sharp\{1\leq k\leq n\mid i_k=\frac{h}{2}-1\}\leq 1\right\},\\
\mathfrak{P}^2(\aHn)&=\left\{\un i\in \mathfrak{P}(\aHn)\mid \sharp\{1\leq k\leq n\mid i_k=\frac{h}{2}-1\}\geq 2\right\}. 
\end{aligned}\]
Clearly we have 
$$
\mathfrak{P}(\aHn)=\mathfrak{P}^1(\aHn)\cup \mathfrak{P}^2(\aHn)
$$
and moreover if $\un i \in \mathfrak{P}^k(\aHn)$ and $\un j\sim\un i,$ then $\un j\in \mathfrak{P}^k(\aHn)$ for $k = 1, 2.$ Denote by $\mathfrak{P}^k(\Hn)$ the set of $\un i=(i_1,\cdots,i_n)\in \mathfrak{P}^k(\aHn)$ satisfying \eqref{wt} for $k=1,2$. Then, $\mathfrak{P}(\Hn)=\mathfrak{P}^1(\Hn)\sqcup \mathfrak{P}^2(\Hn)$ is the set of weights occurring in irreducible completely splittable $\Hn$-modules by Proposition \ref{prop:Wfin}.
Denote by $$\mathcal{CSP}_h^1(n)=\{\la\in \mathcal{SP}_0(n)\mid \la_1\leq \frac{h}{2}\}$$
and set 
\[\Delta^1_h(n)=\{(\xi,T)|\xi\in \mathcal{CSP}^1_h(n),T\in\op{Std}^{\mathsf{s}}(\xi)\}.\]
We label the residue of nodes in the shifted Young diagram of $\xi\in\mathcal{CSP}^1_h(n)$ 
using the set $\mathbb{I}=\{0,1,\ldots,\frac{h}{2}-1\}$ in \eqref{defn:I} via  the way that the first node in each row has residue $0$ and then follow the repeating pattern
\begin{equation}\label{eq:residue-even}
0, 1,\ldots,\frac{h}{2}-2,\frac{h}{2}-1,\frac{h}{2}-1,\frac{h}{2}-2\ldots,1,0. 
\end{equation}
Again, similar to the case $h$ is odd, the residue  in \eqref{eq:residue-even} is actually to compute
$\mathtt{q}$-values of the usual residue $j-i$ of the nodes $(i,j)$ and the reason for this pattern is due to the observation $\mathtt{q}(a)=\mathtt{q}(b)$ if $a=b\mod h$ or $a+b+1=0\mod h$ for any $a,b\in\mathbb{Z}$. Let $\xi\in\mathcal{CSP}^1_h(n)$ and suppose $T\in \op{Std}^{\mathsf{s}}(\xi)$. Let 
$$
\underline{i}_{(\xi, T)}=(\op{res}(T(1)),\op{res}(T(2)),\cdots,\op{res}(T(n)))\in\mathbb{I}^n
$$ be the residue sequence corresponding to the pair $(\xi,T)$.


\begin{lem}\label{Weven1}
The following map 
\begin{equation}
\mathcal{F}: \Delta^1_h(n)\longrightarrow \mathfrak{P}^1(\Hn), \quad (\xi, T)\mapsto \un i_{(\xi,T)}
\end{equation} 
is a bijection. 
 Moreover,  assume $\un j\in \mathfrak{P}^1(\Hn)$, then $ \un j\sim \un i_{(\xi,T)}$ for some $(\xi, T)\in\Delta_h^1(n)$ if and only if $\un j=\un i_{(\xi,S)}$ for some $S\in \op{Std}^{\mathsf{s}}(\xi)$. Thus there exists a surjective map $\Phi_1: \mathfrak{P}^1(\Hn)\longrightarrow \mathcal{CSP}_h^1(n)$ such that $\un i\sim \un j$ if and only if $\Phi_1(\un i)=\Phi_1(\un j)$. 
\end{lem}
\begin{proof}
By Corollary \ref{cor:coincide}, Proposition \ref{prop:restr.2} (6a-b) and \eqref{wt}, we observe that $\mathfrak{P}^1(\Hn)$ consists of the vectors $\un i\in\mathbb{I}^n$ satisfying: (1) $i_1=0$ and $  \{i_k-1,i_k+1\}\cap\{i_1,\cdots,i_{k-1}\}\neq \emptyset\text{ for }2\leq k\leq n$; (2) if $i_a=i_b=0$ for $1\leq a<b\leq n$ then $1\in \{i_{a+1},i_{a+2},\ldots,i_{b-1}\}$;  (3) if $i_a=i_b=u\neq 0$ for $1\leq a<b\leq n$ then $1\leq u\leq\frac{h}{2}-2$ and moreover $\{u-1,u+1\}\subseteq\{i_{a+1},i_{a+2},\ldots,i_{b-1}\}$; (4) $\sharp\{a\mid i_a=\frac{h}{2}-1, 1\leq a\leq n\}\leq 1$. Then one can show that $\mathcal{F}$ is a well-defined bijective map and moreover the second statement of proposition holds by applying argument same as the proof of \cite[Lemma 5.3, Proposition 5.5, Lemma 5.6]{Wa}. We omit the details as it is straightforward. This proves the proposition. 
\end{proof}

In the following, we shall explore the set of weights $\mathfrak{P}^2(\aHn)$ which is remarkably different from the case of  $\mathfrak{P}^1(\aHn)$. 
Set 
\begin{align}
\underline{\theta}^{(b,m)}_h=&\left\{
\begin{array}{cc}
(\frac{h}{2}-1,\frac{h}{2}-2,\ldots,1,0),\text{ if } b\leq m-\frac{h}{2}+1, \\
(\frac{h}{2}-1,\frac{h}{2}-2,\ldots,\frac{h}{2}-1-m+b),\text{ if } b>m-\frac{h}{2}+1. 
\end{array}
\right.
\end{align}
for $1\leq b\leq m$ and $m\geq 2$. For each $\un i\in \mathfrak{P}^2(\aHn)$, denote by $m_{\un i}=\#\{a\mid i_a=\frac{h}{2}-1, 1\leq a\leq n\}$, i.e. the number of parts of $\un i$ which are equal to $\frac{h}{2}-1.$ Obviously $m_{\un i}\geq 2$ for each $\un i\in \mathfrak{P}^2(\aHn)$. 
\begin{lem}\label{lem:weight2}
Let $\un i \in \mathfrak{P}^2(\aHn)$ and write $m=m_{\un i}$. Then 
$$
\un i\sim (\un i^{(\mathsf{L})}, \underline{\theta}^{(1,m)}_h, \underline{\theta}^{(2,m)}_h,\ldots, \underline{\theta}^{(m,m)}_h,\un i^{(\mathsf{R})})
$$
for some $\un i^{(\mathsf{L})}\in\mathbb{I}^l,\un i^{(\mathsf{R})}\in\mathbb{I}^r$ with $l,r\geq 0$. Moreover $0\leq i^{(\mathsf{L})}_a,  i^{(\mathsf{R})}_b\leq\frac{h}{2}-2$ for $1\leq a\leq l$ and $1\leq b\leq r$. In addition, if $r\geq 1$, then $ i^{(\mathsf{R})}_1=\max\{0, \frac{h}{2}-m-1\}$. 
\end{lem}
\begin{proof}
By Proposition \ref{prop:restr.2}(6b)(iii), we have that $\un i$ will be of the form 
$$
\un i=(\un i^{(0)}, \frac{h}{2}-1, \un i^{(1)}, \frac{h}{2}-2, \un i^{(2)}, \frac{h}{2}-1, \un i^{(3)})
$$
for some $\un i^{(0)}, \un i^{(1)},\un i^{(2)},\un i^{(3)}$ and moreover no $\frac{h}{2}-1, \frac{h}{2}-2$ in $\un i^{(1)},\un i^{(2)}$ and no $\frac{h}{2}-1$ in $\un i^{(3)}$. 
 Then by Corollary \ref{Cor-2}, Definition \ref{defn:admissible} and Corollary \ref{cor:coincide} we can use admissible transpositions to consecutively swap the $\frac{h}{2}-1$ between $\un i^{(0)}$ and $\un i^{(1)}$ with all entries in $\un i^{(1)}$ and also consecutively swap the $\frac{h}{2}-1$ between $\un i^{(2)}$ and $\un i^{(3)}$ with all entries  in $\un i^{(2)}$  to get
$$
\un i\sim (\un i^{(0)}, \un i^{(1)}, \frac{h}{2}-1,\frac{h}{2}-2,\frac{h}{2}-1,\un i^{(2)},\un i^{(3)})=:\un i'. 
$$
As $\un i'\in \mathfrak{P}^2(\aHn)$, apply the proof of Proposition \ref{prop:restr.2}(6b)(iii) to $\un i'$ one can obtain that $\un i'$ must be of the form 
$$
\un i'=(\un u^{(0)}, \frac{h}{2}-1,\un u^{(1)}, \frac{h}{2}-2, \un u^{(2)}, \frac{h}{2}-3,\un u^{(3)}, \frac{h}{2}-1,\frac{h}{2}-2,\frac{h}{2}-1,\un i^{(2)},\un i^{(3)})
$$
for some $\un u^{(0)},\un u^{(1)},\un u^{(2)},\un u^{(3)}$ and moreover no $\frac{h}{2}-1, \frac{h}{2}-2$ in $\un u^{(1)}$ and no $\frac{h}{2}-1, \frac{h}{2}-2,\frac{h}{2}-3$ in $\un u^{(2)}, \un u^{(3)}$.  
Then again by Corollary \ref{Cor-2}, Definition \ref{defn:admissible} and Corollary \ref{cor:coincide} we can use admissible transpositions to consecutively swap the entries in $\un u^{(1)}$ with $\frac{h}{2}-1$ on the left hand side and then swap the entries in $\un u^{(2)}$ with $\frac{h}{2}-1, \frac{h}{2}-2$ on the left hand side as well as swapping the entries in $\un u^{(3)}$ with $ \frac{h}{2}-1,\frac{h}{2}-2,\frac{h}{2}-1$ on the right hand side to get 
$$
\un i'\sim (\un u^{(0)}, \un u^{(1)}, \un u^{(2)}, \frac{h}{2}-1,\frac{h}{2}-2,\frac{h}{2}-3,\frac{h}{2}-1,\frac{h}{2}-2,\frac{h}{2}-1,\un u^{(3)}, \un i^{(2)},\un i^{(3)}). 
$$
Continue in this way, we eventually obtain that 
\begin{equation}\label{eq:i-sim-1}
\un i\sim (\un r^{(0)}, \underline{\theta}^{(1,m)}_h, \underline{\theta}^{(2,m)}_h,\ldots, \underline{\theta}^{(m,m)}_h,\un r^{(1)}). 
\end{equation}
for some $\un r^{(0)}, \un r^{(1)}$ such that no $\frac{h}{2}-1$ appearing in $\un r^{(0)}, \un r^{(1)}$. Observe that if $m\leq\frac{h}{2}-1$, then we can consecutively swap $ r^{(1)}_1$ with entries in $\underline{\theta}^{(1,m)}_h, \underline{\theta}^{(2,m)}_h,\ldots, \underline{\theta}^{(m,m)}_h$ whenever $ r^{(1)}_1<\frac{h}{2}-m-1$ and continue doing this to eventually get 
$$
\un i\sim (\un i^{\mathsf{(L)}}, \underline{\theta}^{(1,m)}_h, \underline{\theta}^{(2,m)}_h,\ldots, \underline{\theta}^{(m,m)}_h,\un i^{\mathsf{(R)}})
$$
for some $\un i^{\mathsf{(L)}}, \un i^{\mathsf{(R)}}$ with either $\un i^{\mathsf{(R)}}=\emptyset$ or $ i^{\mathsf{(R)}}_1\geq \frac{h}{2}-m-1$. In addition, by Proposition \ref{prop:restr.2}(6a), it is easy to see that $ i^{\mathsf{(R)}}_1= \frac{h}{2}-m-1$. Otherwise  $m>\frac{h}{2}-1$, then by \eqref{eq:i-sim-1} and Proposition \ref{prop:restr.2}(6a) we obtain that $ i^{\mathsf{(R)}}_1=0$ if $\un i^{\mathsf{(R)}}\neq\emptyset$.  Putting together, the lemma is verified. 
\end{proof}

Write $\un{\rho}_h:=(0,1,\ldots,\frac{h}{2}-2,0,1,\ldots,\frac{h}{2}-3,\ldots,0,1,2,0,1,0)$ and 
for each $2\leq m\leq \frac{h}{2}-1$, we set 
\[
\un\delta_{h,m}=
(0,1,\ldots,\frac{h}{2}-2,0,1,\ldots,\frac{h}{2}-3,\ldots,0,1,\ldots,\frac{h}{2}-m-1). 
\]

\begin{lem}\label{lem:Wfin2}
Let $\un i \in\mathfrak{P}^2(\Hn)$ and write $m=m_{\un i}$. Keep the notation in Lemma \ref{lem:weight2}. Then 
$$
\un i\sim
\left\{\begin{array}{cc}
 (\un\delta_{h,m}, \un \theta_h^{(1,m)},\un \theta_h^{(2,m)},\ldots,\un\theta_h^{(m,m)}, \un{\widehat{i}}),&\text{ if }m\leq\frac{h}{2}-1,\\
 (\un \rho_h, \un \theta_h^{(1,m)}, \un\theta_h^{(2,m)},\ldots,\un\theta_h^{(m,m)}, \un{\widehat{i}}),&\text{ if }m> \frac{h}{2}-1
 \end{array}
 \right.
$$
for some $ \un{\widehat{i}}$ with all entries being less than $\frac{h}{2}-1$. 
\end{lem}
\begin{proof}
By Lemma \ref{lem:weight2}, we have  $\un i\sim (\un i^{\mathsf{(L)}}, \un \theta_h^{(1,m)}, \un\theta_h^{(2,m)},\ldots,\un\theta_h^{(m,m)},\un i^{\mathsf{(R)}})$ such that no $\frac{h}{2}-1$ in $\un i^{\mathsf{(L)}}, \un i^{\mathsf{(R)}}$. Since no $\frac{h}{2}-1$ in $\un i^{\mathsf{(L)}}$, by Proposition \ref{prop:restr.2} we obtain 
\begin{equation}\label{eq:uni-number}
\sharp\left\{1\leq k \leq \ell(\un i^{\mathsf{(L)}})\mid  i^{\mathsf{(L)}}_k=b\right\}\leq \frac{h}{2}-1-b
\end{equation}
for $0\leq b\leq\frac{h}{2}-1$. 
As the first entry in $\un\theta_h^{(1,m)}$ is $\frac{h}{2}-1$, we observe that $\frac{h}{2}-2$ appears exactly once in $\un i^{\mathsf{(L)}}$ by \eqref{wt} and Proposition \ref{prop:restr.2} since $\un i \in\mathfrak{P}^2(\Hn)$. That is, there exists a unique $a_1$ such that $ i^{\mathsf{(L)}}_{a_1}=\frac{h}{2}-2$ and for any $k\neq a_1$ we have $ i^{\mathsf{(L)}}_k\neq \frac{h}{2}-2$. Then again by \eqref{wt} and Proposition \ref{prop:restr.2},  the element  $\frac{h}{2}-3$ appears only once on the left hand side of $ i^{\mathsf{(L)}}_{a_1}=\frac{h}{2}-2$ in $\un i^{\mathsf{(L)}}$ which means there exists a unique $1\leq a_2<a_1$ such that $ i^{\mathsf{(L)}}_{a_2}=\frac{h}{2}-3$ and for any $1\leq k\neq a_2<a_1$ we have $ i^{\mathsf{L}}_k\neq \frac{h}{2}-3, \frac{h}{2}-2,\frac{h}{2}-1$. Hence  by Corollary \ref{Cor-2}, Definition \ref{defn:admissible} and Corollary \ref{cor:coincide} we can use admissible transpositions to consecutively swap  $\frac{h}{2}-2$ with $ i^{\mathsf{(L)}}_k$ for $a_2<k<a_1$ to obtain
$$
\un i\sim (\un i^{\mathsf L,0},  \frac{h}{2}-3,\frac{h}{2}-2, \un i^{\mathsf L,1}, \un\theta_h^{(1,m)}, \un\theta_h^{(2,m)},\ldots, \un\theta_h^{(m,m)}, \un i^{(\mathsf{R})})
$$
for some $\un i^{\mathsf L,0},\un i^{\mathsf L,1}$ and moreover no $\frac{h}{2}-3,\frac{h}{2}-2,\frac{h}{2}-1$ in $\un i^{\mathsf L,0}$ and no $\frac{h}{2}-2,\frac{h}{2}-1$ in $\un i^{\mathsf L,1}$.  
Applying  the same argument to $ \un i^{\mathsf L,0}$ one can eventually to get 
\begin{equation}\label{eq:uni-mid}
\un i\sim (0,1,\ldots, \frac{h}{2}-3,\frac{h}{2}-2, \un i^{\mathsf{L,2}}, \un\theta_h^{(1,m)},  \un\theta_h^{(2,m)},\ldots, \un\theta_h^{(m,m)}, \un i^{(\mathsf{R})})=:\un j
\end{equation}
for some $\un i^{\mathsf{L,2}}$ such that $ i^{\mathsf{L,2}}_k<\frac{h}{2}-2$ for all admissible $k$.  This together with the fact  $ \un\theta_h^{(1,m)}=(\frac{h}{2}-1, \frac{h}{2}-2,\ldots)$ obtain that 
 $\frac{h}{2}-3$ appears exactly once in $\un i^{\mathsf{L,2}}$, that is there exists $b$ such that $ i^{\mathsf{L,2}}_b=\frac{h}{2}-3$ and for all admissible $k\neq b$ we have $ i^{\mathsf{L,2}}_k< \frac{h}{2}-3$. 
 Observe that in sequence $\un j$ on the right hand side of \eqref{eq:uni-mid}, we have $j_{\frac{h}{2}-2}=\frac{h}{2}-3=j_{b+\frac{h}{2}-1}$ and hence by Proposition \ref{prop:restr.2} we obtain the $\frac{h}{2}-4$ must appears exactly once on the left hand side of $ i^{\mathsf{L,2}}_b=\frac{h}{2}-3$ in $\un i^{\mathsf{L,2}}$. Then one can apply the argument similar to the swapping process from $\un i$ to $\un j$ to eventually obtain 
 \begin{equation}\label{eq:uni-mid}
\un i\sim \un j\sim (0,1,\ldots, \frac{h}{2}-3,\frac{h}{2}-2, 0,1,\ldots, \frac{h}{2}-3, \un i^{\mathsf{L,3}}, \un\theta_h^{(1,m)},  \un\theta_h^{(2,m)},\ldots, \un\theta_h^{(m,m)}, \un i^{(\mathsf{R})})=:\un j'
\end{equation}
for some $\un i^{\mathsf{L,3}}$ such that $ i^{\mathsf{L,3}}_k<\frac{h}{2}-3$ for $1\leq k\leq\ell(\un i^{\mathsf{L,3}})$. 
Continue in this way,  in the case  $m< \frac{h}{2}-1$,  we eventually obtain 
$$
\un i\sim  ( \un\delta_{h,m},\un i^{\mathsf{L, m+1}},  \un\theta_h^{(1,m)},\un\theta_h^{(2,m)}\ldots,\un\theta_h^{(m,m)}, \un i^{\mathsf{(R)}})
$$
for some $\un i^{\mathsf{L, m+1}}$ such that $ i^{\mathsf{L,m+1}}_k<\frac{h}{2}-m-1$ for $1\leq k\leq\ell(\un i^{\mathsf{L,m+1}})$. Then by Corollary \ref{Cor-2}, Definition \ref{defn:admissible} and Corollary \ref{cor:coincide} we get 
$$
 ( \un\delta_{h,m},\un i^{\mathsf{L, m+1}}, \un\theta_h^{(1,m)},\un\theta_h^{(2,m)}\ldots,\un\theta_h^{(m,m)}, \un i^{(\mathsf{R})})\sim  ( \un\delta_{h,m},  \un\theta_h^{(1,m)},\un\theta_h^{(2,m)}\ldots,\un\theta_h^{(m,m)},\un i^{\mathsf{L, m+1}}, \un i^{(\mathsf{R})}). 
$$
Similarly, in the case $m\geq \frac{h}{2}-1$,  we eventually obtain 
$$
\un i\sim  ( \un\rho_{h},\un i^{\mathsf{L'}},\un\theta_h^{(1,m)},\un\theta_h^{(2,m)}\ldots,\un\theta_h^{(m,m)}, \un i^{(\mathsf{R})})
$$
for some $\un i^{\mathsf{L'}}$. Then by \eqref{eq:uni-number} we have $\un i^{\mathsf{L'}}=\emptyset$. Putting together, this proves the lemma. 
\end{proof}
\begin{cor}\label{cor:n-bound}
If $\mathfrak{P}^2(\Hn)\neq\emptyset$, then $n\geq\frac{mh}{2}$. 
\end{cor}
Set
\[\un \tau_h=(0,1,\cdots,\frac{h}{2}-1). \]
\begin{prop}\label{prop:uni-simple}
Let $\un i \in \mathfrak{P}^2(\Hn)$ and write $m=m_{\un i}$.
Then 
\[\un i\sim (\underbrace{\un \tau_h,\un\tau_h,\ldots,\un \tau_h}_m,\widehat{\un i}),\]
for some $\widehat{\un i}\in \mathbb{I}^a$ with $a=n-\frac{mh}{2}$.  Moreover $\widehat{\un i}$ satisfies that $0\leq \widehat{ i}_k\leq\frac{h}{2}-2$ for any $1\leq k\leq a$ as well as the properties in Proposition \ref{prop:restr.2} and \eqref{wt}. 
\end{prop}
\begin{proof}
First, suppose $2\leq m\leq \frac{h}{2}-1$. By Lemma \ref{lem:Wfin2}, we obtain 
$$
\un i\sim  (\un\delta_{h,m}, \un \theta_h^{(1,m)},\un \theta_h^{(2,m)},\ldots,\un\theta_h^{(m,m)}, \un{\widehat{i}}). 
$$
and we can write $\un\delta_{h,m}=(\un \delta^{(1)}_{h,m}, \un \delta^{(2)}_{h,m},\ldots, \un \delta^{(m)}_{h,m})$ with $\un \delta^{(b)}_{h,m}=(0,1,\ldots, \frac{h}{2}-1-b)$ for $1\leq b\leq m$. 
Then clearly by Corollary \ref{Cor-2}, Definition \ref{defn:admissible} and Corollary \ref{cor:coincide} we can apply admissible transpositions to swap the unique $\frac{h}{2}-1$ in $\un \theta_h^{(1,m)}$ with all entries in $\un \delta^{(2)}_{h,m},\ldots, \un \delta^{(m)}_{h,m}$, swap the unique $\frac{h}{2}-2$ in $\un \theta_h^{(1,m)}$ with all entries $\un \delta^{(3)}_{h,m},\ldots, \un \delta^{(m)}_{h,m}$, etc. Eventually we obtain 
$$
\un i\sim  (\un \tau_{h},\un \delta^{(1)}_{h,m},\ldots,\un \delta^{(m-1)}_{h,m}, \un \theta_h^{(2,m)},\ldots,\un\theta_h^{(m,m)}, \un{\widehat{i}}). 
$$
Continuing in this way we eventually obtain 
$$
\un i\sim  (\underbrace{\un \tau_{h},\un \tau_{h},\ldots,\un \tau_{h}}_m, \un{\widehat{i}}). 
$$
This proves the first statement of proposition in the case $2\leq m\leq\frac{h}{2}-1$.  Now assume $m>\frac{h}{2}-1$. Then similar to the above argument, in the first step we will get 
$$
\un i\sim  (\un \tau_{h},\un \rho^{(1)}_{h},\ldots,\un \rho^{(\frac{h}{2}-2)}_{h},0, \un \theta_h^{(2,m)},\ldots,\un\theta_h^{(m,m)}, \un{\widehat{i}}), 
$$
where $\un\rho^{(b)}_{h}=(0,1,\ldots,\frac{h}{2}-1-b)$ for $1\leq b\leq \frac{h}{2}-1$. 
Continuing in this way one can 
$$
\begin{aligned}
\un i\sim & \Big(\underbrace{\un \tau_{h},\ldots,\un\tau_h}_{m-\frac{h}{2}+1},\un\rho^{(1)}_{h} ,\un\rho^{(2)}_{h},\ldots,\un\rho^{(h-m-2)}_{h}, 0,1,0,2,1,0,\ldots, m-\frac{h}{2}, m-\frac{h}{2}-1,\ldots,1,0,\\ 
&\un \theta_h^{(m-\frac{h}{2}+2,m)},\ldots,\un\theta_h^{(m,m)}, \un{\widehat{i}}\Big)\\
\sim& \Big(\underbrace{\un \tau_{h},\ldots,\un\tau_h}_{m-\frac{h}{2}+1}, \un\rho^{(1)}_{h} ,\un\rho^{(2)}_{h},\ldots,\un\rho^{(h-m-2)}_{h}, \\
&\un\rho^{(h-m-1)}_{h}, \un\rho^{(h-m)}_{h},\ldots,\un\rho^{(\frac{h}{2}-2)}_{h},\un\rho^{(\frac{h}{2}-1)}_{h}, \un \theta_h^{(m-\frac{h}{2}+2,m)},\ldots,\un\theta_h^{(m,m)}, \un{\widehat{i}}\Big)
\end{aligned}
$$
in the the case $m-\frac{h}{2}+1\leq \frac{h}{2}-1$. Then observe that the subsequence $\Big(\un\rho^{(1)}_{h} ,\un\rho^{(2)}_{h},\ldots,\un\rho^{(h-m-2)}_{h}, \\
\un\rho^{(h-m-1)}_{h}, \un\rho^{(h-m)}_{h},\ldots,\un\rho^{(\frac{h}{2}-2)}_{h},\un\rho^{(\frac{h}{2}-1)}_{h}, \un \theta_h^{(m-\frac{h}{2}+2,m)},\ldots,\un\theta_h^{(m,m)}, \un{\widehat{i}}\Big)
$ 
is the same pattern as the case $m\leq\frac{h}{2}-1$ and hence we can get 
\begin{align*}
&(\underbrace{\un \tau_{h},\ldots,\un\tau_h}_{m-\frac{h}{2}+1}, \un\rho^{(1)}_{h},\un\rho^{(2)}_{h},\ldots,\un\rho^{(\frac{h}{2}-1)}_{h} , \un \theta_h^{(m-\frac{h}{2}+2,m)},\ldots,\un\theta_h^{(m,m)}, \un{\widehat{i}})\\
&\sim 
(\underbrace{\un \tau_{h},\ldots,\un\tau_h}_{m-\frac{h}{2}+1},\underbrace{\un \tau_{h},\ldots,\un\tau_h}_{\frac{h}{2}-1}, \un{\widehat{i}})=(\underbrace{\un \tau_{h},\ldots,\un\tau_h}_{m}, \un{\widehat{i}}). 
\end{align*}
Meanwhile $m-\frac{h}{2}+1>\frac{h}{2}-1$,  one can obtain 
\begin{align*}
\un i\sim  (\underbrace{\un \tau_{h},\ldots,\un\tau_h}_{\frac{h}{2}-1}, 0,1,0,2,1,0,\ldots, \frac{h}{2}-2, \frac{h}{2}-3,\ldots,1,0, \un \theta_h^{(m-\frac{h}{2}+2,m)},\ldots,\un\theta_h^{(m,m)}, \un{\widehat{i}})\\
\sim (\underbrace{\un \tau_{h},\ldots,\un\tau_h}_{\frac{h}{2}-1}, 0,1,\ldots, \frac{h}{2}-2,0,1,\ldots, \frac{h}{2}-3,\ldots,0,1,0, \un \theta_h^{(m-\frac{h}{2}+2,m)},\ldots,\un\theta_h^{(m,m)}, \un{\widehat{i}})
\end{align*}
in the case.  Then we can apply the same strategy to the subsequence $(0,1,\ldots, \frac{h}{2}-2,0,1,\ldots, \frac{h}{2}-3,\ldots,0,1,0, \un \theta_h^{(m-\frac{h}{2}+2,m)},\ldots,\un\theta_h^{(m,m)})$ and continue in this way we eventually are reduced to the case $m-\frac{h}{2}+1\leq \frac{h}{2}-1$. This proves the first statement of the proposition. Then it is straightforward to check the second statement by Proposition \ref{prop:restr.2}. 
\end{proof}

\begin{lem} \label{lem:uni-unihat}
Suppose $\un i,\un j\in\mathfrak{P}^2(\Hn)$. Then $\un i\sim \un j$ if and only if $\un{\widehat{i}}\sim \un{\widehat{j}}$. 
\end{lem}
\begin{proof}
If $\un{\widehat{i}}\sim \un{\widehat{j}}$, then clearly $\un i\sim \un j$. Conversely, suppose $\un{\widehat{i}}\nsim \un{\widehat{j}}$. Then by Proposition \ref{prop:uni-simple} we observe that 
both $\un{\widehat{i}}, \un{\widehat{j}}$ belong to the set $\mathfrak{P}^1(\mathcal{H}_a(q))$ with $a=n-\frac{mh}{2}$ and hence there exist $(\xi,T), (\gamma,S)\in \Delta^1_h(a)$ such that $\un{\widehat{i}}=\un i_{(\xi,T)}$ 
and $\un{\widehat{j}}=\un i_{(\gamma,S)}$ by Lemma \ref{Weven1}. Since  $\un{\widehat{i}}\nsim \un{\widehat{j}}$, we obtain $\xi\neq \gamma$ which means the set (repeating allowed) of entries in $\un{\widehat{i}}$ is different from the set (repeating allowed) of entries in $\un{\widehat{j}}$. Hence the set (repeating allowed) of entries in $\un i$ is different from the set (repeating allowed) of entries in $\un j$. This implies $\un i\nsim \un j$. The lemma is verified. 
\end{proof}

Let 
$$
\mathcal{CSP}^2_h(n):=\{\xi=(\underbrace{\frac{h}{2},\ldots,\frac{h}{2}}_{m}, \gamma)\mid m\geq 2, \gamma\in\mathcal{SP}_0(n-\frac{mh}{2}),\gamma_1\leq\frac{h}{2}-1\}. 
$$

\begin{prop}\label{prop:Phi2}
There exists a surjective map $\Phi_2: \mathfrak{P}^2(\Hn)\longrightarrow \mathcal{CSP}^2_h(n)$ such that $\un i\sim \un j$ if and only if $\Phi_2(\un i)=\Phi_2(\un j)$. 
\end{prop}
\begin{proof}
Suppose $\un i\in\mathfrak{P}^2(\Hn)$. By Proposition \ref{prop:uni-simple} and  the proof of Lemma \ref{lem:uni-unihat}, we obtain that $\un{\widehat{i}}$ belong to the set $\mathfrak{P}^1(\mathcal{H}_a(q))$ with $a=n-\frac{mh}{2}$ and hence there exists a unique  $(\gamma, T)\in \Delta^1_h(n-\frac{mh}{2})$ such that $\un{\widehat{i}}=\un i_{(\gamma,T)}$ by Lemma \ref{Weven1}. Then we set $\Phi_2(\un i)=(\underbrace{\frac{h}{2},\ldots,\frac{h}{2}}_{m}, \gamma)$. If $\un i\sim \un j$, then by Lemma \ref{lem:uni-unihat} we have $\un{\widehat{i}}\sim \un{\widehat{j}}$ and then by  Lemma \ref{Weven1} we have $\Phi_2(\un i)=\Phi_2(\un j)$. 
Conversely, given $\xi=(\underbrace{\frac{h}{2},\ldots,\frac{h}{2}}_{m}, \gamma)\in\mathcal{CSP}^2_h(n)$, set 
$$
\un i_\xi=(\underbrace{\un \tau_h,\ldots,\un \tau_h}_m, 0,1,\ldots,\gamma_1-1,0,1,\ldots,\gamma_2-1, \ldots,0,1,\gamma_l-1), 
$$
where $l=\ell(\gamma)$. Then it is straightforward to check that $\un i_\xi\in \mathfrak{P}^2(\Hn)$ and moreover $\Phi_2(\un i_\xi)=\xi$.  This verifies the proposition.
\end{proof}

Let 
$$
\mathcal{CSP}_h(n)=\mathcal{CSP}^1_h(n)\cup\mathcal{CSP}^2_h(n). 
$$
For $\xi \in \mathcal{CSP}_h(n)$, set 
\begin{equation}\label{eq:gamma0-even}
\gamma_0(\xi)=\ell(\xi)+\sharp\{1\leq k\leq\ell(\xi)\mid \xi_k=\frac{h}{2}\}. 
\end{equation}

\begin{thm}\label{thm:HC-classification}
For each $\xi \in \mathcal{CSP}_{h}(n)$ there exists an irreducible $\Hn$-module $D(\xi)$ such that

(1) $D(\xi)$ is type $\texttt{M}$ if $\gamma_0(\xi)$ is even and is type $\texttt{Q}$ if $\gamma_0(\xi)$ is odd. Moreover, $\op{dim}D(\xi)=2^{n-\lfloor\frac{\gamma_0(\xi)}{2}\rfloor}\sharp\op{Std}^{\mathsf{s}}(\xi)$ for $\xi\in \mathcal{CSP}^1_h(n)$ and 
$\op{dim}D(\xi)=2^{n-\lfloor\frac{\gamma_0(\xi)}{2}\rfloor}\sharp \Phi_2^{-1}(\xi)$ for $\xi\in \mathcal{CSP}^2_h(n)$.

(2) $\{D(\xi)\mid \xi \in \mathcal{CSP}_{h}(n)\}$ is a complete set of pairwise non-isomorphic irreducible completely splittable $\Hn$-modules.
\end{thm}
\begin{proof}
For each $\xi \in \mathcal{CSP}^\epsilon_{h}(n)$ with $\epsilon =0,1$, choose a $\un i\in\mathfrak{P}^\epsilon(\Hn)$ such that $\Phi_\epsilon(\un i)=\xi$. Then set
$$
D(\xi)=D^{\un i}. 
$$
Clearly by Proposition \ref{prop:Phi2} and Lemma \ref{Weven1}, we obtain that $X_1=1$ on $D(\xi)$ and hence $D(\xi)$ is an irreducible completely splittable $\Hn$-modules and moreover it is independent of the choice of $\un i$. Then the remaining statements of the theorem follows directly from Proposition \ref{prop:Phi2} and Lemma \ref{Weven1} as well as Theorem \ref{thm:Classficiation}. 
\end{proof}

\begin{rem}
It is interesting to see whether the set $\Phi_2^{-1}(\xi)$ can be identified with the set of certain standard tableaux of shape $\xi$. 
\end{rem}

\begin{example}\label{ex:example-even}
In the case $h=4$, clearly $\xi=(2,1)\in\mathcal{CSP}^1_h(3)$ and moreover by Theorem \ref{thm:HC-classification} we have $D((2,1))$ is of type $\texttt{Q}$ and $\dim D((2,1))=2^2$. 
In the case $h=6$, clearly $\mu=(3,1)\in \mathcal{CSP}^1_h(4)$ and $\xi=(3,2)\in\mathcal{CSP}^1_h(5)$ and moreover by Theorem \ref{thm:HC-classification} we have $D((3,1))$ is of type $\texttt{Q}$ and $\dim D((3,1))=2^4$. 
Meanwhile $D((3,2))$ is of type $\texttt{Q}$ and $\dim D((3,2))=2^5$. 
\end{example}
}

\section{Semisimplicity criteria on finite Hecke-Clifford algebra $\Hn$}\label{sec:semsimplicity}

In this section, we shall provide a semisimplicity criteria for $\Hn$ at roots of unity. 

\subsection{Branching rules and cystal graph}
{
For integers $k\geq 0$, recall that $\mathcal{SP}_k(n)$ denotes the set of $k$-strict partition of $n$ for any $k\geq 0$ and in addition a $k$-strict partition $\la\in\mathcal{SP}_k(n)$ is called $k'$-restricted if
\[
\begin{cases}
\la_r-\la_{r+1}<k' & \text{if }k\mid \la_r,\\
\la_r-\la_{r+1}\leq k' & \text{if }k\nmid \la_r,
\end{cases}
\]
Denote by $\mathcal{RP}_h(n)=\{\la\in\mathcal{SP}_h(\la)\mid \la\text{ is }h\text{-restricted}\}$ which is the subset of $\mathcal{SP}_h(\la)$ consisting of $h$-restricted partitions in the case $h$ is odd and denote  by $\mathcal{DRP}_{h}(n)=\{\la\in\mathcal{SP}_\frac{h}{2}(n)\mid \la \text{ is }h\text{-restricted}\}$ which is the subset of $\mathcal{SP}_{\frac{h}{2}}(\la)$ consisting of  $h$-restricted partition of $n$ in the case $h$ is even. In \cite{Hu}, the partitions in $\mathcal{DRP}_h(n)$ are said to be  the double restricted $\frac{h}{2}$-strict partitions. 


In this section, in the case $h$ is odd (resp. $h$ is even) we identify an arbitrary $h$-strict partition ($\frac{h}{2}$-strict partition) in $\mathcal{SP}_h(n)$ (resp. $\mathcal{SP}_{\frac{h}{2}}(n)$) with its Young diagram. We label the residue of nodes in the Young diagram of $\la$ using the set $\mathbb{I}$ in \eqref{defn:I} via the way that the first node in each row has residue $0$ and then follow the repeating pattern
$$
\left\{
\begin{array}{ll}
0,1,\cdots,\frac{h-3}{2},\frac{h-1}{2},\frac{h-3}{2},\cdots,1,0, \quad & \text{if }h\text{ is odd},\\
\\
0,1,\cdots,\frac{h}{2}-2,\frac{h}{2}-1,\frac{h}{2}-1,\frac{h}{2}-2,\cdots,1,0,\quad  &\text{if }h\text{ is even}.
\end{array}
\right.
$$
The residue of the node $A$ is denoted $\op{res}A$. Let $i\in \mathbb{I}$ be some fixed residue. A node $A=(r,s)\in \la$ with $\op{res}A=i$ is called $i$-removable (for $\la$) if either $\la_A:=\la-\{A\}$ is again a $h$-strict partition  (resp. $\frac{h}{2}$-strict partition) or the node $B=(r,s+1)$ immediately to the right of $A$ belongs to $\la$ and moreover satisfies $\op{res}B=i$  and both $\la_B:=\la-\{B\}$ and $\la_{A,B}=\la-\{A,B\}$ are $h$-strict partitions (resp. $\frac{h}{2}$-strict partition) in the case $h$ is odd (resp. $h$ is even). Similarly, A node $B=(r,s)\notin \la$ with $\op{res}B=i$ is called $i$-addable (for $\la$) if either $\la_B:=\la\cup\{B\}$ is again a $h$-strict partition (resp. $\frac{h}{2}$-strict partition) or the node $A=(r,s-1)$ immediately to the left of $B$ does not belong to $\la$ and satisfies  $\op{res}A=i$ and both $\la^{A}=\la\cup\{A\}$  and $\la^{A,B}=\la\cup\{A,B\}$ is a $h$-strict partition (resp. $\frac{h}{2}$-strict partition) in the case $h$ is odd (resp. $h$ is even), then $B$ is also called $i$-addable.
Furthermore, label all $i$-addable nodes of the diagram  by $+$ and all $i$-removable nodes by $-$. Then, the $i$-signature of $\la$ is the sequence of $+$ and $-$ obtained by going along the rim of the Young diagram from bottom left to top right and reading off all the signs. The reduced $i$-signature of $\la$ is obtained from the $i$-signature by successively erasing all neighboring pairs of the form $+-$. Nodes corresponding to a $-$ in the reduced $i$-signature are called $i$-normal and the rightmost of these is the $i$-good node, nodes corresponding to a $+$ in the reduced $i$-signature are called $i$-conormal and the leftmost of these is the $i$-cogood node. Define
\begin{equation}\label{eq:two-crystal-operator}
\begin{aligned}
\epsilon_i(\la)&=\#\{i\text{-normal nodes in }\la\}\\
&=\#\{-\text{'s in the reduced }i\text{-signature of }\la\},\\
\varphi_i(\la)&=\#\{i\text{-conormal nodes in }\la\}\\
&=\#\{+\text{'s in the reduced }i\text{-signature of }\la\}.
\end{aligned}
\end{equation}

\begin{example}\label{ex:add-remove}
Suppose $h=14$,  then $\I=\{0,1,2,3,4,5,6\}$ and the residue pattern will be repeating of $0,1,2,3,4,5,6,6,5,4,3,2,1,0$ in each row of partitions belonging to $\mathcal{DRP}_h(n)$. Clearly  $\la=(7,6,5)\in\mathcal{DRP}_h(18)$ and its residues can be  displayed as follows:
\[\ytableausetup{boxsize=1.5em}
\ytableaushort{0123456,012345,01234} 
\]
There is no $6$-removable nodes  in $\la$ and the $6$-addable nodes are: 
\[\ytableausetup{boxsize=1.5em}
\ytableaushort{{}{}{}{}{}{}{}{\none[+]},{}{}{}{}{}{}{\none[+]},{}{}{}{}{}}
\]
So by \eqref{eq:two-crystal-operator} we have $\epsilon_6(\la)=0,\varphi_6(\la)=2$. 
\end{example}
\begin{thm}\cite{Ka, Hu}
For each $\la\in\mathcal{RP}_h(n)$ (resp. $\la\in\mathcal{DRP}_h(n)$) in the case $h$ is odd (resp. $h$ is even) and suppose $B$ is a $i$-cogood node of $\la$ for some $i\in\mathbb{I}$, then $\la^B:=\la\cup B\in\mathcal{RP}_h(n+1)$ (resp. $\la^B:=\la\cup B\in\mathcal{DRP}_h(n+1)$) in the case $h$ is odd (resp. $h$ is even).

\end{thm}
Then we have the following due to \cite{BK1} and \cite{Ts}. 

\begin{thm}\label{thm:typefi}\cite[Theorem 9.4 and 9.6]{BK1}\cite[Corollary 3.16 and Lemma 4.9]{Ts}\label{thm;BK-T}
Suppose $h\geq 3$ and $n\geq 1$. Then 
\begin{enumerate}
\item In case $h$ is odd (resp. $h$ is even), for each $\la\in\mathcal{RP}_h(n)$ (resp. $\la\in\mathcal{DRP}_h(n)$), there exists an irreducible $\Hn$-module $M(\la)$ and moreover $\{M(\la)\mid \la\in\mathcal{RP}_h(n) \}$ (resp. $\{M(\la)\mid \la\in\mathcal{DRP}_h(n) \}$) is a complete set of non-isomorphic irreducible $\Hn$-modules. 
\item Denote by $b_h(\la)=\sharp\{r\geq 1|h\nmid \la_r\}$ (resp. $b_{\frac{h}{2}}(\la)=\sharp\{r\geq 1|\frac{h}{2}\nmid \la_r\}$). Then  $M(\la)$ is type $\texttt{M}$ if $b_h(\la)$ is even (resp. $b_{\frac{h}{2}}(\la)$ is even) and is type $\texttt{Q}$ if $b_h(\la)$ is odd (resp. $b_{\frac{h}{2}}(\la)$ is odd) in case $h$ is odd (resp. $h$ is even). 
\item For each $i\in\mathbb{I}$, there exists a $\Hnn$-supermodule $f_iM(\la)$, unique up to isomorphism, such that $f_iM(\la)\neq 0$ if and only if $\la$ has an $i$-cogood node $B$ in which case $f_iM(\la)$ is indecomposable and the multiplicity of $M(\la^B)$ as a composition factor in $f_iM(\la)$ is $\varphi_i(\la)$.
\end{enumerate}
\end{thm}

\begin{rem}\label{rem:Dxi-Mxi}(cf. \cite[Remark 6.9]{Wa})
It is straightforward to check that $\mathcal{CSP}_h(n)\subset\mathcal{RP}_h(n)$ if $h$ is odd and $\mathcal{CSP}_h(n)\subset\mathcal{DRP}_h(n)$ if $h$ is even. Moreover by investigating the weights of $M(\xi)$ via the study in \cite{BK1,Ts}, we actually have $D(\xi)\cong M(\xi)$ for each $\xi\in\mathcal{CSP}_h(n)$. 
\end{rem}

\subsection{Semisimplicity criterion}

\begin{lem}\label{prop:IrrCSodd}
Suppose $h\geq 3$ is odd. Then $\mathcal{CSP}_h(n)=\mathcal{RP}_h(n)$ if and only if $h\geq n$. Hence, every irreducible $\Hn$-modules is completely splittable if and only if $h\geq n$.
\end{lem}
\begin{proof}
Assume $h\geq 3$ is odd. It is straightforward (cf. \cite{CWZ}) to check that $\mathcal{RP}_h(n)=\mathcal{CSP}_h(n)=\mathcal{SP}_0(n)$ for $h>n$ and $\mathcal{RP}_h(n)=\mathcal{CSP}_h(n)=\mathcal{SP}_0(n)\setminus\{(n)\}$ in the case $h=n$. 
Now suppose $h< n,$ then $n=ah+b$ for some $a\geq 1$ and $0\leq b \leq h-1$. Obviously either of $\la=(h^{a},b)$(the case $b\neq 0$) and $\mu=(h^{a-1},h-1,1)$(the case $b=0$ and then $a> 1$) belongs to $\mathcal{RP}_h(n)$. But neither of them belongs to $\mathcal{CSP}_h(n)$ since $\la_1=\mu_1=h>h-1$. Hence $\mathcal{RP}_h(n)\neq \mathcal{CSP}_h(n)$. Thus the lemma is proved.
\end{proof}
\begin{lem}\label{prop:IrrCSeven}
Suppose $h\geq 4$ is even. Then $\mathcal{CSP}_h(n)=\mathcal{DRP}_h(n)$ if and only if $h\geq 2n$. Hence, every irreducible $\Hn$-modules is completely splittable if and only if $h\geq 2n$.
\end{lem}
\begin{proof}
Clearly $\mathcal{CSP}_h(n)$ is naturally a subset of $\mathcal{DRP}_h(n)$.  If $h < 2n$, then we have $n=a\left(\frac{h}{2}\right)+b$ for some $a\geq 1,0\leq b<\frac{h}{2}$. Notice that either $\la=(\frac{h}{2}+1,\frac{h}{2},\ldots, \frac{h}{2},\frac{h}{2}-1)$(the case $b=0$ and then $a\geq 2$) or $\mu=(\frac{h}{2}+1,\frac{h}{2},\ldots, \frac{h}{2},b-1)$(the case $b\neq 0$) belongs to $\mathcal{DRP}_h(n)$. But neither of them belongs to $\mathcal{CSP}_h(n)$ since $\la_1=\mu_1=\frac{h}{2}+1>\frac{h}{2}$.

Now suppose $h\geq 2n$. Obviously $ \mathcal{CSP}_h(n)=\mathcal{SP}_{0}(n)$. For any $\lambda \in \mathcal{DRP}_h(n)$, if $\lambda_r = \lambda_{r+1}$ for some $r$, then $\frac{h}{2} \mid \lambda_r$. Since $2\lambda_r \leq n \leq \frac{h}{2}$, it follows that $\lambda_r = 0$, which implies $\lambda \in \mathcal{SP}_0(n)$. Thus, $\mathcal{DSP}_h(n) \subseteq \mathcal{SP}_0(n)$. Therefore, $\mathcal{CSP}_h(n)=\mathcal{SP}_0(n)=\mathcal{DRP}_h(n)$ when $h\geq 2n$.  Now the last statement of the lemma  follows from Theorem \ref{thm:HC-classification}.
\end{proof}

\begin{lem}\cite[Corollary 3.2]{Sa}\label{lem:STD}
For any $n\in\mathbb{Z}_+$,
\[\sum_{\la\in\mathcal{SP}_0(n)}2^{n-\ell(\la)}(\sharp\op{Std}^{\mathsf{s}}(\la))^2=n!\]
\end{lem}


\begin{prop}\label{prop:CS-Nonsemi} 
The following holds: 
\begin{enumerate}
\item If $h$ is odd, then $\Hn$ is not semisimple in the case $h=n$. If $h$ is even, then $\Hn$ is not semisimple in the case $h=2n$.
\item If $h>n$ in the case $h$ is odd and $h>2n$ in the case $h$ is even, then the Hecke-Clifford superalgebra $\Hn$ is semisimple.
\end{enumerate}
\end{prop}
\begin{proof}
By Lemma \ref{prop:IrrCSodd} and \ref{prop:IrrCSeven}, every irreducible $\Hn$-module is completely splittable in the case $h=n$ being odd and $h=2n$ being even. First, suppose  $h=n$ is odd, by 
the proof of Lemma \ref{prop:IrrCSodd} we have $\mathcal{CSP}_h(n)=\{\xi\in\mathcal{SP}_0(n)\mid \xi\neq (n)\}$. Then by Theorem \ref{dim2}, Lemma \ref{prop:IrrCSodd} and Lemma \ref{lem:STD} we have
\begin{equation}\label{eq:IrrCSodd}
\begin{aligned}
&\sum_{\substack{\xi\in\mathcal{CSP}_h(n)\\ \ell(\xi)\text{ is even}}}(\op{dim}D(\xi))^2+\sum_{\substack{\xi\in\mathcal{CSP}_h(n)\\ \ell(\xi)\text{ is odd}}}\frac{(\op{dim}D(\xi))^2}{2}\\
&=2^n\sum_{\substack{\xi\in\mathcal{SP}_0(n)\\ \xi\neq (n)}}2^{n-\ell(\xi)}(\sharp\op{Std}^{\mathsf{s}}_h(\xi))^2
<2^n\sum_{\xi\in\mathcal{SP}_0(n)}2^{n-\ell(\xi)}(\sharp\op{Std}^{\mathsf{s}}(\xi))^2=2^nn!. 
\end{aligned}
\end{equation}
This together with Lemma \ref{lem:type MQ} implies that $\Hn$ is not semisimple. Second, assume $h=2n$ is even, again by the proof of Lemma \ref{prop:IrrCSeven} we have $\mathcal{CSP}_h(n)=\mathcal{SP}_0(n)$.  Then by Theorem \ref{thm:HC-classification}, Proposition \ref{prop:IrrCSeven} and Lemma \ref{lem:STD}, we have
\begin{equation}\label{eq:IrrCSeven}
\begin{aligned}
&\sum_{\substack{\xi\in\mathcal{CSP}_h(n)\\ \gamma_0(\xi)\text{ is even}}}(\op{dim}D(\xi))^2+\sum_{\substack{\xi\in\mathcal{CSP}_h(n)\\ \gamma_0(\xi)\text{ is odd}}}\frac{(\op{dim}D(\xi))^2}{2}
=2^n\sum_{\xi\in\mathcal{CSP}_h(n)}2^{n-\gamma_0(\xi)}(\sharp\op{Std}^{\mathsf{s}}(\xi))^2\\
&<2^n\sum_{\xi\in\mathcal{SP}_0(n)}2^{n-\ell(\xi)}(\sharp\op{Std}^{\mathsf{s}}(\xi))^2=2^nn!
\end{aligned}
\end{equation}
since $\gamma_0((n))>\ell((n))$ and $\gamma_0(\xi)=\ell(\xi)$ for $\xi\neq (n)$ by \eqref{eq:gamma0-even}. Again by Lemma \ref{lem:type MQ}, $\Hn$ is not semisimple. This proves (1).  
The statement in (2) is due to \cite{SW} which can also be proved via a similar argument to that of (1) by a dimension comparison and Wedderburn theorem Lemma \ref{lem:type MQ} and we omit the details here. 
\end{proof}

\begin{lem}\label{lem:nonsemi35}
The algebra $\mathcal{H}_3(q)$ is not semisimple in the case $h=4$. And the algebra $\mathcal{H}_5(q)$ is not semisimple in both case $h=3$ and $h=6$. 
\end{lem}
\begin{proof}
We prove each case by a dimension comparison. In the case $h=4$, by Theorem \ref{thm;BK-T}  we have that $\{M((3)),M((2,1))\}$ is the complete set of pairwise non-isomorphic irreducible $\mathcal{H}_3(q)$-supermodule. Besides, both $M((2,1)),M((3))$ are of type $\texttt{Q}$. Meanwhile Example \ref{ex:example-even} and Remark \ref{rem:Dxi-Mxi} show that $\op{dim}M((2,1))=2^2$. If $\mathcal{H}_3(q)$ is semisimple, then by Lemma \ref{lem:type MQ} we should have 
\[ (\op{dim}M((3)))^2+2^4=3!\cdot 2^4=96.\]
But $(\op{dim}M((3)))^2=80$ has no integral solution, a contradiction. Hence $\mathcal{H}_3(q)$ is not semisimple in case $h=4$.

Now consider $\mathcal{H}_5(q)$. In case $h=3$, $\{M((4,1)),M((3,2))\}$ is the complete set of pairwise non-isomorphic irreducible $\mathcal{H}_5(q)$-supermodule by Theorem \ref{thm:typefi}. Besides, $M((4,1))$ is of type $\texttt{M}$ and $M((3,2))$ is of type $\texttt{Q}$. If $\mathcal{H}_5(q)$ is semisimple, then by Lemma \ref{lem:type MQ} we should have
\[ (\op{dim}M((4,1)))^2+\frac{(\op{dim}M((3,2)))^2}{2}=120\cdot 2^5\]
But $a^2+2b^2=240\cdot 32$ has no integral solution by  number theory (cf. \cite{Co}), a contradiction. Hence $\mathcal{H}_5(q)$ is not semisimple in case $h=3$.

In case $h=6$, $\{M((5)),M((4,1)),M((3,2))\}$ is the complete set of pairwise non-isomorphic irreducible $\mathcal{H}_5(q)$-supermodule by Theorem \ref{thm:typefi}. Besides, $M((5)),M((3,2))$ are both of type $\texttt{Q}$ and $M((4,1))$ is of type $\texttt{M}$. Example \ref{ex:example-even} and Remark \ref{rem:Dxi-Mxi} show that $\op{dim}M((3,2))=2^5$. If $\mathcal{H}_5(q)$ is semisimple, then by Lemma \ref{lem:type MQ} we should have
\[ 16\cdot 2^5+\op{dim}M((4,1))^2+\frac{1}{2}\op{dim}M((4,1))=120\cdot 2^5.\]
But $a^2+2b^2=208\cdot 32$ has no integral solution by  number theory (cf. \cite{Co}), a contradiction. Hence $\mathcal{H}_5(q)$ is not semisimple in case $h=6$.
\end{proof}

Based on Proposition \ref{prop:CS-Nonsemi} and Lemma \ref{lem:nonsemi35}, we can show that:
\begin{thm}\label{thm:plusnonsemi}
Suppose $h\geq 3$. If $\Hr$ is not semisimple, then $\Hrr$ is not semisimple.
\end{thm}
\begin{proof}
If $\Hr$ is not semisimple, then by Proposition \ref{prop:CS-Nonsemi}(2)  we have $h\leq r$ in the case $h$ is odd and $h\leq 2r$ in the case $h$ is even. Write $r=ah+b$ with $0\leq b\leq h-1$ in the case $h$ is odd and $r=a(\frac{h}{2})+b$ with $0\leq b\leq \frac{h}{2}-1$ in the case $h$ is even. Then in the case $h$ is odd we take
\begin{equation}\label{eq:sepcail-la-odd}
\la=\begin{cases}
(h,h,\ldots,h,h-1,b+1),&\text{ if }b\neq h-1,h-2,\\
(h,h,\ldots,h,h-1),&\text{ if }b= h-1,\\
(h,h,\ldots,h,h-1,h-2,1),&\text{ if }b= h-2,h>3,\\
(5,3,\ldots,3,2),&\text{ if }b=1,h=3,a\geq 2,\\
\end{cases}
\end{equation}
and in the case $h$ is even we take
\begin{equation}\label{eq:sepcail-la-even}
\la=\begin{cases}
(\frac{h}{2},\frac{h}{2},\ldots,\frac{h}{2},\frac{h}{2}-1,b+1),&\text{ if }b\neq \frac{h}{2}-1,\frac{h}{2}-2\\
(\frac{h}{2},\frac{h}{2},\ldots,\frac{h}{2},\frac{h}{2}-1),&\text{ if }b=\frac{h}{2}-1\\
(h-1,\frac{h}{2},\ldots,\frac{h}{2},\frac{h}{2}-1),&\text{ if }b= \frac{h}{2}-2,a\geq 2\\
(\frac{h}{2}-1,\frac{h}{2}-2,1),&\text{ if }b= \frac{h}{2}-2,a=1,h>6\\
\end{cases}
\end{equation}
It's straightforward to check that $\varphi_i(\la)\geq 2$ for either $i=0$ or $i=\frac{h}{2}-1$. 
For example, if $h$ is even and $b\neq\frac{h}{2}-1,\frac{h}{2}-2$, then $\la=(\frac{h}{2},\frac{h}{2},\ldots,\frac{h}{2},\frac{h}{2}-1,b+1)\in\mathcal{DRP}_h(n)$ satisfies $\varphi_{\frac{h}{2}-1}(\la)=2>1$, see Example \ref{ex:add-remove} for a concrete instance.  Thus, for these $\la$ listed in \eqref{eq:sepcail-la-odd} in the case $h$ is odd and  in \eqref{eq:sepcail-la-even} in the case $h$ is even, we obtain that $f_0M(\la)$ or $f_{\frac{h}{2}}M(\la)$ is indecomposable but not simple by Theorem \ref{thm:typefi}(3). Thus $\Hrr$ is not semisimple. Now we have three remaining cases to prove, that is, to prove $\Hrr$ is not semisimple under the assumption $\Hr$ is not semisimple and $(r,h)$ satisfies ($r=4,h=3$), ($r=4,h=6$) or ($r=2,h=4$). By Lemma \ref{lem:nonsemi35} we already know $\mathcal{H}_{r+1}(q)$ is not semisimple in each case. Hence the theorem is proved.

\end{proof}

Now we can state the main result of this section.

\begin{thm}\label{thm:semisimple}
The Hecke-Clifford superalgebra $\Hn$ is semisimple if and only if $h>n$ in the case $h$ is odd and $h>2n$ in the case $h$ is even.
\end{thm}
\begin{proof}
In the case $h$ is odd, if $h>n$, then $\Hn$ is semisimple by Proposition \ref{prop:CS-Nonsemi}.  Now suppose $h\leq n$. Then $\mathcal{H}_h(q)$ is not semisimple by Proposition \ref{prop:CS-Nonsemi} and hence $\mathcal{H}_{h+1}(q),\mathcal{H}_{h+2}(q),\ldots$ are not semisimple by Theorem \ref{thm:plusnonsemi}. Therefore $\Hn$ is not semisimple. The same discussion holds in the case $h$ being even.
\end{proof}

\begin{rem}
Theorem \ref{thm:semisimple} verifies the equivalence between (1) and (3) in  \cite[Conjecture 4.13]{SW} for $\Hn$.  Meanwhile by Proposition \ref{prop:CS-Nonsemi} one can see that (2) is neither equivalent to (1) nor (3) in \cite[Conjecture 4.13]{SW} for $\Hn$. 
In \cite{Sh}, a different method by extending the trace form introduced in\cite{WW} to cyclotomic Hecke superalgebras and computing the associated Schur elements is applied to derive the necessary condition for cyclotomic Hecke superalgebras including $\Hn$ as a special case to be semisimple under certain assumptions. Our statement shows that the assumption in the case of $\Hn$ in \cite[Theorem 1.3(2)]{Sh} can be removed. 
\end{rem}


\begin{thebibliography}{ABC}





\bibitem[BK1]{BK1} J. Brundan and A. Kleshchev, {\em Hecke-Clifford superalgebras, crystals of type
$A^{(2)}_{2l}$, and modular branching rules for $\widehat{S}_n$},
Repr. Theory {\bf 5} (2001), 317--403.

\bibitem[BK2]{BK2} J. Brundan and A. Kleshchev, 
{\em Representation theory of symmetric groups and
their double covers}, Groups, Combinatorics and Geometry (Durham, 2001), pp.
31--53, World Scientific, Publishing, River Edge, NJ, 2003. 

\bibitem[CWZ]{CWZ} M. Chen, J. Wan and H. Zhao,
{\em A note on irreducible representations of symmetric groups and Sergeev superalgebras}, 
arXiv:2603.28009, 18 pages. 

\bibitem[C1]{C1} I. Cherednik, {\em Special bases of irreducible
representations of a degenerate affine Hecke algebra}, Funct. Anal.
Appl. {\bf 20} (1986), no.1, 76--78.

\bibitem[C2]{C2} I. Cherednik, {\em A new interpretation of Gel'fand-Tzetlin
bases}, Duke. Math. J. {\bf 54} (1987), 563--577.


\bibitem[Co]{Co} D. Cox, 
Primes of the form  $x^2+ny^2$: Fermat, class field theory, and complex multiplication, 3rd ed.,
AMS Chelsea Publishing, Providence, RI, 2022.




\bibitem[G]{G}
I. Grojnowski,  {\em Aﬃne ${\rm sl}_p$ controls the modular representation theory of the symmetric group
and related Hecke algebras,} preprint, arXiv:math/9907129, 1999.

\bibitem[HKS]{HKS}D. Hill, J. Kujawa and J. Sussan,
{\em Degenerate affine Hecke-Clifford algebras and type $Q$ Lie
superalgebras}, Math. Z {\bf  268} (2011), 1091--1158.

\bibitem[Hu]{Hu} J. Hu, 
{\em Mullineux involution and twisted affine Lie algebras},
J. Algebra {\bf 304} (2006) 557–576.



\bibitem[JN]{JN} A. Jones, M. Nazarov, {\em Affine Sergeev algebra and $q$-analogues of
the Young symmetrizers for projective representations of the
symmetric group}, Proc. London Math. Soc. {\bf 78} (1999), 481--512.

\bibitem[Ka]{Ka} S.-J. Kang, 
{\em Crystal bases for quantum affine algebras and combinatorics of Young walls}, Proc. London Math. Soc. 86
(2003), 29--69.

\bibitem[KKT]{KKT} S.~J. Kang, M.~Kashiwara and S.~Tsuchioka, 
{\em Quiver Hecke Superalgebras}, J. Reine Angew. Math., {\bf 711} (2016), 1--54.

\bibitem[KMS]{KMS} I.~ Kashuba, A.~ Molev, V.~ Serganova
{\em On the Jucys–Murphy method and fusion procedure for the Sergeev superalgebra}, J. London Math. Soc., 112: e70302. https://doi.org/10.1112/jlms.70302. 

\bibitem[K1]{K1} A. Kleshchev, {\em Completely splittable representations
of symmetric groups}, J. Algebra {\bf 181} (1996), 584--592.

\bibitem[K2]{K2} A. Kleshchev, { Linear and Projective
Representations of Symmetric Groups}, Cambridge University Press,
2005.

\bibitem[K3]{K3} A. Kleshchev, {\em  Representation Theory of symmetric groups and related Hecke algebras},  Bulletin (New Series) of the American Mathematical Society {\bf 47} (2010), 419--481. 
Volume 47, Number 3, July 2010, Pages 419–481

\bibitem[KL]{KL} A. Kleshchev and M. Livesey, 
{\em RoCK blocks for double covers of symmetric groups and quiver Hecke superalgebras},
Memoirs of the American Mathematical Society Volume 309,  American Mathematical Society, 2025. 

\bibitem[KR]{KR}  A. Kleshchev and A. Ram, 
{\em Homogeneous representations of Khovanov-Lauda algebras}, J. Eur. Math. Soc.(JEMS) {\bf 12} (2010), no. 5, 1293--1306


\bibitem[LS]{LS}
{\sc S.~Li, L.~Shi}, {\em Seminormal bases of cyclotomic Hecke-Clifford algebras},
Lett. Math. Phys., {\bf 115}, (2025), https://doi.org/10.1007/s11005-025-01998-x.	








 \bibitem[Mac]{Mac}  I.G.~ Macdonald,
 { Symmetric functions and Hall polynomials},
  Second edition, Clarendon Press, Oxford, 1995.


\bibitem[M]{M} O. Mathieu, {\em On the dimension of some modular irreducible representations of
the symmetric group}, Lett. Math. Phys. {\bf 38} (1996), 23--32.


\bibitem[Mo]{Mo} M. Mori, 
{\em A cellular approach to the Hecke–Clifford superalgebra},
Preprint, 2014, arXiv:1401.1722, 80 pages.


\bibitem[N1]{N1} M. Nazarov, {\em Young's orthogonal form of
irreducible projective representations of the symmetric group},
 J. London Math. Soc. (2){\bf 42} (1990), no. 3, 437--451.


\bibitem[N2]{N2} M. Nazarov, {\em Young's symmetrizers for projective representations
of the symmetric group}, Adv. Math. {\bf 127} (1997), 
190--257.

\bibitem[OV]{OV} A. Okounkov and A. Vershik, {\em A new approach to
representation theory of symmetric groups}, Selecta Math. (N.S)
{\bf 2} (1996), 581--605.

\bibitem[Ol]{Ol} G.I.~Olshanski,
{\em Quantized universal enveloping superalgebra of type $Q$ and a
super-extension of the Hecke algebra}, Lett. Math. Phys. {\bf  24}
(1992),  93--102.

\bibitem[Ra]{Ra} A. Ram, {\em Skew shape representations are
irreducible}, (English summary) Combinatorial and geometric
representation theory (Seoul, 2001), 161--189, Contemp. Math., 325,
Amer. Math. Soc., Providence, RI, 2003.




\bibitem[Ru]{Ru} O. Ruff, {\em Completely splittable representations
of symmetric groups and affine Hecke aglebras}, J. Algebra {\bf
305} (2006), 1197--1211.

\bibitem[Sa]{Sa} B. Sagan, {\em Shifted tableaux, {S}chur {$Q$}-functions, and a conjecture of
	{R}. {S}tanley}, J. Combin. Theory Ser.  {\bf
	45} (1987), 62--103.
	
 \bibitem[Sch]{Sch}  I.~Schur,
{\em \"Uber die Darstellung der symmetrischen und der alternierenden
Gruppe durch gebrochene lineare Substitutionen}, J. Reine Angew.
Math. {\bf 139} (1911), 155--250.

\bibitem[Sh]{Sh} L. Shi,
{\em On the semisimplicity and Schur elements of (super)symmetric superalgebras}, arXiv:2605.04745. 

\bibitem[SW]{SW} L. Shi and J. Wan 
{\em On representation theory of cyclotomic Hecke-Clifford algebras}, 
J. Alg (to appear), 35 pages, 2026. https://doi.org/10.1016/j.jalgebra.2026.03.031.


\bibitem[Ts]{Ts} S.~Tsuchioka, 
{\em Hecke-Clifford superalgebras and crystals of type $D_l^{(2)}$}, Publ. Res. Inst. Math. Sci. {\bf 46} (2010), 423--471.



\bibitem[W]{Wa} J. Wan, {\em Completely splittable representations of affine
              {H}ecke-{C}lifford algebras}, J. Algebraic Combin. {\bf
32} (2010), 15--58.


\bibitem[WW]{WW}
J. Wan and W. Wang, {\em Frobenius character formula and spin generic degrees
	for Hecke–Clifford algebra}, Proc. London Math. Soc., {\bf 106} (2013), 287--317.
	
\bibitem[Wa]{W} W. Wang, {\em Double affine Hecke-Clifford algebras for
		the spin symmetric group}, preprint, math.RT/0608074, 2006.
		
	
  
\end{thebibliography}
\end{document}